\documentclass[12pt]{article}

\usepackage{amsmath, amssymb, amsthm, amsfonts}
\usepackage{amsthm}
\usepackage{supertabular}
\usepackage{caption}
\newtheorem{theorem}{Theorem}[section]
\newtheorem{proposition}[theorem]{Proposition}
\newtheorem{definition}[theorem]{Definition}
\newtheorem{lemma}[theorem]{Lemma} 
\newtheorem{example}[theorem]{Example}
\newtheorem{remark}[theorem]{Remark}
\newtheorem{corollary}[theorem]{Corollary}

\usepackage[section]{placeins} 
\usepackage{subcaption}
\usepackage{multicol}

\usepackage{graphicx}

\usepackage{enumitem}

\usepackage{fancyhdr}

\usepackage[section]{placeins}

\usepackage[hidelinks]{hyperref}

\usepackage{etoolbox}

\usepackage{indentfirst}

\usepackage[explicit]{titlesec}

\usepackage[T1]{fontenc}

\usepackage{XCharter}

\usepackage{tocloft}

\usepackage[doublespacing]{setspace}

\usepackage{tikz-cd}
\usetikzlibrary{arrows.meta}
\usepackage{pgfplots}
\pgfplotsset{compat=1.5}

\usepackage{longtable}
\usepackage{booktabs}
\usepackage{array}

\usepackage{comment}

\makeatletter
\def\vv(#1,#2,#3){\ensuremath{v_{#1,#2}^{{}^{\langle #3 \rangle}}}}
\makeatother

\newcommand{\Z}{\mathbb{Z}}

\titleformat{\section}[block]{\centering\singlespace\large\bf}{CHAPTER \thesection \hspace{0.4em} \MakeUppercase{#1}}{0em}{}{}

\titleformat{name=\section,numberless}[block]{\centering\singlespace\large\bf}{\hspace{0.4em} \MakeUppercase{#1}}{0em}{}{}

\titlespacing{\section}{0pt}{0pt}{0pt}
\titlespacing{\subsection}{0pt}{0pt}{0pt}
\titlespacing{\subsubsection}{0pt}{0pt}{0pt}

\cftsetindents{section}{0em}{6em}

\makeatletter
\def\thm@space@setup{\thm@preskip=0pt
\thm@postskip=0pt}
\makeatother
\newtheoremstyle{newstyle}      
{} 
{} 
{\mdseries} 
{} 
{\bfseries}  
{.} 
{ } 
{} 

\makeatletter
\renewenvironment{proof}[1][\proofname]{\par
  \pushQED{\qed}%
  \normalfont \topsep0\p@\relax
  \trivlist
  \item[\hskip\labelsep\itshape
  #1\@addpunct{.}]\ignorespaces 
}{
  \popQED\endtrivlist\@endpefalse
}
\makeatother

\patchcmd{\thebibliography}{\section*{\refname}}{}{}{}

\let\OLDthebibliography\thebibliography
\renewcommand\thebibliography[1]{
    \OLDthebibliography{#1}
    \setlength{\parskip}{0pt}
    \setlength{\itemsep}{0pt plus 0.3ex}
}

\usepackage[left=1in,right=1in, top=1in, bottom=1in]{geometry}
\renewcommand{\arraystretch}{0.85}

\newcommand{\Mdef}[2]{\newcommand{#1}{\relax \ifmmode #2 \else $#2$\fi}}

\newcommand{\Vspc}{\vspace*{0.1in}}

\makeatletter \@addtoreset{equation}{section} \makeatother

\renewcommand{\thesection}{\arabic{section}}

\begin{document}

\centerline{\bf ON THE COBORDISM GROUPS OF $O\langle n \rangle$-MANIFOLDS}

\vskip-0.4cm
\thispagestyle{empty}

\begin{center}
    \vspace{-0.4cm}
    by \\
    {\bf HASSAN H. ABDALLAH}\\ 
    {\bf DISSERTATION}\\  
    Submitted to the Graduate School,\\
    of Wayne State University,\\
    Detroit, Michigan\\
    in partial fulfillment of the requirements\\
    for the degree of\\
    {\bf DOCTOR OF PHILOSOPHY} 
\end{center}

\begin{flushleft}
    \vspace*{-0.20in}
    \hspace*{3.09in}2026 
    \hspace*{3.09in}MAJOR: Mathematics\\ 
    \hspace*{3.09in}Approved By:\\
    \hspace*{3.09in}-----------------------------------------------------------\\
    \vspace*{-0.25in}
    \hspace*{3.09in}Advisor\hspace*{1.5in} Date\hspace*{0.1in}\\
    \bigskip
    \hspace*{3.09in}-----------------------------------------------------------\\
    \medskip
    \hspace*{3.09in}-----------------------------------------------------------\\
    \medskip
    \hspace*{3.09in}-----------------------------------------------------------\\
\end{flushleft}

\newpage

\pagestyle{fancy} \chead{} \rhead{} \lhead{}
\pagenumbering{roman} \lfoot{}\cfoot{\thepage}\rfoot{}
\setcounter{page}{2}

\phantomsection
\section*{DEDICATION}

\addcontentsline{toc}{section}{Dedication}
\begin{center}
	\textit{To my son, Idriss.} 
\end{center}

\newpage

\phantomsection
\section*{ACKNOWLEDGEMENTS}

\addcontentsline{toc}{section}{Acknowledgements}
I am deeply indebted to my advisor, Andrew Salch. His support, guidance, and unwavering willingness to discuss mathematics has shaped me into the mathematician I am today. I benefited greatly from the algebraic topology faculty at Wayne State University, including Bob Bruner, Dan Isaksen, and John Klein. I lost count of the number of times 
I stopped by Bob Bruner's office, usually uninvited, with a question in hand. His enthusiasm and willingness to share tricks of the computational stable homotopy theory trade were indispensable to me. Receiving Dan Isaksen's guidance over the years has kept me on a productive path. I could not have asked for a better place to do algebraic topology research. I am grateful to Andy Baker for serving on my thesis committee. 

I would like to thank my mother, my brother Ali, and my sisters Laura and Sara for always believing in me. I would also like to thank my son, Idriss, for being a loving and amusing distraction from mathematics. Seeing him grow into a little toddler these last few years has been my greatest joy. Finally, I would like to thank my wife, Leila, without whom this thesis would not have been possible.


\newpage

\begin{singlespace}

\renewcommand{\contentsname}{\hfill\large TABLE OF CONTENTS \hfill}
\tableofcontents
\newpage

\phantomsection
\addcontentsline{toc}{section}{List of Tables}
\renewcommand{\listtablename}{\hfill\large LIST OF TABLES \hfill}
\listoftables
\newpage

\phantomsection
\addcontentsline{toc}{section}{List of Figures}
\renewcommand{\listfigurename}{\hfill\large LIST OF FIGURES \hfill}
\listoffigures

\end{singlespace}

\clearpage

\pagestyle{fancy} \chead{\thepage} \rhead{} \lhead{}
\pagenumbering{arabic} \lfoot{}\cfoot{}\rfoot{}
\setcounter{equation}{0}

\section{Introduction}
\label{chap:Introduction}

Given a smooth manifold, a tangential $O\langle n \rangle$-structure is a lift of the classifying map of its tangent bundle to $BO\langle n \rangle$. An $O\langle n \rangle$-manifold is a manifold equipped with an equivalence class of $O\langle n \rangle$-structures. The cases $n=1,2\text{ and }4$ recover the familiar notions of unoriented, oriented, and spin manifolds. Larger values of $n$ furnish higher analogs of these notions with additional geometric properties.
The goal of this thesis is to classify $O\langle n\rangle$-manifolds up to \textit{cobordism}, an equivalence relation first described by Henri Poincar\'e in 1895 \cite{Poincare1895}. Two manifolds $M_1$ and $M_2$ are \textit{cobordant} if there is a manifold $N$ of one higher dimension such that the boundary of $N$ is the disjoint union $M_1 \sqcup M_2$. 
Cobordism of manifolds forms an abelian group in each dimension under disjoint union, and, in the cases considered in this thesis, a graded ring under Cartesian product. 
Pontryagin identified characteristic numbers as an invariant for cobordism in 1947 \cite{pontryagin47}, and applied the cobordism of framed manifolds to study the stable homotopy groups of spheres in 1955 \cite{pontryagin55}.
Ren\'e Thom, in his pioneering thesis in 1956 \cite{MR0061823}, built on the work of Pontryagin used stable homotopy theory instead to study cobordism, completing the cobordism classification of unoriented manifolds. In particular, Thom calculated the unoriented cobordism groups $\Omega^{O}_{n}$
by calculating the stable homotopy groups of the spectrum $MO$. This identification holds more broadly: the $O\langle n \rangle$-cobordism groups $\Omega^{\langle n \rangle}_{*}$ (i.e. the cobordism groups of $O\langle n \rangle$-manifolds) are isomorphic to the stable homotopy groups of the spectrum $MO\langle n \rangle$. 
The case of oriented cobordism ($n=2$) was calculated by the collective work of Milnor, Novikov, and Wall around 1960 \cite{MR119209} \cite{MR0120654}. In 1966, Anderson-Brown-Peterson successfully calculated the spin cobordism groups ($n=4$) \cite{MR0190939}. The next stage, $n=8$, corresponds to the string cobordism groups. At the prime 2, the first few dozen string cobordism groups were calculated by 
Giambalvo \cite{giambalvo70} and Gorbunov-Mahowald \cite{GorbunovMahowald1993}. Beginning with $n=9$, very few groups $\Omega^{\langle n \rangle}_{*}$ are explicitly written down. In this thesis, we calculate the 2-primary component of the groups $\Omega^{\langle n \rangle}_{*}$ for $9 \leq n \leq 17$ in a range of degrees depending on $n$. We provide the results of our calculation at the end of this section in Theorem \ref{thm: combined bordism groups}.

Beyond the classification itself, these calculations have applications to other phenomena in geometry, topology, and physics. In what follows, we describe a few of these applications and the relevant results 
in this thesis. 

For $m \geq 5$, let $\Theta_m$ denote the group of oriented, smooth, closed manifolds $\Sigma$ that are homotopy equivalent to the $m$-sphere, where the group operation is connected sum. By the work of Smale on the $h$-cobordism theorem \cite{smale1962}, $\Theta_{m}$ is the group of exotic spheres (under diffeomorphism). There is the Kervaire-Milnor exact sequence
\begin{align*}
0 \rightarrow bP_{m+1} \rightarrow \Theta_{m} \rightarrow (\text{Coker } J)_{m}
\end{align*}
where $bP_{m+1}\subseteq \Theta_m$ consists of all homotopy spheres that are the boundaries of parallelizable $(m+1)$-manifolds and $J$ is the stable $J$-homomorphism. In \cite{Stolz1985}, Stolz gave a condition for when the boundary of an almost closed manifold also bounds a parallelizable manifold that depends on
the kernel of the map $\pi_{*} \mathbb{S} \rightarrow \pi_{*} MO\langle k \rangle$, where $\mathbb{S}$ is the sphere spectrum. The calculations in this thesis provide information about this map. For instance, we prove 
\begin{theorem}
The boundary of any $8$-connected, almost closed $(18+d)$-manifold also bounds a parallelizable manifold for $d=7,8,10,11$.
\end{theorem}
We also show that this fails in other dimensions.
\begin{theorem}
There exists an $8$-connected, almost closed $(18+d)$ manifold with boundary non-trivial in $(\mathrm{Coker}\, J)_{18+d-1}$ for $d=3,4,5,6,14$.
\end{theorem}
These results deal with a range not handled by the main theorems of Stolz \cite{Stolz1985} or Burklund-Hahn-Senger \cite{BHS2023Boundaries} on this topic.

The spectra $MO \langle 2 \rangle$, $MO \langle 4 \rangle$, and $MO \langle 8 \rangle$ carry orientation maps that land in successively higher chromatic targets: 
\begin{align*}
MO \langle 2 \rangle &\rightarrow H\mathbb{Z} \\
MO \langle 4 \rangle &\rightarrow ko \\
MO \langle 8 \rangle &\rightarrow tmf 
\end{align*}
where $H\mathbb{Z}$ is the Eilenberg-MacLane spectrum representing integral homology, $ko$ is the connective real k-theory spectrum, and $tmf$ is the connective topological modular forms spectrum.
The latter two are known as the Atiyah-Bott-Shapiro and Ando-Hopkins-Rezk orientations respectively. 
The first two maps, after localizing at 2, split off in homotopy with the target occurring as a bottom summand, and the third map has long been suspected to. 
Here $ko$ and $tmf$ can be viewed as connective versions of the higher real $K$–theories $EO_1$ and $EO_2$. Hovey conjectured that this pattern continues at the prime 2, with the next stage being an orientation map from $MO\langle 9 \rangle = MFivebrane$ to $EO_3$ \cite{Hovey1997}. In light of these possible orientation maps and splittings, we prove 
\begin{theorem}
The spectrum $MO \langle 9 \rangle$ is indecomposable through degree 22.
\end{theorem}
As an $\mathcal{A}$-module, where $\mathcal{A}$ is the mod 2 Steenrod algebra, the mod 2 cohomology of $MO \langle 9 \rangle$ has four direct summands through degree 22. As a consequence of the above theorem, this splitting of $\mathcal{A}$-modules does not lift to a splitting of spectra.
Furthermore, our calculations show that the homotopy groups of $MO \langle 9 \rangle$ in this range are more complicated than the homotopy groups of $EO_{3}$ \cite{hahnshi}. 

\begin{corollary}
Suppose there exists an orientation map $MO\langle 9 \rangle \rightarrow eo_{3}$. Then the map does not split. 
\end{corollary}

If $MO\langle 9\rangle$ does admit a splitting, what is a (likely height 3) spectrum with the requisite homotopy and cohomology through degree 23? 
The orientation maps above also have geometric implications when considering their induced maps in homotopy, or their corresponding \textit{genera}. The $\hat{A}$-genus ($\hat{A}: \pi_{*}MO \langle 4 \rangle \rightarrow \pi_{*}ko$), introduced by Hirzebruch \cite{Hirzebruch1954}, 
vanishes on a spin manifold that admits a metric of everywhere positive scalar curvature; this is the content of the Lichnerowicz theorem~\cite{Lichnerowicz1963}, refined by Hitchin~\cite{Hitchin1974}. Analogously, the Witten genus $w: \pi_{*} MO \langle 8 \rangle \rightarrow tmf$ \cite{Witten1987} is conjectured 
to vanish on string manifolds that admit a metric of everywhere positive Ricci curvature. This is often referred to as the Stolz conjecture \cite{Stolz1996}. Recent work of Botvinnik--Labbi suggests that fivebrane bordism classes are tied to metrics of everywhere positive 3-curvature, hinting at a continuation of this pattern \cite{BotvinnikLabbi2014}.  

Cobordism groups of various types have become increasingly relevant to physicists, with applications to Symmetry Protected Topological (SPT) phases~\cite{Kapustin2014, Kapustin2015, FreedHopkins2016} and higher anomalies in quantum field theories~\cite{McNamaraVafa2019, Debray2023, Andriot2022, WanWang2019}.
In fact, many spectra $MO\langle n \rangle$ have distinguished names inherited from the relationship of $BO \langle n \rangle$ to string theory and M-theory \cite{SatiSchreiberStasheff2009Fivebrane} \cite{SatiSchreiberStasheff2012TwistedDifferential} \cite{Sati2015Ninebrane}, including the already mentioned $MO\langle 9 \rangle = MFivebrane$.
These physics applications often require knowledge of $\Omega^{F}(BG)$ for some $F$ and $G$. To date, choices of $F$ as $string$, $spin$, or $spin^c$ have proven fruitful. This thesis makes calculations such as $\Omega^{\langle n \rangle}_{*}(BG)$ more accessible. 
The author is interested in pursuing some of these applications in the future. For instance, such calculations may be relevant to determining some properties of $6D$ superconformal field theories with $N = (2,0)$ and $N = (1,0)$ supersymmetry, or more generally properties of backgrounds with $M5$-branes.

\subsection{Conventions}
Outside of Chapter \ref{chap:5}, all groups are implicitly 2-completed and all spectra are implicitly $2$-complete. The notation $(G)^n$ denotes the direct sum of $n$ copies of $G$. The sphere spectrum is denoted $\mathbb S$. We write $(0)$ for the trivial group.

\subsection{Computer automation}
There is a long tradition of computer automation in stable homotopy theory. This thesis continues in that vein. We primarily use Bob Bruner's \textbf{ext} \cite{bruner_ext_code}
program and Weinan Lin's \textbf{sseqcpp} program \cite{lin_sseqcpp}. Relevant Steenrod algebra module structures are calculated using Sage. Massey products are calculated using \textbf{ext}. 
The program \textbf{sseqcpp} includes some automated differential-deduction functionality. Differentials whose proofs were generated by \textbf{sseqcpp} are labeled accordingly.

\subsection{Table of $O\langle n \rangle$-cobordism groups}
\begin{theorem}\label{thm: combined bordism groups}
Table \ref{tab:combined-bordism-groups} describes the 2-primary component of the $O\langle n\rangle$-cobordism groups
$\Omega^{\langle n \rangle}_m$ for the values of $m$ listed. Entries of the form "$x$" where $x$ is a natural number indicate the order of the relevant group. The symbol $F$ represents a possibly trivial finite abelian group and $l$ is a finite multiple of $2$. 
\end{theorem}

\setlength{\tabcolsep}{4pt} 
\label{tab:combined-bordism-groups}
\tablehead{\hline%
$m$ & $\Omega^{\langle 9 \rangle}_m$ & $\Omega^{\langle 10\rangle}_m$ & $\Omega^{\langle 12\rangle}_m$ & $\Omega^{\langle 16\rangle}_m$ & $\Omega^{\langle 17\rangle}_m$\\%
\hline}
\tabletail{\hline%
\multicolumn{6}{r}{%
\small\slshape to be continued on the next page}\\}
\tablelasttail{\hline}
\begin{supertabular}{|c|p{3.0cm}|p{2.7cm}|p{2.9cm}|p{2.9cm}|p{2.9cm}|}
$0$  & $\Z$ & $\Z$ & $\Z$ & $\Z$ & $\Z$\\
$1$  & $\Z/2$ & $\Z/2$ & $\Z/2$ & $\Z/2$ & $\Z/2$\\
$2$  & $\Z/2$ & $\Z/2$ & $\Z/2$ & $\Z/2$ & $\Z/2$\\
$3$  & $\Z/8$ & $\Z/8$ & $\Z/8$ & $\Z/8$ & $\Z/8$\\
$4$  & $(0)$ & $(0)$ & $(0)$ & $(0)$ & $(0)$\\
$5$  & $(0)$ & $(0)$ & $(0)$ & $(0)$ & $(0)$\\
$6$  & $\Z/2$ & $\Z/2$ & $\Z/2$ & $\Z/2$ & $\Z/2$\\
$7$  & $\Z/16$ & $\Z/16$ & $\Z/16$ & $\Z/16$ & $\Z/16$\\
$8$  & $\Z/2$ & $\Z/2 \oplus \Z/2$ & $\Z/2 \oplus \Z/2$ & $\Z/2 \oplus \Z/2$ & $\Z/2 \oplus \Z/2$\\
$9$  & $\Z/2 \oplus \Z/2$ & $\Z/2 \oplus \Z/2$ & $(\Z/2)^{3}$ & $(\Z/2)^{3}$ & $(\Z/2)^{3}$\\
$10$ & $\Z/2$ & $\Z/2$ & $\Z/2$ & $\Z/2$ & $\Z/2$\\
$11$ & $(0)$ & $(0)$ & $(0)$ & $(0)$ & $\Z/8$\\
$12$ & $\Z$ & $\Z$ & $\Z$ & $(0)$ & $(0)$\\
$13$ & $(0)$ & $(0)$ & $(0)$ & $(0)$ & $(0)$\\
$14$ & $\Z/2 \oplus \Z/2$ & $\Z/2 \oplus \Z/2$ & $\Z/2 \oplus \Z/2$ & $\Z/2 \oplus \Z/2$ & $\Z/2 \oplus \Z/2$\\
$15$ & $\Z/2$ & $\Z/2$ & $\Z$ & $\Z$ & $\Z/32$\\
$16$ & $\Z \oplus \Z/2$ & $\Z \oplus \Z/2$ & $\Z\oplus \Z/2$ & $\Z \oplus \Z/2$ & $\Z/2$\\
$17$ & $\Z/2 \oplus \Z/2$ & $(\Z/2)^3$ & $(\Z/2)^3$ & $(\Z/2)^3$ & $(\Z/2)^3$\\
$18$ & $\Z/2$ & $\Z/8 \oplus \Z/2$ & $\Z/8 \oplus \Z/2$ & $\Z/8 \oplus \Z/2$ & $\Z/8 \oplus \Z/2$\\
$19$ & $(0)$ &  & $\Z/2$ & $\Z/2$ & $\Z/2$\\
$20$ & $\Z \oplus \Z/8$ & $\Z \oplus F$ & $\Z/8 \oplus \Z$ & $\Z/8 \oplus \Z$ & $\Z$\\
$21$ & $\Z/2$ &  & $\Z/2 \oplus \Z/2$ & $\Z/2 \oplus \Z/2$ & $\Z/2 \oplus \Z/2$\\
$22$ & $\Z/2$ &  & $\Z/2 \oplus \Z/2$ & $\Z/2 \oplus \Z/2$ & $\Z/2 \oplus \Z/2$\\
$23$ & $2\text{ or }4$&  & $\Z/2 \oplus \Z/4$ & $(\Z/2)^2 \oplus \Z/4$ & $\Z/8 \oplus \Z/2$\\
$24$ & $(\Z)^2\oplus(\Z/2)^2 \text{ or }(\Z)^2\oplus \Z/2$ &  & $\Z/2 \oplus (\Z)^2$ & $\Z/2 \oplus \Z$ & $\Z \oplus \Z/2 \oplus \Z/2$\\
$25$ &  &  & $2\text{ or }4$ & $\Z/2$ & $\Z/2$\\
$26$ &  &  &  & $\Z/2 \oplus \Z/2$ & $\Z/2 \oplus \Z/2$\\
$27$ &  &  &  & $(0)$ & $(0)$\\
$28$ & $\Z \oplus \Z \oplus \Z/2$ &  &$(\Z)^2\oplus \Z/2^l \oplus \Z/2$ & $\Z/2 \oplus \Z$ & $\Z$\\
$29$ &  &  & $(0)$ & $(0)$ & $(0)$\\
$30$ & $\Z/2$ &  & $\Z/2 \oplus \Z/2 \text{ or } \Z/4$ & $\Z/2 \oplus \Z/2$ & $\Z/2$\\
$31$ &  &  &  & $\Z/2 \oplus \Z/2$ & $\Z/2 \oplus \Z/2$\\
$32$ &  &  &  & $\Z \oplus \Z \oplus (\Z/2)^3$ & $\Z \oplus (\Z/2)^3$\\
$33$ &  &  &  & $32$ & $(\Z/2)^4$\\
$34$ &  &  &  &  & $\Z/4 \oplus (\Z/2)^4$\\
$35$ &  &  &  &  & \\
$36$ &  &  &  & $\Z \oplus \Z \oplus F$ & \\
\end{supertabular}


\clearpage

\section{Preliminaries}
\label{chap:Related Work}
\subsection{Tangential structures}
\FloatBarrier
Any real $n$-dimensional vector bundle $V$ over a space $X$ is classified by a map $f: X \rightarrow BO(n)$. 
We restrict to the case where $X=M$ is an $n$-dimensional compact manifold and $V=TM$, the tangent bundle of $M$. 
\begin{definition}\label{def: theta_structure}
Suppose $B\theta$ is a topological space (in particular a classifying space). A \textit{tangential $\theta$-structure} on a manifold $M$ is lift of the classifying map of $TM$ to $B\theta$:
\[\begin{tikzcd}
	TM & B\theta \\
	M & {BO(n)}
	\arrow[from=1-1, to=2-1]
	\arrow[from=1-2, to=2-2]
	\arrow[dashed, from=2-1, to=1-2]
	\arrow["f", from=2-1, to=2-2]
\end{tikzcd}\]
\end{definition}
\begin{remark}
Definition \ref{def: theta_structure} need not restict to classifying spaces. We do so because our examples of interest are classifying spaces, and for
convenient notation in the following definition. 
\end{remark}
\begin{definition}
A $\theta$-manifold is a manifold with an equivalence class of tangential $\theta$-structures.
\end{definition}
The $\theta$-manifolds of interest to this thesis are related to connective covers of $BO$. Recall that the \emph{Whitehead tower} of a (simply connected) CW complex $X$ is a sequence of fibrations
\[
\begin{tikzcd}[row sep=small]
\vdots \arrow[d] \\
X\langle n\rangle \arrow[d] \\
X\langle n-1\rangle \arrow[d] \\
\vdots \arrow[d] \\
X\langle 3\rangle \arrow[d] \\
X\langle 2\rangle \arrow[d] \\
X\langle 1\rangle \arrow[d] \\
X,
\end{tikzcd}
\]
together with maps $X\langle n\rangle \to X$ such that
\[
\pi_i\bigl(X\langle n\rangle\bigr)=0 \text{ for } i\leq n,
\qquad
\pi_i\bigl(X\langle n\rangle\bigr)\xrightarrow{\ \cong\ }\pi_i(X) \text{ for } i> n.
\]
Equivalently, $X\langle n\rangle$ is the \emph{$(n)$-connected cover} of $X$.

The maps $X\langle n\rangle \to X\langle n-1\rangle$ fit into principal fibrations
\[
K(\pi_{n-1}X,\,n-1)\longrightarrow X\langle n\rangle \longrightarrow X\langle n-1\rangle,
\]
classified by maps
\[
X\langle n-1\rangle \xrightarrow{\ k_n\ } K(\pi_{n-1}X,\,n-1),
\]
i.e.\ by cohomology classes
\[
k_n \in H^{n}\!\bigl(X\langle n-1\rangle;\,\pi_{n-1}X\bigr),
\]
 Given a finite CW complex $Y$, a map $Y\to X\langle n-1\rangle$ lifts through
the fibration $X\langle n\rangle \to X\langle n-1\rangle$ if and only if the pullback of
$k_n$ vanishes in $H^{n}(Y;\pi_{n-1}X)$.

Consider the case where $X=BO$. The Whitehead tower for $BO$ is given in Figure \ref{bo_whitehead}, where $w_i$ is the $i$th Stiefel-Whitney class and $p_{i}$ is the $i$th Pontryagin class. The classes $x_{9}$ and $x_{10}$ are exotic characteristic classes related to 
the fundamental classes of the cohomology of Eilenberg-Maclane spaces. In particular, consider the transgression $\tau_{9}$
induced by the fibration $BO\langle 9 \rangle \rightarrow BO\langle 8 \rangle \rightarrow K(\Z,8)$
\[
\dots \rightarrow H^{9}(BO\langle 9 \rangle; \Z/2) \xrightarrow{\tau_9} H^{10}(K(\Z,8); \Z/2) \rightarrow H^{10}(BO\langle 8 \rangle; \Z/2) \rightarrow \dots
\]
and the transgression $\tau_{10}$ induced by the fibration $BO\langle 10 \rangle \rightarrow BO\langle 9 \rangle \rightarrow K(\Z/2,9)$ 
\[
\dots \rightarrow H^{10}(BO\langle 10 \rangle; \Z/2) \xrightarrow{\tau_{10}} H^{11}(K(\Z/2,9); \Z/2) \rightarrow H^{11}(BO\langle 9 \rangle; \Z/2) \rightarrow \dots 
\]
Sati proves the following.
\begin{proposition}[Sati \cite{Sati2015Ninebrane}]\label{prop: sati_x9_x10}
The generators $x_{9}$ and $x_{10}$ are related to the fundamental classes $i_{8}$ and $i_{9}$ of $K(\Z,8)$ and $K(\Z/2,9)$ via $\tau_{9}(x_{9})=Sq^{2}i_{8}$ and $\tau_{10}(x_{10})=Sq^{2}i_{9}$. 
\end{proposition}
\begin{figure}[ht]
\centering
\[
\begin{tikzcd}[row sep=2.6em, column sep=4.6em]
K(\mathbb{Z},11) \arrow[r] &
BO\langle 13\rangle = BO\langle 16\rangle = BNinebrane \arrow[d] & \\
K(\mathbb{Z}/2,9) \arrow[r] &
BO\langle 11\rangle = BO\langle 12\rangle \arrow[d]
  \arrow[r, "\frac{1}{240}p_3"] &
K(\mathbb{Z},12)\simeq B^{11}U(1) \\
K(\mathbb{Z}/2,8) \arrow[r] &
BO\langle 10\rangle \arrow[d]
  \arrow[r, "x_{10}"] &
K(\mathbb{Z}/2,10)\simeq B^{9}\mathbb{R}P^\infty \\
K(\mathbb{Z},7) \arrow[r] &
BO\langle 9\rangle = BFivebrane \arrow[d]
  \arrow[r, "x_{9}"] &
K(\mathbb{Z}/2,9)\simeq B^{8}\mathbb{R}P^\infty \\
K(0,6) \arrow[r] &
BO\langle 8\rangle = BString \arrow[d, equals]
  \arrow[r, "\frac{1}{6}p_2"] &
K(\mathbb{Z},8)\simeq B^{7}U(1) \\
K(\mathbb{Z},3) \arrow[r] &
BO\langle 5\rangle = BString \arrow[d]
  \arrow[r] &
K(0,5)\simeq \mathrm{pt} \\
K(0,2) \arrow[r] &
BO\langle 4\rangle = BSpin \arrow[d, equals]
  \arrow[r, "\frac{1}{2}p_1"] &
K(\mathbb{Z},4)\simeq B^{3}U(1) \\
K(\mathbb{Z}/2,1) \arrow[r] &
BO\langle 3\rangle = BSpin \arrow[d]
  \arrow[r] &
K(0,3)\simeq \mathrm{pt} \\
K(\mathbb{Z}/2,0) \arrow[r] &
BO\langle 2\rangle = BSO \arrow[d]
  \arrow[r, "w_2"] &
K(\mathbb{Z}/2,2)\simeq B\mathbb{R}P^\infty \\
&
BO\langle 1\rangle = BO \arrow[r, "w_1"] &
K(\mathbb{Z}/2,1)\simeq \mathbb{R}P^\infty
\end{tikzcd}
\]
\caption{The first few stages of the Whitehead tower of $BO$. Taken from Andriot--Carqueville--Cribiori, Fig.\ 3.1.\label{bo_whitehead}}
\end{figure}
The characteristic classes in Figure \ref{bo_whitehead} give obstructions to the existence of $O\langle n \rangle$-structures, as shown in this diagram

\[\begin{tikzcd}
	& {} && {BFivebrane \cong BO\langle9\rangle} \\
	&&& {BString \cong BO\langle8\rangle} \\
	&&& {BSpin \cong BO\langle4\rangle} \\
	TM &&& {BSO\cong BO\langle2\rangle} \\
	M &&& BO
	\arrow[from=1-4, to=2-4]
	\arrow[from=2-4, to=3-4]
	\arrow[from=3-4, to=4-4]
	\arrow[from=4-1, to=5-1]
	\arrow[from=4-4, to=5-4]
	\arrow["{\frac{1}{6}p_2=0}"{description, pos=0.6}, shift left=2, dashed, from=5-1, to=1-4]
	\arrow["{\frac{1}{2}p_1=0}"{description, pos=0.6}, dashed, from=5-1, to=2-4]
	\arrow["{w_2=0}"{description, pos=0.6}, dashed, from=5-1, to=3-4]
	\arrow["{w_1=0}"{description, pos=0.6}, dashed, from=5-1, to=4-4]
	\arrow[from=5-1, to=5-4]
\end{tikzcd}\]
Thus, a manifold admits an $O\langle 2\rangle$-structure (equivalently an \textit{orientation}) if and only if $w_1(TM)=0$. This is the familiar notion of an \textit{orientation}. Furthermore, $M$ admits an $O\langle 4 \rangle$-structure (equivalently \textit{spin}) if and only if $w_{1}(TM)=0$
and $w_2(TM)=0$. 
\begin{remark}
There are lucid geometric interpretations of $O\langle 2\rangle$ and $O\langle 4\rangle$-structures. A manifold is an $O\langle 2\rangle$-manifold if the restriction of $TM$ to any embedded loop
is not twisted. A manifold is an $O\langle 4\rangle$-manifold if additionally the bundle on any embedded surface induced by $TM$ is not twisted \cite{Francis_connective_real_K_theory_KZ4}.
\end{remark}

\FloatBarrier
\subsection{Cobordism}
\FloatBarrier
\begin{definition}\label{def: cobordism}
  A \textit{cobordism} of smooth $n$-manifolds $M$ and $M^{'}$ is an $(n+1)$-manifold $N$ with a diffeomorphism $M\sqcup M^{'}\xrightarrow{\cong} \partial N$. We say $M$ 
  and $M^{'}$ are \textit{cobordant}. 
\end{definition}
An example of a cobordism is illustrated in Figure \ref{fig: cobordism}.
\begin{figure}[ht]\label{fig: cobordism}
\centering
\includegraphics[scale=0.3]{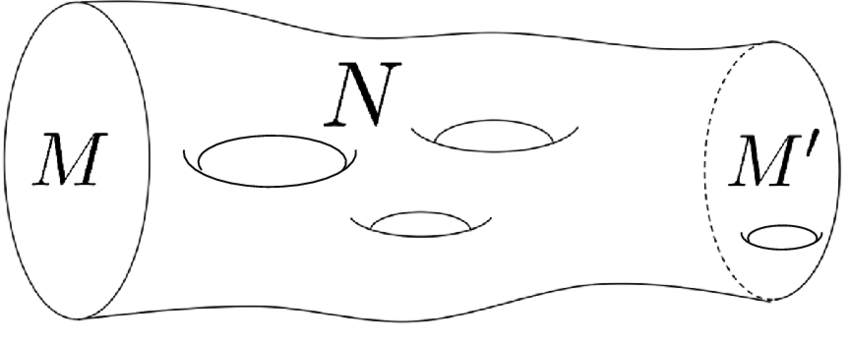}
\caption{A cobordism between manifolds $M$ and $M^{'}$. Taken from Wan-Wang \cite{WanWang2020}.\label{fig: cobordism}}
\end{figure}
\FloatBarrier
Definition \ref{def: cobordism} can be extended to the case of $\theta$-manifolds: 

\begin{definition}
Two smooth $n$-dimensional $\theta$-manifolds $M$ and $M^{'}$ are \textit{$\theta$-cobordant} if there is a compact $(n+1)$-dimensional $\theta$-manifold and a diffeomorphism 
$M\sqcup M^{'}\xrightarrow{\cong} \partial N$ inducing an equivalence of $\theta$-structures $\partial N \cong M\sqcup (-M^{'})$.
\end{definition}

The set of equivalence classes of $n$-dimensional $\theta$-manifolds under $\theta$-cobordism is denoted $\Omega^{\theta}_{n}$. The set $\Omega^{\theta}_{n}$ is an abelian group under disjoint union.

The primary tools for calculating the groups $\Omega^{\theta}_{*}$ come from stable homotopy theory. We now discuss how to translate the problem of computing the groups $\Omega^{\theta}_{n}$ into one in stable homotopy theory. We assume some familiarity with spectra. For a textbook account, see Malkiewich \cite{malkiewich_spectra}.

Let $\xi:V \rightarrow X$ be a vector bundle. Choose a metric on $\xi$ and denote the unit disc bundle of $\xi$ as $D(\xi)$ and the unit sphere bundle of $\xi$ as $S(\xi)$.
\begin{definition}\label{def:thom space}
The Thom space $Th(\xi)$ is the quotient $D(\xi)/S(\xi)$.
\end{definition}
Note that the homeomorphism type of $Th(\xi)$ does not depend on the choice of metric. The space $Th(\xi)$ can be thought of as a twisted (reduced) suspension of the base space $X$. In fact, the suspension $\Sigma X$ is a Thom space:
\begin{example}
Let $\xi: X \otimes \mathbb{R}^{n} \rightarrow X$ be the trivial vector bundle of rank $n$, then 
\[
Th(\xi) \cong \Sigma^{n} X
\]
\end{example} 
\begin{example} 
Let $\xi: \mathbb{M} \rightarrow S^{1}$ be the M\"obius bundle, then 
\[
Th(\xi) \cong \mathbb R P^{2}
\]
\end{example}
A Thom space of particular importance is $MO(n)$. 
\begin{definition}
Let $\gamma^{n}: EO(n) \rightarrow BO(n)$ be the universal $n$-plane bundle. Define 
\[
MO(n) := Th(\gamma^{n})
\]
\end{definition}
We are interested in constructing a spectrum out of the sequence of spaces $\{MO(n)\}$. To that end, we have the following sequence of inclusions 
\[
\dots \rightarrow BO(n-1) \xrightarrow{f_{n-1}} BO(n) \xrightarrow{f_n} BO(n+1) \rightarrow \dots
\]
The pullback of the $\gamma^{n+1}$ along $f_n$ is $\gamma^{n} \oplus \mathbb{R}$, giving maps 
\[
\Sigma MO_{n} = \Sigma Th(\gamma^{n})\cong Th(\gamma^{n} \oplus \mathbb R) \rightarrow Th(\gamma^{n+1})=MO_{n+1}
\]
\begin{definition}
Let $MO$ denote the spectrum given by the sequence of spaces $\{MO(n)\}$ with structure maps given above. 
\end{definition}
Let $\theta=O\langle m\rangle$ for some $m \geq 1$. There is an analogous construction for a universal $\theta$-bundle $\gamma_{\theta}^{n}: E\theta(n) \rightarrow B\theta(n)$. Denote the Thom space of $\gamma^{n}_{\theta}$ as $M\theta({n})$.
\begin{definition}
Let $M\theta$ be the spectrum given by the sequence of spaces $\{M\theta(n)\}$. 
\end{definition}
The spectra $MO$ and $M\theta$ are examples of \textit{Thom spectra}. The following theorem illustrates the fundamental relation Thom spectra play in cobordism theory. Write $\Omega^{\langle n \rangle}_{*}$ for the cobordism groups of $O\langle n \rangle$-manifolds. 

\begin{theorem}[Pontryagin \cite{pontryagin55}, Thom \cite{MR0061823}]
There is an isomorphism $\Omega^{\langle n\rangle}_{m} \cong \pi_{m}MO\langle n \rangle$. 
\end{theorem}
Pontryagin considered this isomorphism for the case of $n=\infty$, or \textit{framed cobordism}, where $\text{lim}_{n\rightarrow \infty} MO\langle n \rangle \cong \mathbb{S}$ for the sphere spectrum $\mathbb S$. Thom considered the case of $n=1$ and $n=2$, or unoriented and oriented cobordism.
The general statement of the isomorphism  is due to Lashof \cite{lashof63}.

\subsection{The Adams spectral sequence}
 In this section, we discuss the primary tool used for computing stable homotopy groups in this thesis, the \textit{Adams spectral sequence}.
The mod $2$ Steenrod algebra $\mathcal{A}$ is the graded associative algebra over $\mathbb{Z}/2$
generated by the stable cohomology operations $Sq^i$ ($i\ge 1$), subject to the Adem relations. For a spectrum $X$, the mod $2$ cohomology $H^*(X;\mathbb{Z}/2)$ is a graded left $\mathcal{A}$-module.
The Adams spectral sequence leverages this module structure. For a prime $p$ and a spectrum $E$ of finite type, there is a spectral sequence
\[
E^{s,t}_{2}=\text{Ext}^{s,t}_{\mathcal{A}}(H^*(E;\mathbb{Z}/2),\mathbb{Z}/2) \Rightarrow (\pi_{t-s}E)^{\wedge}_{2}
\]
with differentials
\[
d_r \colon \text{E}_r^{s,t}\to \text{E}_r^{s+r,t+r-1}.
\]
This spectral sequence is called the \textit{Adams spectral sequence}. For an introduction to the Adams spectral sequence, see Adams \cite{Adams_1974}. For a recent introduction with some example calculations, see Beaudry-Campbell \cite{beaudrycambpell2018}.
There are many algebraic constraints present in this spectral sequence that aid in calculating differentials. We discuss a few here that are used in later calculations. 

The Leibniz rule
\[
d_r(ab)=d_r(a)b+ad_r(b)
\]
usually determines a surprisingly large number of differentials. There is a generalization of the Leibniz rule to higher-order products, so-called \textit{Massey Products}. They were introduced by Massey in \cite{MR98366}. See also May \cite{MR238929} for an exhaustive account. The $E_r$-page of the Adams spectral sequence is the homology of a differential graded algebra, and so each has Massey products. The Massey product of elements $a$, $b$, and $c$ is denoted by the bracket 
$\langle a, b, c \rangle$. It is only defined when $ab=0$ and $bc=0$. Moss gave a generalization of the Leibniz rule with Massey products
\begin{theorem}[Moss Theorem 1.1 (ii) \cite{MR266216}]\label{mossleibniz}
Let $a$, $b$, and $c$ be elements in the $E_r$-page of the Adams spectral sequence for a spectrum $X$. Suppose that $ab=0$, $bc=0$, $d_{r}(b)a=0$, and $bd_r(c)=0$. Then
\[
d_r\langle a,b,c \rangle \subseteq \langle d_r(a),b,c \rangle + \langle a,d_r(b),c \rangle + \langle a,b,d_r(c) \rangle
\]
where all brackets are calculated in the $E_r$-page. 
\end{theorem}

Following Isaksen-Wang-Xu \cite{MR4588596}, we refer to this as \textit{Moss' higher Leibniz rule}. For possible differentials between $h_0$-towers, there is the following theorem of May-Milgram. 

\begin{theorem}[May-Milgram \cite{maymilgram}]\label{maymilgram}
Consider the Adams spectral sequence computing the $2$-completed stable homotopy groups of a space or spectrum $X$. There is a bijection between
nonzero $d_{r}$ differentials out of $h_0$-towers originating on the $0$-line in topological degree $k$ and $\Z/2^{r}$ summands in $H^{k+1}(X; \Z)$.
\end{theorem}

This is not how it is stated in the original paper. See Debray et al \cite{debraybordia}, specifically Lemma $D.7$ and Theorem $D.9$, for a derivation of this phrasing. 

The spectra $MO\langle n \rangle$ are all ring spectra, so there exists a unit map $\mathbb S \rightarrow MO\langle n \rangle$. This map induces a map of Adams spectral sequence $E_{2}$-pages 
\[
\text{Ext}^{s,t}_{\mathcal{A}}(\mathbb{Z}/2,\mathbb{Z}/2) \rightarrow \text{Ext}^{s,t}_{\mathcal{A}}(H^*(MO\langle n \rangle;\mathbb{Z}/2),\mathbb{Z}/2)
\]
that is natural with respect to the $d_r$-differential. Thus, differentials for $\mathbb{S}$ can help determine differentials for $MO\langle n \rangle$. This is a powerful relation due to the  substantial amount of work done 
to calculate the Adams spectral sequence for $\mathbb{S}$. The Adams spectral sequence for $\mathbb S$, as it pertains to $\pi_{*}\mathbb S$ for $* \leq 44$, was calculated completely by the collective work of 
May \cite{may1965}, Mahowald-Tangora \cite{BarrattJonesMahowald1984}, Barratt-Mahowald-Tangora \cite{BarrattMahowaldTangora1970}, Milgram \cite{milgram1972}, Tangora \cite{tangora1970}, and Bruner \cite{bruner1984}. 
A comprehensive account for this range is given in Chapter 11 of Bruner-Rognes \cite{BrunerRognestmf}. The author found this account particularly convenient as a reference, and so 
it is frequently cited in this thesis for facts about $\pi_{*}\mathbb S$.
\begin{proposition}\label{prop: sphere differentials}
Table $\ref{tab:sphere_differentials}$ contains all nonzero $d_r$-differentials, $r \geq 2$, on $E_r(\mathbb S)$ in stems $\leq 44$.
\end{proposition}
\renewcommand{\arraystretch}{1.15}
\begin{longtable}{l l}
\caption[]{Differentials in the Adams spectral sequence for the sphere spectrum $\mathbb S$.}
\label{tab:sphere_differentials}\\
\toprule
$r$ & Differential\\
\midrule
\endfirsthead

\toprule
$r$ & Differential\\
\midrule
\endhead

\bottomrule
\endfoot

$2$ & $d_2(h_4)=h_0h_3^2$\\
$2$ & $d_2(e_0)=h_1^2d_0$\\
$2$ & $d_2(f_0)=h_0^2e_0$\\
$2$ & $d_2(i)=h_0Pd_0$\\
$2$ & $d_2(j)=h_0Pe_0$\\
$2$ & $d_2(k)=h_0d_0^2$\\
$2$ & $d_2(h_5)=h_0h_4^2$\\
$2$ & $d_2(l)=h_0d_0e_0$\\
$2$ & $d_2(P_j)=h_0^2P^2e_0$\\
$2$ & $d_2(m)=h_0d_0g$\\

$3$ & $d_3(h_0h_4)=h_0d_0$\\
$3$ & $d_3(\Delta h_2^2)=h_0^2k$\\
$3$ & $d_3(e_1)=h_1t$\\
$3$ & $d_3(h_0^3h_5)=h_0\Delta h_2^2$\\

$4$ & $d_4(d_0e_0)=P^2d_0$\\
$4$ & $d_4(h_0^3\Delta h_3^2)=h_0^2d_0i$\\
$4$ & $d_4(e_0g)=d_0Pd_0$\\
$4$ & $d_4(h_3h_5)=h_0x$\\

\end{longtable}
\begin{proof}
For proofs of these differentials, see Theorems 11.52, 11.54, and 11.56 of Bruner-Rognes \cite{BrunerRognestmf}. 
\end{proof}
\subsection{The $J$-homomorphism}
The \textit{$J$-homomorphism} is the homomorphism 
\[
\pi_{i}(SO(n)) \rightarrow \pi_{i+n}(\Sigma^{n}SO(n)) \xrightarrow{\mu} \pi_{i+n}(S^n)
\]
The map $\mu$ is defined using the \textit{Hopf construction} \cite{hopf1935}, which takes a map $S^{q} \times S^{r} \rightarrow X$ and constructs a map 
$\Sigma (S^{q} \wedge  S^{r}) \rightarrow \Sigma X$. 

Choose an element $[f]\in \pi_{i}(SO(n))$, i.e. the homotopy class of a map $f: S^{i} \rightarrow SO(n)$.
There is also a map 
\[
\alpha: S^{n-1} \times SO(n) \rightarrow S^{n-1}
\]
given by the action of $SO(n)$ on $\mathbb R^{n}$ restricted to $S^{n-1}$. Composing these maps we have
\[
 S^{i} \times S^{n-1} \xrightarrow{f\times id} SO(n) \times S^{n-1} \xrightarrow{\alpha} S^{n-1}
\]
Applying the Hopf construction to this composition produces a map 
\[
\Sigma(S^{i}\wedge S^{n-1})=S^{i+n} \rightarrow S^{n}
\]
which is an element of $\pi_{i+1}(S^n)$ representing the image of $f$ under $\mu$. Stabilizing, we obtain the \textit{stable J-homomorphism}
\[
J: \pi_{i}(SO) \rightarrow \pi_{i} \mathbb{S}
\]
We will refer to the stable version as simply the $J$-homomorphism for the rest of this thesis. 
\begin{definition}
Let $J^{\langle n \rangle}$ denote the composition
\[
\pi_{*}(SO) \xrightarrow{J} \pi_{*} S \xrightarrow{\eta} \pi_{*} MO \langle n \rangle
\]
where $\eta$ is the unit map. 
\end{definition}

There are a couple of results regarding $\text{Im } J$ that will be used in later chapters. The first is due to Adams. 
\begin{theorem}[Theorem 1.3 Adams \cite{adamsjx}]\label{adamsimj}
Suppose that $r \cong 1 \pmod{8}$ and $r > 1$. Then $J$ is a monomorphism and $\pi_{r} \mathbb{S}$ contains a direct summand $\mathbb{Z}/2 \oplus \mathbb{Z}/2$, one summand being generated by $\mu_{r}$ and the other being $\text{Im } J$. 
\end{theorem}

Here, $\{\mu_{r}\}$ is a family of elements in $\pi_{*} \mathbb{S}$ beginning with $\mu_{1}=\eta$ and $\mu_{2}=\eta^{2}$. Mahowald described the elements of $\text{Im } J$ on the $E_{2}$-page of the Adams spectral sequence. 
\begin{theorem}[Mahowald \cite{Mahowald1970OrderImageJ}]\label{mahowaldimj}
The elements $P^ic_{0}$, $P^{i}h_{1}c_{0}$, $i\geq 1$, $P^ih_2$, $i\geq1$, in $Ext_{\mathcal{A}}(\mathbb{Z}/2,\mathbb{Z}/2)$ represent the image of $J$ in dimension $j\cong 0, 1, 3 \pmod{8}$. In dimension $8j-1$, an $h_0$-tower which ends at the ``Adams edge" represents the image of $J$ in that dimension.
\end{theorem}

\subsection{Background on $MO\langle n \rangle$} 
The purpose of this section is to outline some results about $MO\langle n \rangle$ that are used in later chapters. 
\subsubsection{The cohomology of $MO\langle n \rangle$}
By the Thom isomorphism, $H^*(MO\langle n \rangle; \Z/2)$ is isomorphic to $H^*(BO\langle n \rangle; \mathbb{Z}/2)\{U\}$ as an $H^*(BO\langle n \rangle; \mathbb{Z}/2)$-module, where $U$ is the Thom class.
Stong calculated the ring $H^*(BO\langle n \rangle; \mathbb{Z}/2)$ for all $n$.
\begin{theorem}[ Stong \cite{Stong1963Determination}]
    If $k \equiv 0,1,2,4 \pmod{8}$, then there is an isomorphism
    \[
         H^{*}(BO\langle n \rangle;\mathbb{Z}/2) \cong H^{*}(K(\pi_{n}(BO),n))/\langle Sq^i i_{n} \rangle \otimes \mathbb{Z}/2[\theta_{i} | L(i) > \phi(0,n)]
    \]
\end{theorem}
where $L(i)$ is one plus the number of ones in the binary expansion of $i-1$ and $\theta_i \cong w_i$ modulo decomposables. The action of Steenrod squares on each $w_i$ is determined by the Wu formula:
\[
Sq^i w_j = \sum_{k=0}^i \binom{j+k-i-1}{k} w_{i-k}w_{j+k} 
\]
The value of $Sq^i$ on a product $xy$ is given by the Cartan formula:
\[
Sq^i(xy) = \sum_{j=0}^{i} Sq^{j}(x)\, Sq^{i-j}(y).
\]
If $\pi_{n}(BO) \cong \Z/2$, the graded ring $H^{*}(K(\pi_{n}(BO),n); \Z/2)$ is polynomial on generators $Sq^I I$ for each admissible sequence $I$ of excess $<$ n \cite{cartan1954}. If $\pi_{n}(BO) \cong \Z$, it is the same except 
admissible sequences that contain the integer 1 are excluded because they are equal to zero \cite{Serre1953}. The image of $Sq^i U$ restricted to the second factor of  $H^{*}(BO\langle n \rangle;\mathbb{Z}/2)$ is $w_i U$. There are, however, squaring operations applied to $U$ that mix 
terms from the two factors. For example, if $n=9$, we have 
\[
 Sq^{16}U = (Sq^4Sq^2Sq^1+Sq^7)i_{9}
\]
while $\theta_{16}=w_{16}=0\in  H^{*}(BO\langle 9 \rangle;\mathbb{Z}/2)$. Relations like these make the $\mathcal{A}$-module structure of $H^{*}(MO\langle n \rangle;\mathbb{Z}/2)$ potentially very complicated. 
\subsubsection{The unit map}
The unit map $\mathbb S \rightarrow MO\langle n \rangle$ and the $J$-homomorphism have an illuminating relationship given by the following theorem of Hovey.

\begin{theorem}[Theorem 2.2.1 Hovey \cite{hovey97}]\label{hoveyimj}

The composite 
\[
\text{Im } J \rightarrow \pi_{*} \mathbb{S} \rightarrow \pi_{*} MO \langle n \rangle 
\]
is injective in dimensions $\leq n-2$ and 0 in dimensions $\geq n-1$. 
\end{theorem}
Thus, $\text{Im } J$ is contained in the the kernel of $\pi_{*} \mathbb S \rightarrow \pi_{*} MO\langle n \rangle$ for $* \geq n-1$. There are special cases where 
we know that $\text{Im } J$ in fact generates the kernel. 

\begin{theorem}[Senger-Zhang \cite{SengerZhang2025Inertia} Stolz \cite{Stolz1985}]\label{thm: 2n mo_n kernel}
The kernel of the unit map 
\[
\pi_{2n} \mathbb S \rightarrow \pi_{2n} MO \langle n \rangle
\]
is generated by the image of the $J$-homomorphism for $n\geq 10$.
\end{theorem}
And there are exceptional cases, such as the following.
\begin{theorem}[Theorem 6.1 Burklund-Senger \cite{BurklundSenger2024Geography}]
\label{lem:mo9_17_kernel}
The kernel of the map $\pi_{17} \mathbb{S} \rightarrow \pi_{17} MO\langle 9 \rangle$ is generated by the image of the $J$-homomorphism and $\eta\eta_{4}$. 
\end{theorem}

\clearpage

\section{Fivebrane Cobordism }
\label{chap:fivebrane}
In this chapter, we calculate $\pi_* MO\langle 9 \rangle$ in a range of degrees, i.e. the fivebrane cobordism groups. 

\subsection{The cohomology of $MO\langle 9 \rangle$}
\begin{proposition}[\cite{Stong1963Determination}]
    $H^{*}(BO\langle9\rangle;\mathbb{Z}/2)$ is isomorphic to:
    \[
        H^{*}(K(\mathbb{Z}/2,9); \Z/2)/\langle Sq^2 i_9 \rangle \otimes \mathbb{Z}/2[\theta_{i} | L(i) > 5]
    \]
\end{proposition}
where $L(i)$ is one plus the number of ones in the binary expansion of $i-1$. We calculate the $\mathcal{A}$-module structure through degree 55 using a computer implementation in Sage.
\begin{proposition}
Through degree 23, $H^*(MO\langle 9 \rangle;\mathbb{Z}/2)$ splits as an $\mathcal{A}$-module as
\[
\mathcal{A}//\mathcal{A}(3) \oplus \Sigma^9 \mathcal{A}/\mathcal{A}(Sq^2,Sq^{10},Sq^{12}, Sq^{26}) \oplus \Sigma^{19} \mathcal{A}/\mathcal{A}(Sq^{5},Sq^{21}) \oplus \Sigma^{22} \mathcal{A}/\mathcal{A}(Sq^{7},Sq^{23})
\]
\end{proposition}
While only the mod 2 cohomology of $MO \langle 9 \rangle$ is relevant to calculating the $E_2$-page of the Adams spectral sequence for $MO\langle 9 \rangle$, the rational cohomology is useful for certain Adams differentials. 
\begin{proposition}\label{prop: rational bo8}
The rational cohomology of $BO\langle 8 \rangle$ is isomorphic to 
\[
 \mathbb{Q}[p_2,p_3,p_4,p_5,\dots]
\]
where $|p_i|=4i$. 
\end{proposition}
\begin{proof}
Consider the Serre spectral sequence associated to the fibration $K(\mathbb{Z},3) \rightarrow BO\langle 8\rangle  \rightarrow BO\langle 4 \rangle $ 
\[
E_2^{s,t}\cong H^s(BO\langle 4 \rangle ;H^t(K(\mathbb{Z},3);\mathbb{Q})) \Rightarrow H^{s+t}(BO\langle 8 \rangle; \mathbb{Q})
\]
Note that $BO\langle 4 \rangle $ is simply-connected so the $E_2$-page need not be stated with local coefficients. Recall that $H^*(K(\mathbb{Z},3);\mathbb{Q}) \cong \bigwedge[x_3]$ where $|x_3|=3$ (Theorem 5.1 Berglund \cite{Berglundnotes}, also \cite{morita2001} \cite{griffithsmorgan2013}) and $H^*(BO\langle 4 \rangle ;\mathbb{Q})\cong \mathbb{Q}[p_1,p_2,p_3,\dots]$ where $|p_i|=4i$.
The first possible nonzero differential is $d_4(x_3)$. This differential is the transgression (Theorem 6.8 McCleary \cite{McCleary}). Furthermore, the fibration $K(\mathbb{Z},3) \rightarrow BO\langle 8 \rangle \rightarrow BO\langle 4 \rangle $ is pulled back over the classifying map 
$\theta:BO\langle 4 \rangle  \rightarrow K(\mathbb{Z},4)$ corresponding to the class $\frac{p_1}{2} \in H^4(BO\langle 4 \rangle ;\mathbb{Z})$. Since this is a principal $K(\mathbb{Z},3)$ fibration, the transgression $d_4(x_3)$ equals the pullback under this map of the fundamental class of $K(\mathbb{Z},4)$ (see proof of Lemma $8^{bis}.28$ McCleary \cite{McCleary}). Thus, $d_4(x_3)=\frac{p_1}{2}$. 
By the Leibniz rule, $d_4(p_1^k x_3)=\frac{p_1^{k+1}}{2}$, so all powers of $p_1$ die. No further differentials are possible for degree reasons, so
$H^*(BO\langle 8\rangle ;\mathbb{Q})\cong \mathbb{Q}[p_2,p_3,\dots]$.
\end{proof}
\begin{proposition}\label{prop: rational bo9}
The rational cohomology of $BO\langle 9 \rangle$ is isomorphic to 
\[
 \mathbb{Q}[p_3,p_4,p_5,\dots]
\]
where $|p_i|=4i$. 
\end{proposition}
\begin{proof}
Consider the Serre spectral sequence associated to the fibration $K(\mathbb{Z},7) \rightarrow BO\langle 9 \rangle \rightarrow BO\langle 8 \rangle$ 
\[
E_2^{s,t}\cong H^s(BO\langle 8 \rangle;H^t(K(\mathbb{Z},7);\mathbb{Q})) \Rightarrow H^{s+t}(BO\langle 9 \rangle; \mathbb{Q})
\]
Recall that $H^*(K(\mathbb{Z},7);\mathbb{Q}) \cong \bigwedge[x_7]$ where $|x|=7$ (Theorem 5.1 Berglund \cite{Berglundnotes} also \cite{morita2001} \cite{griffithsmorgan2013}) and $H^*(BO\langle 8 \rangle;\mathbb{Q})\cong \mathbb{Q}[p_2,p_3,\dots]$ by Proposition \ref{prop: rational bo8}. 
The result follows from a similar argument to the proof of  Proposition \ref{prop: rational bo8}. Here, the classifying map $BO\langle 8 \rangle \rightarrow K(\mathbb{Z},8)$ corresponds to the class $\frac{p_2}{6} \in H^8(BO\langle 8 \rangle; \mathbb{Z})$, and so $d_8(x_7)=\frac{p_2}{6}$. 
\end{proof}

\begin{proposition}\label{prop: bo9 cohomology groups}
Table \ref{tab:bo9_groups} describes the 2-adic cohomology groups $H^n(BO\langle 9 \rangle;\mathbb{Z})$ for $0\leq n \leq 23$. 
\end{proposition}
\begin{center}
\captionof{table}[2-adic cohomology of $BO\langle 9 \rangle$]{The 2-adic cohomology groups of $BO\langle 9 \rangle$ through degree $23$.}
\label{tab:bo9_groups}
\tablehead{\hline%
$n$ & $H^n(BO\langle 9 \rangle;\Z)$ & $n$ & $H^n(BO\langle 9 \rangle;\Z)$\\%
\hline}
\tabletail{\hline%
\multicolumn{4}{r}{%
\small\slshape to be continued on the next page}\\}
\tablelasttail{\hline}
\begin{supertabular}{|c|p{3.5cm}|c|p{3.5cm}|}
$0$  & $\Z$    & $12$ & $\Z$\\
$1$  & $(0)$   & $13$ & $(0)$\\
$2$  & $(0)$   & $14$ & $\Z/2$\\
$3$  & $(0)$   & $15$ & $(0)$\\
$4$  & $(0)$   & $16$ & $\Z/2 \oplus \Z$\\
$5$  & $(0)$   & $17$ & $\Z/2$\\
$6$  & $(0)$   & $18$ & $\Z/2$\\
$7$  & $(0)$   & $19$ & $\Z/2$\\
$8$  & $(0)$   & $20$ & $\Z/2 \oplus \Z$\\
$9$  & $(0)$  & $21$ & $(0)$\\
$10$ & $\Z$   & $22$ & $(\Z/2)^{2}$\\
$11$ & $(0)$  & $23$ & $(\Z/2)^{2}$\\
\end{supertabular}
\end{center}
\FloatBarrier
\begin{proof}
We will begin by working in homology instead of cohomology. The homology groups $H_{*}(BO \langle 9 \rangle; \mathbb{Z})$ are isomorphic to $H\mathbb{Z}_{*}(\Sigma^{\infty}BO\langle 9 \rangle) \cong \pi_{*}(H\mathbb{Z} \wedge \Sigma^{\infty}BO\langle 9 \rangle)$, where $H\mathbb{Z}$ is the Eilenberg-Maclane spectrum representing integral homology and $\Sigma^{\infty}BO\langle 9 \rangle$ is the suspension spectrum of 
$BO\langle 9\rangle$. Consider the Adams spectral sequence
\[
E^{s,t}_{2} = \text{Ext}^{s,t}_{\mathcal{A}}(H^*(H\mathbb{Z} \wedge \Sigma^{\infty} BO\langle 9 \rangle;\mathbb{Z}/2),\mathbb{Z}/2) \Rightarrow \pi_{t-s}(H\mathbb{Z}\wedge \Sigma^{\infty} BO\langle 9 \rangle) \cong H\mathbb{Z}_{t-s}(\Sigma^{\infty} BO\langle 9 \rangle)
\]
It is well-known that the mod 2 cohomology of $H\mathbb{Z}$ is isomorphic to $\mathcal{A}/Sq^1\mathcal{A}$ as an $\mathcal{A}$-module. By the change-of-rings isomorphism, the $E_2$-page then simplifies to 
\[
\text{Ext}^{s,t}_{\mathcal{A}/Sq^1\mathcal{A}}(H^*(\Sigma^{\infty} BO\langle 9 \rangle;\mathbb{Z}/2),\mathbb{Z}/2)
\]
Thus, we need only calculate the action of $Sq^1$ on $H^*(BO\langle 9 \rangle;\mathbb{Z}/2)$. The $E_2$-page is then given by Figure \ref{fig:hz_bo9}. All possible differentials are ruled out for degree reasons or by $h_0$-linearity. Thus, the spectral sequence collapses at the $E_2$-page and there are no possible hidden $2$-extensions, and so the groups $H_{n}(BO\langle 9 \rangle ; \Z)$ can be read off the $E_2$-page. The cohomology groups and homology groups of $BO\langle 9 \rangle$ are related by the Universal Coefficient Theorem
\[
0 \longrightarrow \text{Ext}_{\Z}\!\bigl(H_{n-1}(BO\langle 9 \rangle;\Z),\,\Z\bigr)
\longrightarrow H^{n}(BO\langle 9 \rangle;\Z)
\longrightarrow \text{Hom}_{\Z}\!\bigl(H_{n}(BO\langle 9 \rangle;\Z),\,\Z\bigr)
\longrightarrow 0.
\]
Based on our calculation, $H_{n}(BO\langle 9 \rangle ; \Z)$ is a direct sum of copies of $\Z$ and $\mathbb{Z}/2$, so the relevant $\text{Ext}$ and $\text{Hom}$ groups are
\begin{align*}
\text{Hom}_{\Z}({\Z,\Z}) &\cong \Z \\
\text{Hom}_{\Z}({\Z/2,\Z}) &\cong 0 \\
\text{Ext}_{\Z}({\Z,\Z}) &\cong 0 \\ 
\text{Ext}_{\Z}({\Z/2,\Z}) &\cong \Z/2
\end{align*}
Applying the Universal Coefficient Theorem for $0\leq n \leq 23$ then takes the form 

\[
0 \rightarrow (\Z/2)^{a} \rightarrow H^n(BO\langle 9 \rangle; \Z) \rightarrow (\Z)^{b} \rightarrow 0 
\]

where $a$ is the number of $\Z/2$ summands in $H_{n-1}(BO\langle 9 \rangle; \Z)$ and $b$ is the number of $\Z$ summands in $H_{n}(BO\langle 9 \rangle; \Z)$. Of course, $\Z$ is a projective $\Z$-module, so the short exact sequence will split and 
$H_{n}(BO\langle 9 \rangle; \Z) \cong (\Z/2)^{a} \oplus (\Z)^{b}$. The particular values for $a$ and $b$ for each $n$ can be read off the Adams chart, and the result follows.
\begin{figure}
\includegraphics[width=\linewidth]{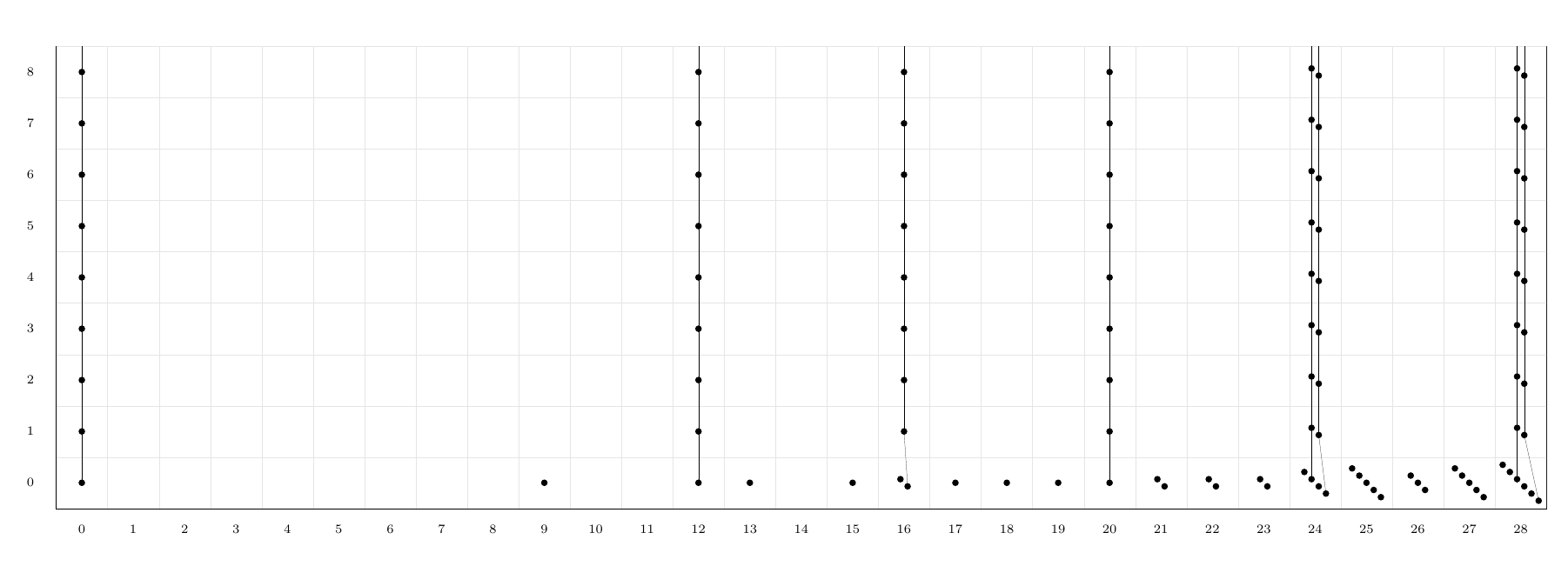}
\caption{The $E_2$-page of the Adams spectral sequence for $HZ_{*}BO\langle 9 \rangle$.\label{fig:hz_bo9}}
\end{figure}
\end{proof}
\FloatBarrier
\subsection{Adams charts}
This section contains charts for the Adams spectral sequence for $MO\langle 9 \rangle$. The $y$-axis is the Adams filtration $s$ and the $x$ axis is the stem $t-s$. The $E_r$-page is displayed along with the $d_r$-differentials. The elements $h_0$, $h_1$, and $h_2$ have had their names suppressed. The dotted lines on the $E_{\infty}$-page represent possibly nonzero differentials that have yet to be determined. 
\clearpage
\FloatBarrier
\begin{figure}[!p]
\centering

\begin{subfigure}{\linewidth}
  \centering
  \includegraphics[width=\linewidth,
    trim=0.5cm 0cm 0cm 0cm,clip]{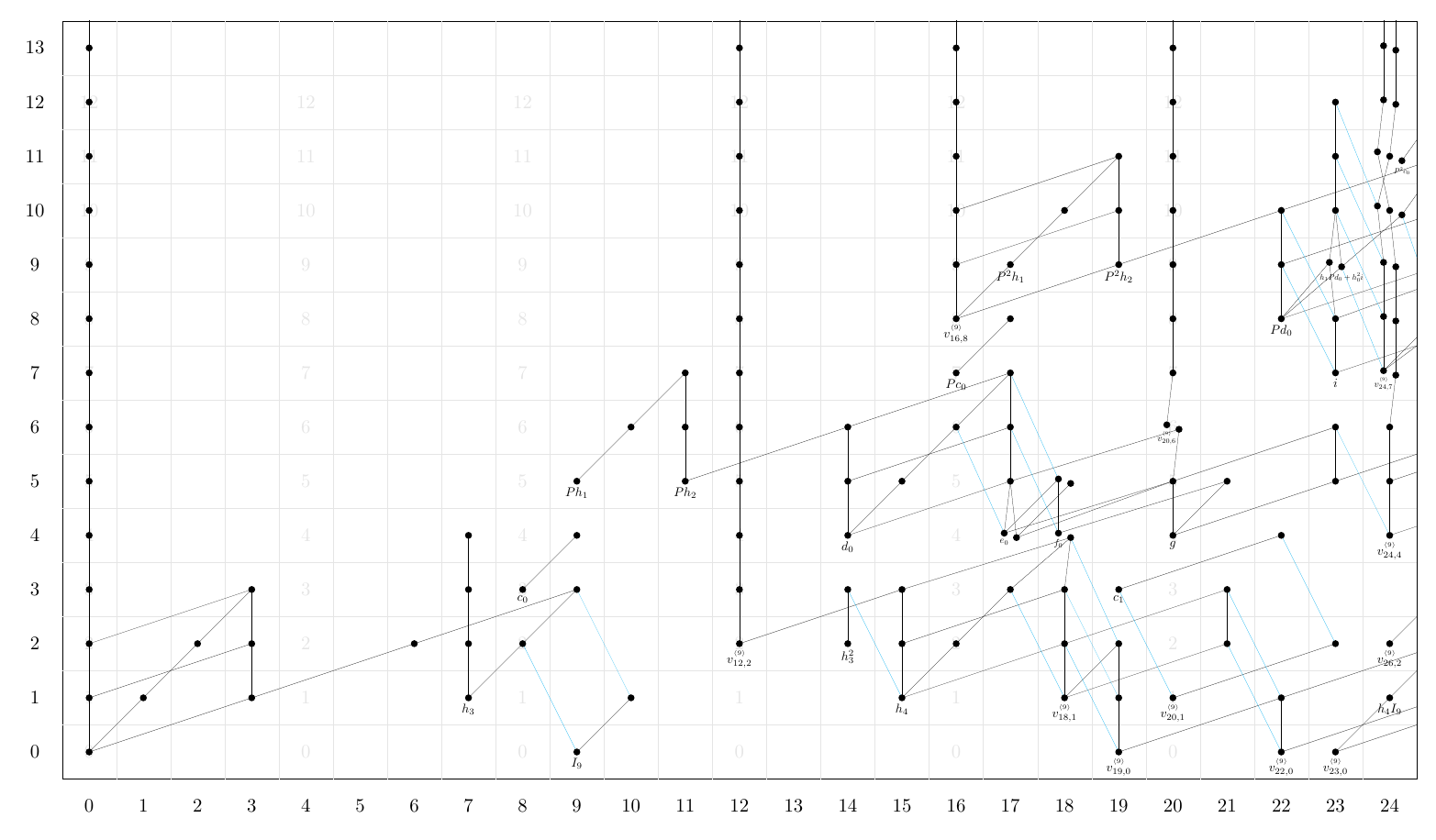}
  \caption{$E_2$-page $MO \langle 9 \rangle $ (stems 0-24)}
  \label{fig:mo9_e2_0_24}
\end{subfigure}

\vspace{-0.4\baselineskip}

\begin{subfigure}{\linewidth}
  \centering
  \includegraphics[width=\linewidth,
    trim=0.5cm 0cm 0cm 0cm,clip]{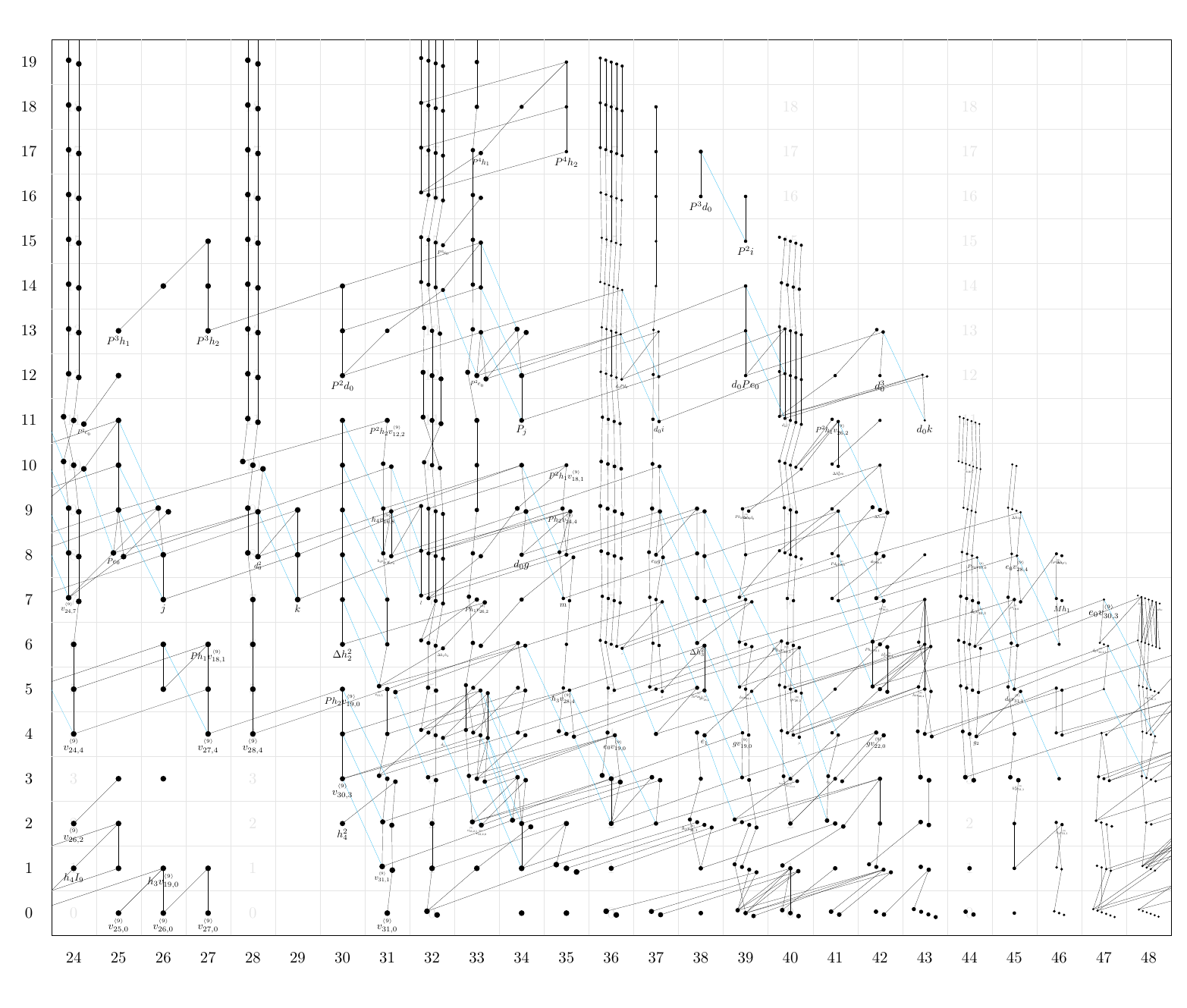}
  \caption{$E_2$-page $MO \langle 9 \rangle $ (stems 24-48)}
  \label{fig:mo9_e2_24_48}
\end{subfigure}

\caption{$E_2$-page $MO\langle 9 \rangle$\label{fig:mo9_e2}}
\end{figure}

\begin{figure}[!p]
\centering
\begin{subfigure}{\linewidth}
  \centering
  \includegraphics[width=\linewidth,
    trim=0.5cm 0cm 0cm 0cm,clip]{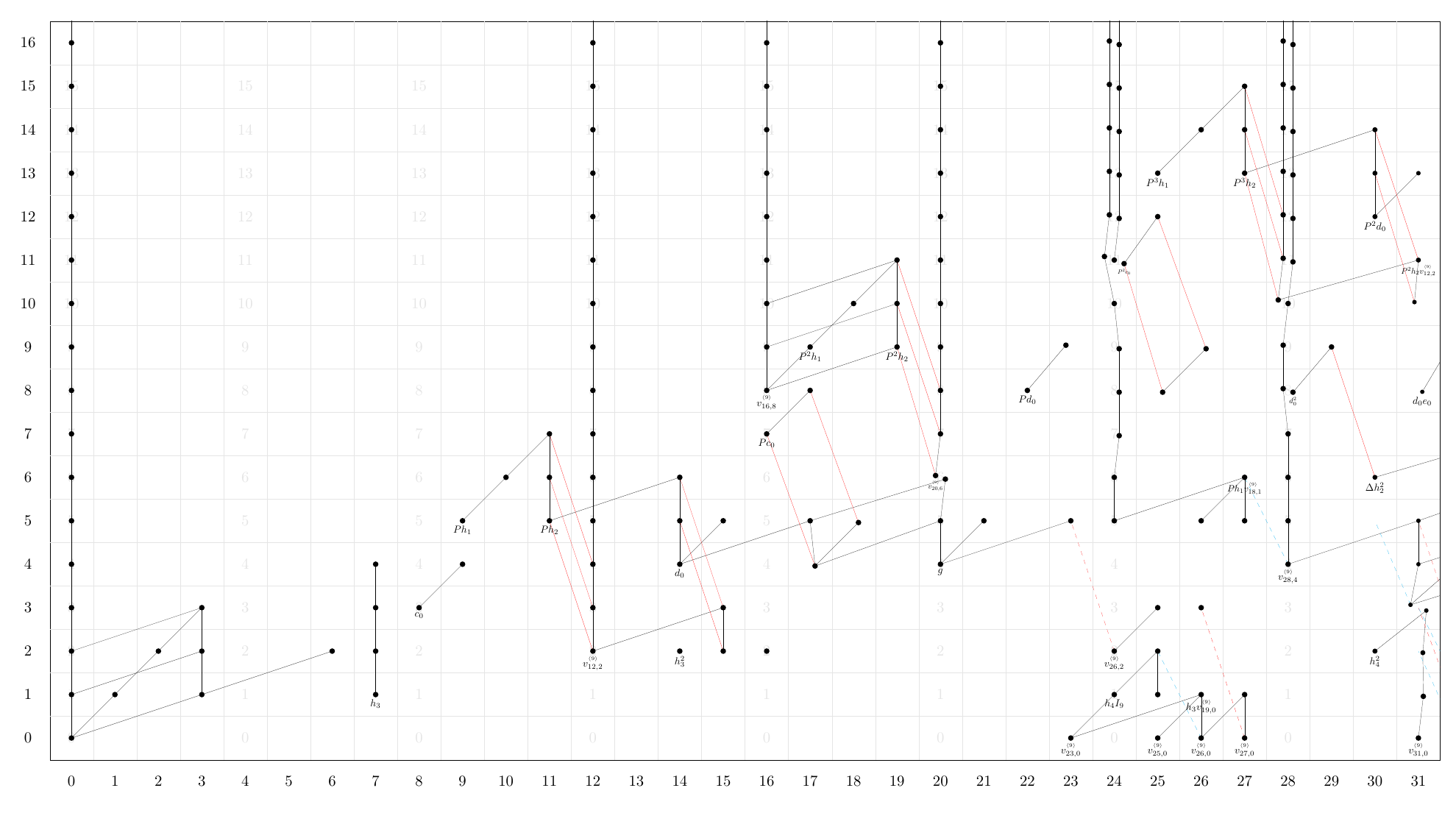}
  \caption{$E_3$-page $MO \langle 9 \rangle $ (stems 0-31)}
  \label{fig:mo9_e3_0_31}
\end{subfigure}


\begin{subfigure}{\linewidth}
  \centering
  \includegraphics[width=\linewidth,
    trim=0.5cm 0cm 0cm 0cm,clip]{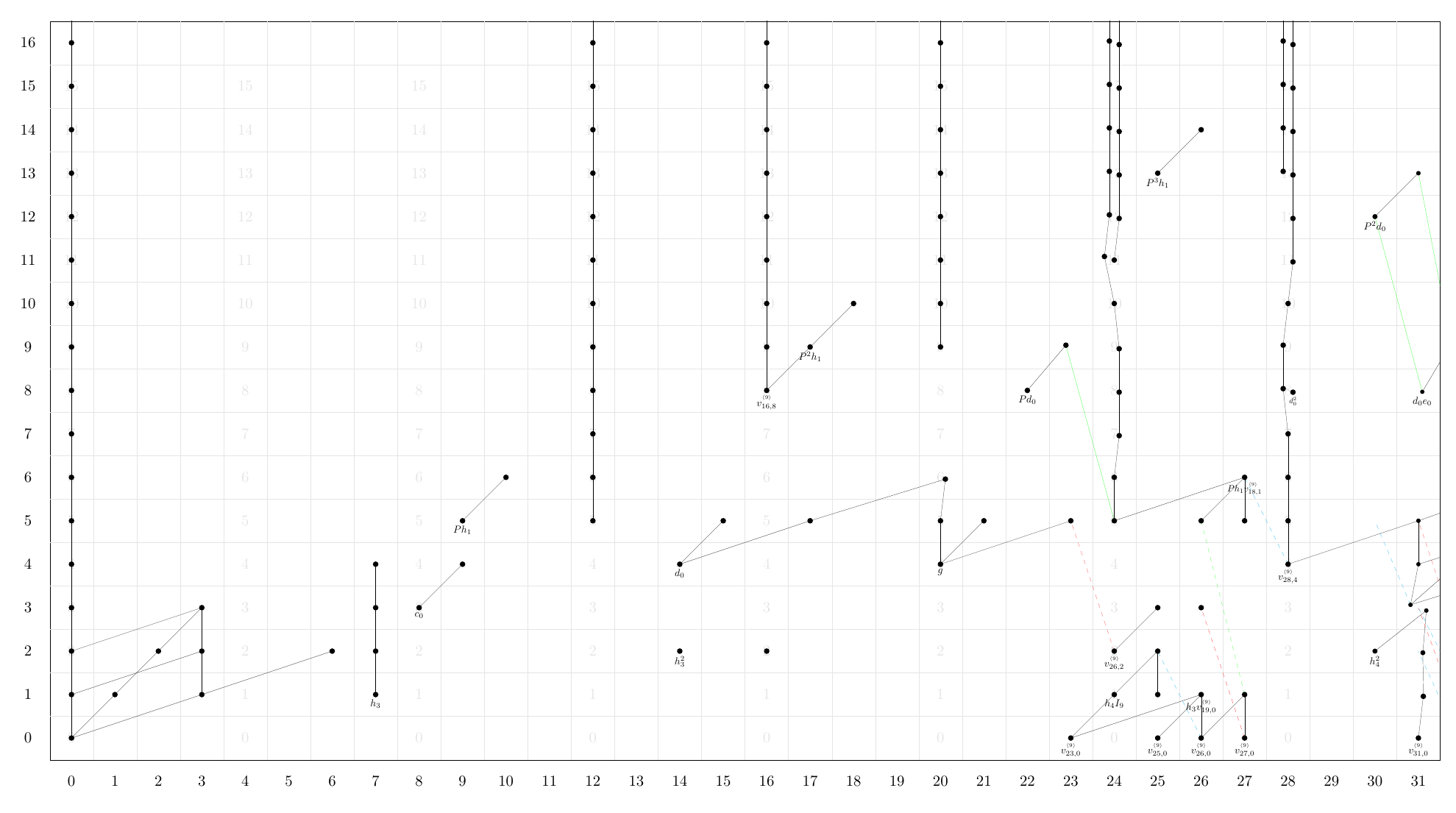}
  \caption{$E_4$-page $MO \langle 9 \rangle $ (stems 0-31)}
  \label{fig:mo9_e4_0_31}
\end{subfigure}

\caption{$E_3$/$E_4$-page $MO\langle 9 \rangle$ \label{fig:mo9_e4}}
\end{figure}
\begin{figure}[!p]
\centering
\begin{subfigure}{\linewidth}
  \centering
  \includegraphics[width=\linewidth,
    trim=0.5cm 0cm 0cm 0cm,clip]{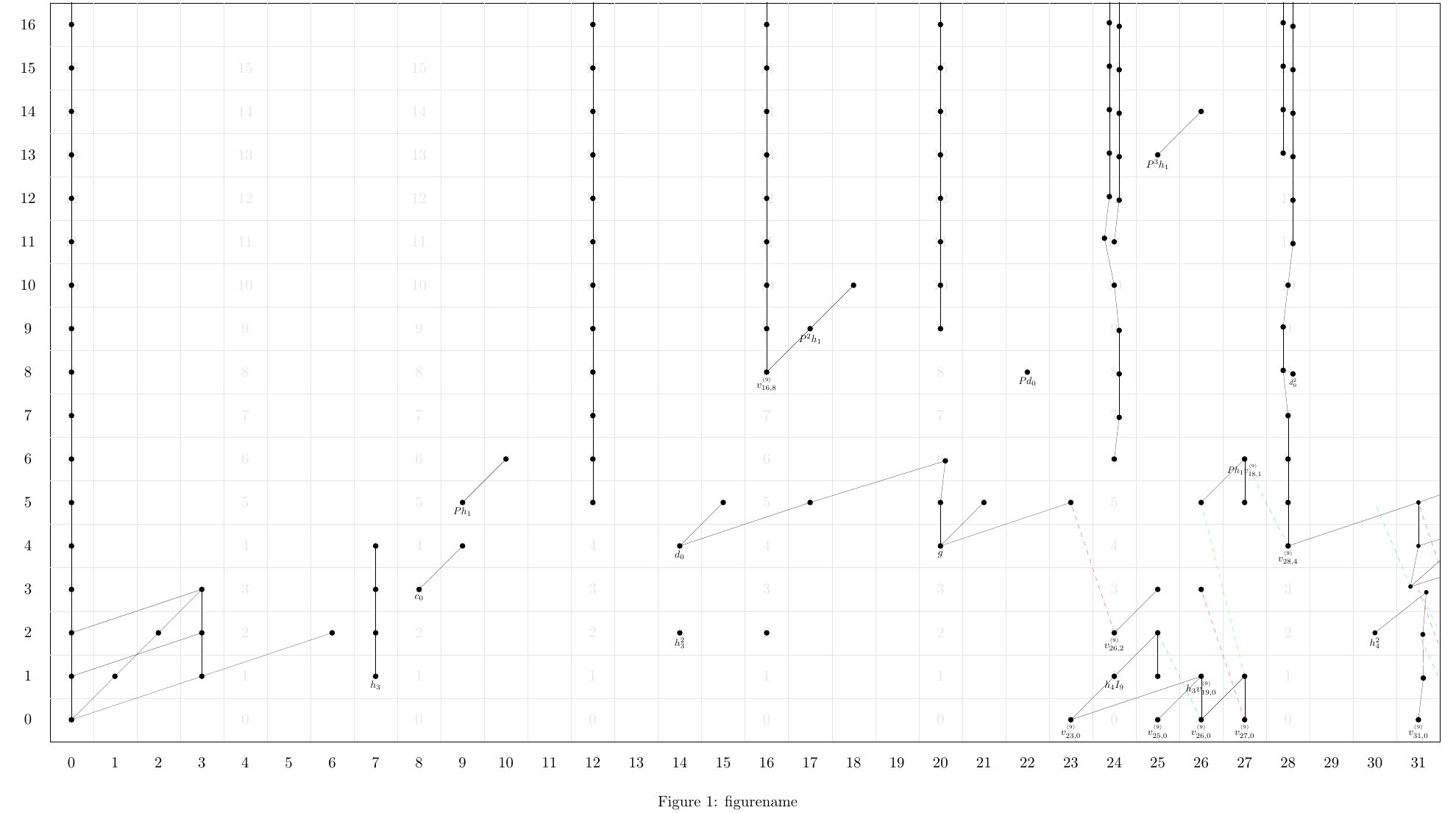}
  \caption{$E_{\infty}$-page $MO \langle 9 \rangle $}
  \label{fig:mo9_einf_0_31}
\end{subfigure}



\caption{$E_{\infty}$-page \(MO\langle 9\rangle\)\label{fig:mo9_einf}}

\end{figure}
\FloatBarrier

\subsection{Adams differentials}
In this section, we calculate nearly all Adams differentials through stem 31. We also calculate some differentials in higher stems that are either used for determining lower stem information, or are relevant in later chapters. The first few uses of Massey products include explicit calculations of any indeterminacy, while later uses omit such details. The proofs of differentials generated by Weinan Lin's \textbf{sseqcpp} program are indicated. 
\setlength{\tabcolsep}{15pt}
\subsection{$d_2$-differentials}
\begin{proposition}
Table \ref{tab:Adams d_2 mo9} describes the non-zero $d_2$-differentials in the Adams spectral sequence for $MO\langle 9 \rangle$  on all indecomposables on $E_2$ through stem 31 and on select elements in stems $> 31$.  
\end{proposition}
The following table lists the only indecomposables through stem 31 that, for degree reasons, can support a $d_2$-differential. It also includes select differentials in stems $>31$.
 \begin{longtable}{llllc}
    \caption[Possible $d_2$-differentials $MO\langle 9 \rangle$]{Possible non-zero $d_2$-differentials on indecomposable elements
    \label{tab:Adams d_2 mo9}
    } \\
    \toprule
    $x$ & $(t-s,s)$ & $d_2(x)$ & Occurs & Proof\\
    \midrule \endfirsthead
    \caption[]{Possible non-zero $d_2$-differentials on indecomposable elements} \\
    \toprule
    $x$ & $(t-s,s)$ & $d_2(x)$ & Occurs & Proof\\
    \midrule \endhead
    \bottomrule \endfoot
        $h_1$ & $(1, 1)$ & $h_0^3$ & No & \ref{prop: sphere differentials}\\
        $I_9$ & $(9,0)$ & $h_1h_3$ & Yes & \ref{lem:d_2 differential in stem 9 filtration 0 mo9}\\
        $h_4$ & $(15, 1)$ & $h_0h_3^2$ & Yes & \ref{prop: sphere differentials}\\
        $e_0$ & $(15, 4)$ & $h_1^2d_0$ & Yes & \ref{prop: sphere differentials}\\
        $\vv(18,1,9)$ & $(18, 1)$ & $h_1^2h_4$ & Yes & \ref{lem:d_2-unit-kernel-of-mo9-stem-17}\\
        $f_0$ & $(18, 4)$ & $h_0^2e_0$ & Yes & \ref{prop: sphere differentials}\\
        $\vv(19,0,9)$ & $(19, 0)$ & $h_2h_4$ & Yes & \ref{lem:d_2 differential in 19 stem mo9}\\
        $c_1$ & $(19, 3)$ & $Ph_1I_9$ & No & \ref{d_2 differential in 20 stem mo9}\\
        $\vv(20,1,9)$ & $(20, 1)$ & $c_1$ & Yes & \ref{d_2 differential in 20 stem mo9}\\
        $\vv(22,0,9)$ & $(22, 0)$ & $h_2\vv(18,1,9)$ & Yes & \ref{d_2 differential in 22 stem mo9}\\
        $i$ & $(23, 7)$ & $h_0Pd_0$ & Yes & \ref{prop: sphere differentials}\\
        $\vv(24,4,9)$ & $(24, 4)$ & $h_0h_2g$ & Yes & \ref{d_2 differential in stem 24 filtration 4 mo9}\\
        $\vv(24,7,9)$ & $(24, 7)$ & $h_1Pd_0+h_0^2i$ & Yes & \ref{d_2 differential in stem 24 filtration 7 mo9}\\
        $\vv(25,0,9)$ & $(25, 0)$ & $v_{24,2}$ & Yes & \ref{d_2 differential in stem 25 filtration 0 mo9}\\
        $\vv(26,0,9)$ & $(26, 0)$ & $h_1h_4I_9$ & ? &  \\
        $j$ & $(26, 7)$ & $h_0Pe_0$ & Yes & \ref{prop: sphere differentials}\\
        $\vv(27,4,9)$ & $(27, 4)$ & $h_2^2g$ & Yes & \ref{lem:d_2 differential in mo9 stem 27 filtration 4}\\
        $\vv(28,4,9)$ & $(28, 4)$ & $Ph_1\vv(18,1,9)$ & ? &  \\
        $k$ & $(29, 7)$ & $h_0d_0^2$ & Yes & \ref{prop: sphere differentials}\\
        $\vv(31,0,9)$ & $(31, 0)$ & $h_4^2$ & No & \ref{lem:d_2 differential in stem 31 filtration 0 mo9}  \\
        $\vv(31,1,9)$ & $(31, 1)$ & $\vv(30,3,9)$ & Yes & \ref{lem:d_2 differential in stem 31 filtration 1 mo9} \\
        $\vv(31,5,9)$ & $(31, 5)$ & $h_0\Delta h_2^2$ & Yes & \ref{lem:d_2 differential in stem 31 filtration 5 mo9} \\
        $\vv(32,3,9)$ & $(32, 3)$ & $n$ & Yes & \ref{lem:d_2 differential in stem 32 filtration 3 mo9}\\
         $l$ & $(32, 7)$ & $h_0d_0e_0$ & Yes & \ref{prop: sphere differentials}\\
    \end{longtable}
\begin{proof}
For proofs of these differentials see the following series of Lemmas. 
\end{proof}

\begin{lemma}\label{lem:d_2 differential in stem 9 filtration 0 mo9}
$d_2(I_9)=h_1h_3$.
\end{lemma}
\begin{proof}
The element $h_1h_3$ detects $\eta\sigma \in \pi_{8}\mathbb{S}$. The class $\eta\sigma$ is in $\text{Im } J$ (see Lemma 11.46 and Theorem 11.61 of Bruner-Rognes \cite{BrunerRognestmf}.) By Theorem \ref{hoveyimj}, $h_1h_3$ must not survive the Adams spectral sequence for $MO\langle 9 \rangle$. For degree reasons, $I_9$ is the only possible source of a differential with target $h_1h_3$.
Thus, $d_2(I_9)=h_1h_3$.
\end{proof}
\begin{lemma}\label{lem:d_2-unit-kernel-of-mo9-stem-17}
$d_2(\vv(18,1,9))=h_1^2h_4$.
\end{lemma}
\begin{proof}
The element $\eta\eta_{4} \in \pi_{17}\mathbb{S}$ is detected by $h_1^2h_4 \in Ext_{\mathcal{A}}(\mathbb{Z}/2,\mathbb{Z}/2)$. By Theorem \ref{lem:mo9_17_kernel}, $\eta\eta_{4}$ maps to zero in homotopy, and so $h_1^2h_4$ must be killed in Adams spectral sequence for $MO\langle 9 \rangle$. For degree reasons, $\vv(18,1,9)$ is the only element that may support a differential with target $h_1^2h_4$. 
\end{proof}

\begin{lemma}\label{lem:d_2 differential in 19 stem mo9}
$d_2(\vv(19,0,9))=h_2h_4$.
\end{lemma}
\begin{proof}
 Consider the relation $h_0^2\vv(19,0,9)=h_1\vv(18,1,9)$. It follows from Lemma \ref{lem:d_2-unit-kernel-of-mo9-stem-17} that $d_2{h_1\vv(18,1,9)}=h_1^3h_4=h_0^2h_2h_4$. Since $d_2$-differentials are $h_0$-linear, this implies $d_2(\vv(19,0,9))=h_2h_4$
\end{proof}
\begin{lemma}\label{d_2 differential in 20 stem mo9}
$d_2(\vv(20,1,9))=c_1$.
\end{lemma}
\begin{proof}
The element $\vv(20,1,9)$ is contained in the Massey product
$\langle h_2,h_3,I_9\rangle$ with zero indeterminacy, since
\[
h_2\cdot \text{Ext}^{0,17}(MO\langle 9\rangle) + \text{Ext}^{0,13}(MO\langle 9\rangle)\cdot h_3=\{0\}.
\]
By Moss's higher Leibniz rule \ref{mossleibniz},
\begin{align*}
d_2(\vv(20,1,9)) \in\;&
\langle d_2(h_2),h_3,I_9\rangle+\langle h_2,d_2(h_3),I_9\rangle+\langle h_2,h_3,d_2(I_9)\rangle\\
=\;&\langle 0,h_3,I_9\rangle+\langle h_2,0,I_9\rangle+\langle h_2,h_3,h_1h_3\rangle.
\end{align*}
The first two brackets vanish, and their indeterminacies are zero:
\[
0 \cdot \text{Ext}^{0,17}(MO\langle 9\rangle)+\text{Ext}^{2,14}(MO\langle 9\rangle)\cdot h_3=\{0\},
\qquad
h_2 \cdot \text{Ext}^{2,18}(MO\langle 9\rangle)+\text{Ext}^{0,13}(MO\langle 9\rangle)\cdot 0=\{0\}.
\]
For the third term, $\langle h_2,h_3,h_1h_3\rangle$ contains $c_1$ with indeterminacy:
\[
h_2 \cdot \text{Ext}^{2,19}(MO\langle 9\rangle)+\text{Ext}^{1,12}(MO\langle 9\rangle)\cdot h_1h_3=\{0\}.
\]
Hence the only possible value is $d_2(\vv(20,1,9))=c_1$.
\end{proof}

\begin{lemma}\label{d_2 differential in 22 stem mo9}
$d_2(\vv(22,0,9))=h_2\vv(18,1,9)$.
\end{lemma}
\begin{proof}
The products $h_2^2h_4$ and $h_2\vv(19,0,9)$ are nonzero on the $E_2$-page. By Lemma \ref{lem:d_2 differential in 19 stem mo9} and the Leibniz rule
\[
d_2(h_2\vv(19,0,9))=d_2(h_2)\vv(19,0,9)+h_2d_2(\vv(19,0,9))=h_2^2h_4
\]
There are relations $h_0\vv(22,0,9)=h_2\vv(19,0,9)$ and $h_0h_2\vv(18,1,9)=h_2^2h_4$. Since the $d_2$-differential is $h_0$-linear, this implies $d_2(\vv(22,0,9))=h_2\vv(18,1,9)$.
\end{proof}
\begin{lemma}\label{d_2 differential in stem 24 filtration 4 mo9}
$d_2(\vv(24,4,9))=h_0h_2g$.
\end{lemma}
\begin{proof}
We will first calculate the $d_2$-differential on the element $d_0\vv(24,4,9)$ in bidegree $(38,8)$. The Massey product $\langle h_1,h_1c_0,\vv(27,4,9) \rangle$ contains $d_0\vv(24,4,9)$ with zero indeterminacy. By Moss's higher Leibniz rule \ref{mossleibniz} and Lemma \ref{lem:d_2 differential in mo9 stem 27 filtration 4},
\begin{align*}
d_2(d_0\vv(24,4,9)) \in\;&
\langle d_2(h_1),h_1c_0,\vv(27,4,9)\rangle+\langle h_1,d_2(h_1c_0),\vv(27,4,9) \rangle+\langle h_1,h_1c_0,d_2(\vv(27,4,9)) \rangle\\
=\;&\langle 0,h_1c_0,\vv(27,4,9) \rangle+\langle h_1,0,\vv(27,4,9) \rangle+\langle h_1,h_1c_0,h_2^2g \rangle.
\end{align*}
The first two terms vanish with zero indeterminacy and the last term contains $h_0^5x$ with zero indeterminacy. Hence, the only possibility is $d_2(d_0\vv(24,4,9))=h_0^5x$. 
By the Leibniz rule, $d_2(d_0\vv(24,4,9))=h_0^5x=d_0d_2(\vv(24,4,9))$. Consider the relation $d_0h_0h_2g=h_0^5x$. Since $d_0$ is a permanent cycle and $h_0h_2g$ is the only element in bidegree $(23,6)$ of the $E_2$-page, this implies $d_2(\vv(24,4,9))=h_0h_2g$. 
\end{proof}
\begin{lemma}\label{d_2 differential in stem 24 filtration 7 mo9}
$d_2(\vv(24,7,9))=h_1Pd_0 + h_0^2i$.
\end{lemma}
\begin{proof}
This proof was generated by \textbf{sseqcpp}. The element $d_2(\vv(24,7,9))$ must support a $d_2$-differential since $h_1^2\vv(24,7,9)=h_0^2j$ and $d_2(h_0^2j)=h_0^3Pe_0=h_1^3Pd_0$. Consider the relations $d_0\vv(24,7,9)=0$, $d_0h_1Pd_0=h_1h_0^2i$. Suppose that either $d_2(\vv(24,7,9))=h_1Pd_0$ or $d_2(\vv(24,7,9))=h_0^2i$. By the Leibniz rule: 
\begin{align*}
d_2(d_0\vv(24,7,9))&= d_2(0)\\
                &= 0 \\
                &= d_0d_2(\vv(24,7,9)) \\
                &= h_1d_0Pd_0
\end{align*}
This is a contradiction since $h_1d_0Pd_0$ is nonzero on the $E_2$-page. Thus, the only possibility is $d_2(\vv(24,7,9))=h_1Pd_0 + h_0^2i$. 
\end{proof}
\begin{lemma}\label{d_2 differential in stem 25 filtration 0 mo9}
$d_2(\vv(25,0,9))=0$.
\end{lemma}
\begin{proof}
The only possible target of a $d_2$ differential on $\vv(25,0,9)$ is $v_{24,2}$. Suppose $d_2(\vv(25,0,9))=v_{24,2}$. Consider the relation $h_1\vv(25,0,9)=h_3\vv(19,0,9)$. By the Leibniz rule, 
\begin{align*}
d_2(h_1\vv(25,0,9))&=h_1d_2(\vv(25,0,9)) \\
                &=h_1v_{24,2} \\
                &=h_3d_2(\vv(19,0,9))\\
                &=h_3h_2h_4 = 0
\end{align*}
But $h_1v_{24,2} \neq 0$, so this is a contradiction. Hence, $d_2(\vv(25,0,9))=0$. 
\end{proof}
\begin{lemma}\label{lem:d_2 differential in mo9 stem 27 filtration 4}
        $d_2(\vv(27,4,9))=h_2^2g$.
\end{lemma}
\begin{proof}
The element $\vv(27,4,9)$ is contained in $\langle h_3, h_0^4,\vv(19,0,9)\rangle$
 with zero indeterminacy, since:
 \[
h_3\cdot \text{Ext}^{3,23}(MO\langle 9\rangle) + \text{Ext}^{0,27}(MO\langle 9\rangle)\cdot h_0^4=\{0\}.
\]
 By Moss' higher Leibniz rule \ref{mossleibniz},
\begin{align*}
d_2(\vv(27,4,9)) \in\;&\langle d_2(h_3), h_0^4,\vv(19,0,9) \rangle + \langle h_3, d_2(h_0^4),\vv(19,0,9) \rangle + \langle h_2, h_0^4,h_2h_4 \rangle\\
=\;&\langle 0, d_0,I_9 \rangle + \langle h_2, 0,I_9 \rangle + \langle h_2, d_0,h_1h_3 \rangle
\end{align*}
The first two terms vanish with no indeterminacy, and $\langle h_3, h_0^4,h_2h_4 \rangle = h_2^2g$ with indeterminacy:
\[
h_3\cdot \text{Ext}^{5,24} + \text{Ext}^{4,9} \cdot h_2h_4 = \{0\}
 \]
Hence the only possible values is $d_2(\vv(27,4,9))=h_2^2g$. 
\end{proof}
\begin{lemma}\label{lem:d_2 differential in stem 31 filtration 0 mo9}
$d_2(\vv(31,0,9))=0$.
\end{lemma}
\begin{proof}
The only possible target is $h_4^2$, however $h_1h_4^2$ is nonzero while $h_1\vv(31,0,9)=0$. Since the $d_2$-differential is $h_1$-linear, $d_2(\vv(31,0,9))$ must be zero. 
\end{proof}
\begin{lemma}\label{lem:d_2 differential in stem 31 filtration 1 mo9}
$d_2(\vv(31,1,9))=\vv(30,3,9)$.
\end{lemma}
\begin{proof}
Suppose $d_2(\vv(31,1,9))=0$. The product of $\vv(31,1,9)$ and $f_0$ is zero. Applying the Leibniz rule:
\begin{align*}
d_2(\vv(31,1,9)f_0)&=d_2(\vv(31,1,9))f_0+\vv(31,1,9)d_2(f_0) \\
                &=\vv(31,1,9)h_0^2e_0
\end{align*}
The lefthand side is not zero on the $E_2$-page and the righthand side equals $d_2(0)=0$, a contradiction. There is only one element in bidegree $(30,3)$, $\vv(30,3,9)$. Hence, $d_2(\vv(31,1,9))=\vv(30,3,9)$. 
\end{proof}
\begin{lemma}\label{lem:d_2 differential in stem 31 filtration 5 mo9}
$d_2(\vv(31,5,9))=h_0\Delta h_2^2$.
\end{lemma}
\begin{proof}
Consider the relation $h_0^3\vv(31,5,9)=h_3\vv(24,7,9)$. By the Leibniz rule, 
\begin{align*}
d_2(h_3\vv(24,7,9))&=h_3d_2(\vv(24,7,9))\\
                &=h_3(h_1Pd_0+h_0^2i) \\
                &= h_0^4\Delta h_2^2
\end{align*}
Since $h_0$ is a permanent cycle, this implies that $d_2(\vv(31,5,9))=h_0\Delta h_2^2$.
\end{proof}

\begin{lemma}\label{lem:d_2 differential in stem 32 filtration 3 mo9}
$d_2(\vv(32,3,9))=n$
\end{lemma}
\begin{proof}
This proof was generated by \textbf{sseqcpp}. The differential $d_2(\vv(32,3,9))$ is one of the eight linear combination of the classes $\vv(31,5,9)$, $h_1\vv(28,4,9)$, and $n$. 
The linear combinations $\vv(31,5,9)$, $\vv(31,5,9)+n$, $\vv(31,5,9) + \vv(28,4,9)$, $\vv(31,5,9) + \vv(28,4,9)+n$ are ruled out by $h_0$-linearity 
of the $d_2$-differential. 

\smallskip
\noindent Suppose $d_2(\vv(32,3,9))=0$.
Using the Leibniz rule with $h_3$ and $d_2(h_3)=0$, we have
\[
d_2(g\vv(19,0,9))=0
\]
But $g\vv(19,0,9)$ already supports a nonzero $d_2$, a contradiction.

\smallskip
\noindent Suppose $d_2(\vv(32,3,9))=h_2\vv(28,4,9)$.
Again using $h_3$ with $d_2(h_3)=0$, the Leibniz rule forces
\[
d_2(g\vv(19,0,9))=0,
\]
But $g\vv(19,0,9)$ already supports a nonzero $d_2$, a contradiction.

\smallskip
\noindent Suppose $d_2(\vv(32,3,9))=h_2\vv(28,4,9)+n$.
Using the Leibniz rule with $P h_2$ and $d_2(P h_2)=0$, we have
\[
d_2(0)=Pc_0\vv(26,3,9)
\]
A contradiction.
\smallskip
The only remaining possibility is $d_2(\vv(32,3,9))=n$, as claimed.
\end{proof}
\clearpage 
The $d_2$-differentials above have the following consequence. 
\begin{theorem}
The spectrum $MO \langle 9 \rangle$ is indecomposable through degree 22.
\end{theorem}
\begin{proof}
This follows from the nonzero $d_2$-differentials on $I_{9}$, $\vv(19,0,9)$, $\vv(22,0,9)$.
\end{proof}
\subsection{$d_3$-differentials}
\begin{proposition}
Table \ref{tab:Adams d_3 mo9} describes the non-zero $d_3$-differentials in the Adams spectral sequence for $MO\langle 9 \rangle$  on all indecomposables on $E_3$ through stem 31, and on select elements in stems $> 31$.
\end{proposition}
The following table lists the only indecomposables through stem 31 that, for degree reasons, can support a $d_3$-differential. It also includes select differentials in stems $>31$.
 \begin{longtable}{llllc}
    \caption[Possible $d_3$-differentials $MO\langle 9 \rangle$]{Possible non-zero $d_3$-differentials on indecomposable elements
    \label{tab:Adams d_3 mo9}
    } \\
    \toprule
    $x$ & $(t-s,s)$ & $d_3(x)$ & Occurs & Proof\\
    \midrule \endfirsthead
    \caption[]{Possible non-zero $d_3$-differentials on indecomposable elements} \\
    \toprule
    $x$ & $(t-s,s)$ & $d_3(x)$ & Occurs & Proof\\
    \midrule \endhead
    \bottomrule \endfoot
        $\vv(12,2,9)$ & $(12, 2)$ & $Ph_2$ & Yes & \ref{lem:d_3 differential in stem 12 filtration 2 mo9}\\
        $h_0h_4$ & $(15,2)$ & $h_0d_0$ & Yes & \ref{prop: sphere differentials}\\
        $h_1h_4$ & $(16,2)$ & $h_1d_0$ & No & \ref{prop: sphere differentials}\\
        $v_{17,4}+e_0$ & $(17, 4)$ & $Pc_0$ & Yes & \ref{lem:d_3-differentials-imj-mo9}\\
        $\vv(20,6,9)$ & $(20, 6)$ & $P^2h_2$ & Yes & \ref{lem:d_3-differentials-imj-mo9}\\
        $v_{24,2}$ & $(24, 2)$ & $h_2g$ & ? & \\
        $h_1\vv(24,7,9)+Pe_0$ & $(25, 8)$ & $P^2c_0$ & Yes & \ref{lem:d_3 differential in stem 25 filtration 8 mo9}\\
        $\vv(28,10,9)$ & $(28, 10)$ & $P^3h_2$ & Yes & \ref{lem:d_3 differential in stem 28 filtration 10 mo9}\\
        $\Delta h_2^2$ & $(30,6)$ & $h_0^2k$ & Yes & \ref{prop: sphere differentials}\\
    \end{longtable}
    \begin{proof}
        For proofs of these differentials, see the following series of Lemmas. 
    \end{proof}

\begin{lemma}\label{lem:d_3-differentials-imj-mo9}
\begin{enumerate}
\item $d_3(v_{17,4}+e_0)=Pc_0$.
\item $d_3(\vv(20,6,9))=P^2h_2$.
\end{enumerate}
\end{lemma}
\begin{proof}
By Theorem \ref{mahowaldimj}, the target of each of these differentials is in $\text{Im }J$. Since $\text{Im } J$ is contained in the kernel of the unit map $\mathbb{S} \rightarrow MO\langle n \rangle$, these elements must not survive the Adams spectral sequence for $MO\langle 9 \rangle$. 
For (1), the other possible sources of a differential with target $Pc_0$ are $d_3(e_0)$ and $d_4(h_1^2h_3)$, but the former already supports a $d_2$-differential and the ladder is a target of a $d_2$-differential (see Table \ref{tab:Adams d_2 mo9}). For (2), the only other possible sources of a differential with target $P^2h_2$ is $g$, but $d_5(g)$ is zero by comparing with the sphere spectrum $\mathbb{S}$ and using naturality. Thus, $d_3(v_{17,4}+e_0)=Pc_0$ and $d_3(\vv(20,6,9))=P^2h_2$.
\end{proof}
\begin{lemma}\label{lem:d_3 differential in stem 12 filtration 2 mo9}
$d_3(\vv(12,2,9))= Ph_2$.
\end{lemma}
\begin{proof}
There are relations $h_2\vv(12,2,9)=h_0^2h_4$ and $h_2Ph_2=h_0^2d_0$. By Proposition \ref{prop: sphere differentials}, $d_3(h_0^2h_4)=h_0^2d_0$. Since the $d_3$-differential is $h_2$-linear, this implies $d_3(\vv(12,2,9))=Ph_2$. 
\end{proof}
\begin{lemma}\label{lem:d_3 differential in stem 25 filtration 8 mo9}
$d_3(h_1\vv(24,7,9)+Pe_0)= P^2c_0$.
\end{lemma}
\begin{proof}
This proof was generated by \textbf{sseqcpp}. The $d_3$-differential on the product $h_1(\vv(24,7,9)+Pe_0)$ will be calculated using the relations: 
\begin{align*}
h_1(\vv(24,7,9)+Pe_0)&=P^2h_1I_9+h_0^2j \\
Ph_1(v_{17,4}+e_0)&=P^2h_1I_9+h_0^2j  \\
Ph_1 Pc_0 = h_1P^2c_0 
\end{align*}
Applying the relations above, the Leibniz rule, and Lemma \ref{lem:d_3-differentials-imj-mo9},
\begin{align*}
d_3(h_1(\vv(24,7,9)+Pe_0))&= d_3(P^2h_1I_9+h_0^2j)\\
                       &=d_3(Ph_1(v_{17,4}+e_0)) \\
                       &=Ph_1d_3(v_{17,4}+e_0) \\
                       &=Ph_1Pc_0 \\
                       &=h_1P^2c_0 
\end{align*}
Since the $d_3$-differential is $h_1$-linear, this implies $d_3(h_1\vv(24,7,9)+Pe_0)= P^2c_0$. 
\end{proof}
\begin{lemma}\label{lem:d_3 differential in stem 28 filtration 10 mo9}
$d_3(\vv(28,10,9))=P^3h_2$.
\end{lemma}
\begin{proof}
The Massey product $\langle h_0, h_0^3h_3, \vv(20,6,9) \rangle$ contains the element $\vv(28,10,9)$  with indeterminacy $\{h_0^6\vv(28,4,9)\}$. By Moss's higher Leibniz rule \ref{mossleibniz}:
\begin{align*}
d_3(\langle h_0, h_0^3h_3, \vv(20,6,9) \rangle) \in\;&
\langle d_3(h_0),h_0^3h_3,\vv(20,6,9)\rangle+\langle h_0,d_3(h_0^3h_3),\vv(20,6,9)\rangle+\langle h_0,h_0^3h_3,d_3(\vv(20,6,9))\rangle\\
=\;&\langle 0,h_3,\vv(20,6,9)\rangle+\langle h_0,0,\vv(20,6,9)\rangle+\langle h_0,h_0^3h_3,P^2h_2\rangle.
\end{align*}
The first two terms vanish with zero indeterminacy and the third term contains $P^3h_2$ with zero indeterminacy. Hence
\begin{align*}
        d_3(\langle h_0, h_0^3h_3, \vv(20,6,9) \rangle)=d_3(\vv(28,10,9))+d_3(\langle h_0^6\vv(28,4,9)\rangle) \subset \{P^3h_2\}
\end{align*}
By sparsity, $d_3(h_0^5\vv(28,4,9))=0$. Since the $d_3$-differential is $h_0$-linear, this implies $d_3(h_0^6\vv(28,4,9))=0$. Thus, we conclude that $d_3(\vv(28,10,9))=P^3h_2$. 
\end{proof}
\subsection{$d_4$-differentials}
\begin{proposition}
Table \ref{tab:Adams d_4 mo9} describes the non-zero $d_4$-differentials in the Adams spectral sequence for $MO\langle 9 \rangle$  on all indecomposables on $E_4$ through stem $31$, and on select elements in stems $>31$. 
\end{proposition}
The following table lists the only indecomposables through stem 31 that, for degree reasons, can support a $d_4$-differential. It also includes select differentials in stems $>31$
 \begin{longtable}{llllc}
    \caption[Possible $d_4$-differentials $MO\langle 9 \rangle$]{Possible non-zero $d_4$-differentials on indecomposable elements.
    \label{tab:Adams d_4 mo9}
    } \\
    \toprule
    $x$ & $(t-s,s)$ & $d_4(x)$ & Occurs & Proof\\
    \midrule \endfirsthead
    \caption[]{Possible non-zero $d_4$-differentials on indecomposable elements} \\
    \toprule
    $x$ & $(t-s,s)$ & $d_4(x)$ & Occurs & Proof\\
    \midrule \endhead
    \bottomrule \endfoot
        $h_0e_0$ & $(17, 5)$ & $h_0v_{16,8}$ & No & \ref{prop: sphere differentials}\\
        $h_0\vv(24,4,9)$ & $(24,5)$ & $h_0^2i$ & Yes & \ref{lem:d_4 differential in stem 24 filtration 5 mo9} \\
        $h_1v_{24,2}$ & $(25,3)$ & $h_0^3\vv(24,4,9)$ & No & \ref{lem:permanent cycle in stem 25 filtration 3 mo9} \\
        $h_1\vv(26,0,9)$ & $(27,1)$ & $v_{26,5}$ & ? &  \\
        $d_0e_0$ & $(31,8)$ & $P^2d_0$ & Yes & \ref{prop: sphere differentials} \\
        $h_0^3 \Delta h_3^2$ & $(38,9)$ & $h_0^2d_0i$ & Yes & \ref{prop: sphere differentials} \\
    \end{longtable}
\begin{proof}
For proofs of these differentials see the following series of Lemmas. 
\end{proof}
\begin{lemma}\label{lem:d_4 differential in stem 24 filtration 5 mo9}
$d_4(h_0\vv(24,4,9))=h_0^2i$. 
\end{lemma}
\begin{proof}
This proof was generated by $\textbf{sseqcpp}$. Suppose $h_0\vv(24,4,9)$ is a permanent cycle. Since $d_0h_0\vv(24,4,9)=h_0^3 \Delta h_3^2$ where $d_0$ is a permanent cycle, this implies that $h_0^3 \Delta h_3^2$ is a permanent cycle. This is a contradiction because $h_0^3 \Delta h_3^2$ supports a $d_4$ differential by Proposition \ref{prop: sphere differentials}. 
Thus, $h_0\vv(24,4,9)$ must support a differential. The only possible target is $h_0^2i$, and so $d_4(h_0\vv(24,4,9))=h_0^2i$.
\end{proof}
\begin{lemma}\label{lem:permanent cycle in stem 25 filtration 3 mo9}
The element $h_1v_{24,2}$ is a permanent cycle. 
\end{lemma}
\begin{proof}
All elements in stem 24 and filtration greater than 3 of the $E_4$-page are $h_0$-torsion free while $h_1\vv(24,4,9)$ is $h_0$-torsion. 
\end{proof}
\begin{lemma}
The element $\vv(23,0,9)$ is a permanent cycle. 
\end{lemma}
\begin{proof}
The only possible differential with source $\vv(23,0,9)$ is a $d_8$-differential with target $Pd_0$. To rule out this possibility, consider the composition of maps:
\begin{align*}
\pi_{n} \mathbb{S} \rightarrow \pi_{n} MO \langle 9 \rangle  \rightarrow \pi_{n} MO \langle 8 \rangle \rightarrow \pi_{n} tmf 
\end{align*}
According to Theorem 11.80 of Bruner-Rognes \cite{BrunerRognestmf}, $\eta^2\overline{\kappa}\in \pi_{22} \mathbb{S}$ and $\eta\mu_{25} \in  \pi_{26} \mathbb{S}$ have nonzero images under this composition. The element $Pd_0 \in \text{Ext}^{8,30}_{\mathbb{A}}(\mathbb{Z}/2,\mathbb{Z}/2)$ detects $\eta^2\overline{\kappa}$ and $h_1P^3h_1 \in \text{Ext}^{14,40}_{\mathbb{A}}(\mathbb{Z}/2,\mathbb{Z}/2)$ detects $\eta\mu_{25}$.
Thus, if $Pd_0$ and $h_1P^3h_1$ were boundaries in the Adams spectral sequence for $MO\langle 9 \rangle$, $\eta^2\overline{\kappa}$ and $\eta\mu_{25}$ would map to zero in $\pi_{*} tmf$ under the above composition. Hence, $d_8(\vv(23,0,9))=0$ and $\vv(23,0,9)$ is a permanent cycle. 
\end{proof}
\begin{lemma}
The element $h_4I_9$ is a permanent cycle.
\end{lemma}
\begin{proof}
Consider the relation $h_1\vv(23,0,9)=h_4I_9$. Since $h_1$ and $\vv(23,0,9)$ are permanent cycles, $h_4I_9$ is also a permanent cycle. 
\end{proof}

\subsection{The abutment}
The $E_\infty$-page is the associated graded of the Adams filtration for $MO\langle 9 \rangle$. To obtain the  group structures, we must determine if there are any hidden extensions. The first hidden $2$-extension in the Adams spectral sequence for $\mathbb{S}$ occurs in stem 40 (see Figure 11.14 Bruner-Rognes \cite{BrunerRognestmf}), and so there are no hidden $2$-extensions between elements that are in the image 
of the map of Adams spectral sequences induced by the unit map $\mathbb{S} \rightarrow MO \langle 9 \rangle$. For a reference for claims about products in $\pi_{*} \mathbb{S}$ that are used in the following Propositions, see Theorem 11.61 of Bruner-Rognes \cite{BrunerRognestmf}.
\begin{proposition}\label{prop:extension problem stem 16 mo9}
$\pi_{16}MO \langle 9 \rangle \cong \Z \oplus \Z/2$
\end{proposition}
\begin{proof}
Note that $h_1h_4$ detects $\eta^{*} \in \pi_{16} \mathbb{S}$. The group $\pi_{16}MO \langle 9 \rangle$ lies in the short exact sequence
\begin{align*}
0 \rightarrow \mathbb{Z}\rightarrow \pi_{16}MO \langle 9 \rangle \rightarrow  \mathbb{Z}/2\{\eta^{*}\} \rightarrow 0 
\end{align*}
Suppose this short exact sequence does not split and $\pi_{16}MO \langle 9 \rangle \cong \mathbb{Z}$. This would imply $2\eta^{*} \neq 0$, which is a contradiction since $2\eta^{*}=0\in \pi_{16} \mathbb{S}$.
\end{proof}

\begin{proposition}\label{prop:extension problem stem 20 mo9}
$\pi_{20}MO \langle 9 \rangle \cong \Z \oplus \Z/8$
\end{proposition}
\begin{proof}
The element $g$ detects the class $\overline{\kappa} \in \pi_{20} \mathbb{S}$. The group $\pi_{20}MO \langle 9 \rangle$ lies in the short exact sequence
\begin{align*}
0 \rightarrow \mathbb{Z}\rightarrow \pi_{20}MO \langle 9 \rangle \rightarrow  \mathbb{Z}/8\{\overline{\kappa}\} \rightarrow 0 
\end{align*}
Suppose this extension does not split and $\pi_{20}MO \langle 9 \rangle \cong \mathbb{Z}$. This would imply $8\overline{\kappa} \neq 0$, which is a contradiction since $8\overline{\kappa}=0\in \pi_{20} \mathbb{S}$.  
\end{proof}
We next consider degree 28. While we technically do not have stem 28 of the $E_\infty$-page completely calculated, the ambiguity ($d_2(\vv(28,4,9))=?$) does not impact the group $\pi_{28}(MO\langle 9 \rangle)$ since $\vv(28,4,9)$ supports an $h_0$-tower. The possible hidden $2$-extension also does not involve $\vv(28,4,9)$.

\begin{proposition}\label{prop:extension problem stem 28 mo9}
$\pi_{28} MO\langle 9 \rangle \cong \Z \oplus \Z \oplus \Z/2$
\end{proposition}
\begin{proof}
The element $d_0^2$ detects $\kappa^{2} \in \pi_{28} \mathbb{S}$. The group $\pi_{28}MO \langle 9 \rangle$ lies in the short exact sequence
\begin{align*}
0 \rightarrow \mathbb{Z} \oplus \mathbb{Z} \rightarrow \pi_{28}MO \langle 9 \rangle \rightarrow \mathbb{Z}/2\{\overline{\kappa}^2\} \rightarrow 0 
\end{align*}
Suppose this extension does not split and $\pi_{28}MO \langle 9 \rangle \cong \Z \oplus \Z$. This would imply $2\kappa^{2} \neq 0$, which is a contradiction since $2\kappa^{2}=0\in \pi_{28} \mathbb{S}$.  
\end{proof}

\begin{proposition}\label{prop:pi_24 mo9}
$\pi_{24} MO \langle 9 \rangle \cong \Z \oplus \Z \oplus (\Z/2)^2 \text{ or } \Z \oplus \Z \oplus \Z/2$
\end{proposition}
\begin{proof}
The ``or'' statement depends on whether the differential $d_3(\vv(24,2,9))=h_2g$ or $0$. \\

Case 1: Suppose $d_3(\vv(24,2,9))=h_2g$. Stem 24 of the $E_{\infty}$-page contains $h_4I_9$ in filtration $1$, $h_0^2\vv(24,4,9)$, $h_0^4\vv(24,7,9)$, where the latter two generate $h_0$-towers. 
There is a possible hidden $2$-extension involving the element $h_4I_{9}$. But there is the relation $h_4I_{9} = h_1\vv(23,0,9)$. Thus, $h_4I_{9}$ detects an class in homotopy that is an $\eta$-multiple. Since $2\eta=0$, this rules out a hidden $2$-extension. Thus, $\pi_{24} MO \langle 9 \rangle \cong \Z \oplus \Z \oplus \Z/2$. \\

Case 2: Suppose $d_3(\vv(24,2,9))=0$. Stem 24 of the $E_{\infty}$-page is then identical to Case 1 with the addition of the element $\vv(24,2,9)$. As before, there is no possible hidden $2$-extension involving $h_4I_{9}$, but there is possibly a hidden $2$-extension involving $\vv(24,2,9)$. 
If such a $2$-extension exists, then  $\pi_{24} MO \langle 9 \rangle \cong \Z \oplus \Z \oplus \Z/2$. If not,  $\pi_{24} MO \langle 9 \rangle \cong \Z \oplus \Z \oplus (\Z/2)^2$
\end{proof}

\begin{theorem}\label{thm: fivebrane groups}
Table \ref{tab:fivebrane groups} describes the 2-primary component of the Fivebrane bordism groups
$\Omega^{\mathrm{Fivebrane}}_n$ for the values of $n$ listed.
\end{theorem}
\begin{center}
\begin{minipage}{\textwidth}
\captionof{table}[The 2-primary component of the Fivebrane bordism groups]{The 2-primary component of $\Omega^{\mathrm{Fivebrane}}_n$ for the values of $n$ listed.}
\label{tab:fivebrane groups}
\tablehead{\hline%
$n$ & $\Omega^{\mathrm{Fivebrane}}_n$ & $n$ & $\Omega^{\mathrm{Fivebrane}}_n$\\%
\hline}
\tabletail{\hline%
\multicolumn{4}{r}{%
\small\slshape to be continued on the next page}\\}
\tablelasttail{\hline}
\begin{supertabular}{|c|p{5.2cm}|c|p{5.2cm}|}
$8$  & $\Z/2$ & $18$ & $\Z/2$\\
$9$  & $\Z/2 \oplus \Z/2$ & $19$ & $(0)$\\
$10$ & $\Z/2$ & $20$ & $\Z \oplus \Z/8$\\
$11$ & $(0)$ & $21$ & $\Z/2$\\
$12$ & $\Z$ & $22$ & $\Z/2$\\
$13$ & $(0)$ & $23$ & $|\Omega_{23}^{\mathrm{Fivebrane}}| = 2 \text{ or } 4$\\
$14$ & $\Z/2 \oplus \Z/2$ & $24$ & $\Z \oplus \Z \oplus (\Z/2)^2 \text{ or } \Z \oplus \Z \oplus \Z/2$\\
$15$ & $\Z/2$ & $28$ & $\Z \oplus \Z \oplus \Z/2$\\
$16$ & $\Z \oplus \Z/2$ & $30$ & $\Z/2$\\
$17$ & $\Z/2 \oplus \Z/2$ & & \\
\end{supertabular}
\end{minipage}
\end{center}
\begin{proof}
These groups follow from the $E_{\infty}$-page of the Adams spectral sequence for $MO \langle 9 \rangle$ and the $2$-extension problems in this Section. The ``or'' statements in degree 
23 and 24 depend on the values of $d_3(\vv(24,2,9))$ and $d_2(\vv(26,0,9))$.
\end{proof}


\clearpage
% ...

\section{$O\langle n \rangle$-Cobordism}
\label{chap:o_n_cobordism}
The purpose of this chapter is to calculate the cobordism groups $\Omega^{\langle n \rangle}_{*}$ in a range of degrees for $n=10,12,16,17$. The calculations proceed in a similar fashion to 
Chapter \ref{chap:fivebrane}.

\subsection{$MO\langle 10 \rangle$}
\subsubsection{The cohomology of $MO\langle 10\rangle$} 
\begin{proposition}[\cite{Stong1963Determination}]
    $H^{*}(BO\langle10 \rangle;\mathbb{Z}/2)$ is isomorphic to:
    \[
        H^{*}(K(\mathbb{Z}/2,10))/\langle Sq^4 i_{10} \rangle \otimes \mathbb{Z}/2[\theta_{i} | L(i) > 6]
    \]
\end{proposition}
where $L(i)$ is one plus the number of ones in the binary expansion of $i-1$. We calculate the $\mathcal{A}$-module structure of $H^{*}(MO\langle10 \rangle;\mathbb{Z}/2)$ through degree $31$ using a computer program in Sage.

The $2$-adic cohomology of $BO \langle 10 \rangle$ will be useful in proving certain differentials later in this Section. The space $BO \langle 10 \rangle$
sits in the fibration $K(\mathbb{Z}/2,8) \rightarrow BO\langle 10\rangle \rightarrow BO \langle 9 \rangle$. The 2-adic cohomology of $BO \langle 9 \rangle$ in our degrees of interest was already calculated in Proposition \ref{prop: bo9 cohomology groups}. Alain Cl\'ement's Ph.D. thesis gives an algorithm for calculating the integral cohomology groups of $K(\mathbb{Z}/2^m,n)$. Cl\'ement has also provided a convenient computer program that implements the algorithm. The program is aptly named
the ``Eilenberg-MacLane machine'', and the code is available at the repository \url{https://github.com/aclemen1/EMM}. Using the Eilenberg-MacLane Machine, we obtain the following result.
\begin{proposition}\label{prop: kz2_8 cohomology groups}
Table \ref{tab:k_z2_8_groups} describes the integral cohomology groups $H^n(K(\mathbb{Z}/2,8);\mathbb{Z})$ for $0\leq n \leq 23$. 
\end{proposition}
\begin{center}
\begin{minipage}{\textwidth}
\centering
\captionof{table}[Integral cohomology of $K(\mathbb{Z}/2,8)$]{Integral cohomology groups of $K(\mathbb{Z}/2,8)$ through degree $23$.}
\label{tab:k_z2_8_groups}
\tablehead{\hline%
$n$ & $H^n(K(\mathbb{Z}/2,8),\Z)$ & $n$ & $H^n(K(\mathbb{Z}/2,8),\Z)$\\%
\hline}
\tabletail{\hline%
\multicolumn{4}{r}{%
\small\slshape to be continued on the next page}\\}
\tablelasttail{\hline}
\begin{supertabular}{|c|p{3.5cm}|c|p{3.5cm}|}
$0$  & $\Z$    & $12$ & $\Z/2$\\
$1$  & $(0)$   & $13$ & $\Z/2$\\
$2$  & $(0)$   & $14$ & $\Z/2$\\
$3$  & $(0)$   & $15$ & $(\Z/2)^{2}$\\
$4$  & $(0)$   & $16$ & $(\Z/2)^{2}$\\
$5$  & $(0)$   & $17$ & $\Z/2\oplus\Z/4$\\
$6$  & $(0)$   & $18$ & $(\Z/2)^{3}$\\
$7$  & $(0)$   & $19$ & $(\Z/2)^{4}$\\
$8$  & $(0)$   & $20$ & $(\Z/2)^{4}$\\
$9$  & $\Z/2$  & $21$ & $(\Z/2)^{5}\oplus\Z/4$\\
$10$ & $(0)$   & $22$ & $(\Z/2)^{7}$\\
$11$ & $\Z/2$  & $23$ & $(\Z/2)^{8}$\\
\end{supertabular}
\end{minipage}
\end{center}
\begin{proposition}\label{prop:bo10 cohomology groups}
Table \ref{tab:bo10_groups} describes the 2-adic cohomology groups $H^n(BO\langle 10 \rangle;\Z)$ for $0\leq n \leq 23$. 
\end{proposition}
\begin{center}
\begin{minipage}{\textwidth}
\centering
\captionof{table}[2-adic cohomology of $BO\langle 10 \rangle$]{The 2-adic cohomology groups of $BO\langle 10 \rangle;\Z$ through degree $23$.}
\label{tab:bo10_groups}
\tablehead{\hline%
$n$ & $H^n(BO\langle 10 \rangle; \Z)$ & $n$ & $H^n(BO\langle 10 \rangle; \Z)$\\%
\hline}
\tabletail{\hline%
\multicolumn{4}{r}{%
\small\slshape to be continued on the next page}\\}
\tablelasttail{\hline}
\begin{supertabular}{|c|p{3.5cm}|c|p{3.5cm}|}
$0$  & $\Z$    & $12$ & $\Z$\\
$1$  & $(0)$   & $13$ & $(0)$\\
$2$  & $(0)$   & $14$ & $\Z/2$\\
$3$  & $(0)$   & $15$ & $\Z/2$\\
$4$  & $(0)$   & $16$ & $\Z$\\
$5$  & $(0)$   & $17$ & $\Z/2$\\
$6$  & $(0)$   & $18$ & $(\Z/2)^2$\\
$7$  & $(0)$   & $19$ & $(\Z/2)^2$\\
$8$  & $(0)$   & $20$ & $\Z \oplus \Z/4$\\
$9$  & $(0)$  & $21$ & $\Z/4 \oplus \Z/2$\\
$10$ & $0$   & $22$ & $\Z/2 \oplus \Z/2$\\
$11$ & $\Z/2$  & $23$ & $(\Z/2)^3$\\
\end{supertabular}
\end{minipage}
\end{center}
\begin{proof}
Consider the Adams spectral sequence with signature\footnote{Note that the $E_2$-page description here applies the change-of-rings isomorphism as in the proof of Proposition \ref{prop: bo9 cohomology groups}.}
\[
E^{s,t}_{2} =  \text{Ext}^{s,t}_{\mathcal{A}/Sq^1\mathcal{A}}(H^*(\Sigma^{\infty} BO\langle 10 \rangle;\mathbb{Z}/2),\mathbb{Z}/2)
\Rightarrow \pi_{t-s}(H\mathbb{Z}\wedge \Sigma^{\infty} BO\langle 10 \rangle) \cong H\mathbb{Z}_{t-s}(\Sigma^{\infty} BO\langle 10 \rangle)
\]
The $E_2$-page is shown in Figure \ref{fig:hz_bo10}. The only possible nonzero differential is a $d_r$-differential between the $h_0$-towers in stems $21$ and $20$. By Lemma \ref{lem:d_2 differenital stem 21 filtration 0 bo10}, such a $d_2$-differential does occur. Like the $E_2$-page, all multiplicative generators of the $E_{\infty}$-page are in filtration $0$. 
As a result, a hidden $2$-extension would not impact the group structure in a given degree, it would instead only affect which elements generate the group. Thus, we have calculated the homology groups $H_{n}(BO \langle 10 \rangle; \Z)$. The cohomology groups $H^{n}(BO \langle 10 \rangle; \Z)$ listed in Table \ref{tab:bo10_groups} are then obtained using the Universal Coefficient Theorem as in the proof of Proposition 
\ref{prop: bo9 cohomology groups}.
\begin{figure}
\includegraphics[scale=0.7]{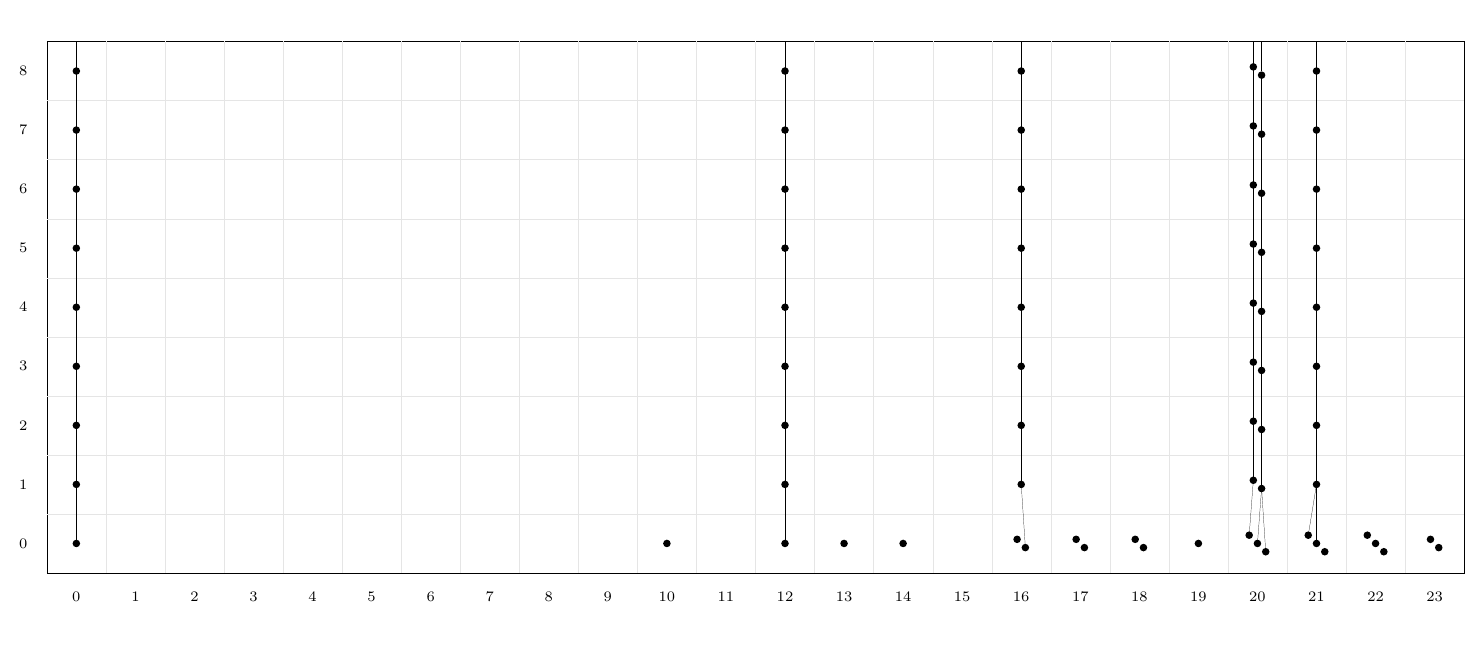}
\caption{The $E_2$-page of the Adams spectral sequence for $HZ_{*}BO\langle 10 \rangle$.\label{fig:hz_bo10}}
\end{figure}

\end{proof}
\begin{lemma}\label{lem:d_2 differenital stem 21 filtration 0 bo10}
There is a nonzero $d_2$-differential from the $h_0$-tower in stem 21 to an $h_0$-tower in stem 20 of the Adams spectral sequence for $H\mathbb{Z}(\Sigma^{\infty} BO\langle 10 \rangle)$.
\end{lemma}
\begin{proof}
Consider the Serre spectral sequence associated to the fibration $K(\mathbb{Z}/2,8) \rightarrow BO \langle 10 \rangle \rightarrow BO \langle 9 \rangle$
\[
E^2_{s,t}\cong H_s(BO\langle 9 \rangle;H_t(K(\Z/2,8);\Z)) \Rightarrow H_{s+t}(BO\langle 10 \rangle; \Z)
\]
The $E_2$-page can be calculated using Proposition \ref{prop: bo9 cohomology groups} and Proposition \ref{prop: kz2_8 cohomology groups}. For degree reasons, the only bidegrees that contribute to $H_{20}(BO\langle 10 \rangle; \Z)$ are $(0,20)$ and $(20,0)$.
These bidegrees contain $H_{20}(K(\mathbb{Z}/2,8);\Z) \cong \mathbb{Z}/4 \oplus (\mathbb{Z}/2)^5$ and $H_{20}(BO\langle 9 \rangle;\Z) \cong \mathbb{Z}$. Figure \ref{fig:hz_bo10} shows one $h_0$-tower in stem 21 and two $h_0$-towers in stem 20, which implies one $h_0$-tower from stem 20 must survive and detect a $\Z$ summand in $H_{20}(BO \langle 10 \rangle; \Z)$. 
Thus, in the Serre spectral sequence above, the $\Z$-summand in bidegree $(20,0)$ must survive. We then have an extension problem of the form
\[
0 \rightarrow E_{20,0}^{\infty} \cong \Z \rightarrow H_{20}(BO \langle 10 \rangle; \Z) \rightarrow E_{0,20}^{\infty} \rightarrow 0 
\]
where $E_{0,20}^{\infty}$ is a quotient of $\mathbb{Z}/4 \oplus (\mathbb{Z}/2)^5$. It is clear that $H_{20}(BO \langle 10 \rangle; \Z)$ will only contain one $\Z$-summand. Consequently, there must be a nonzero $d_r$-differential, $r\geq 2$, from the $h_0$-tower in stem 21 that truncates one of the $h_0$-towers in stem 20 on a later page. Since the extension problem above implies that $H_{20}(BO \langle 10 \rangle; \Z)$ has at most 4-torsion, it must be a $d_2$-differential.
\end{proof}
\subsubsection{Adams charts}
This section contains charts for the Adams spectral sequence for $MO\langle 10 \rangle$. The $E_r$-page is displayed along with the $d_r$-differentials. The $y$-axis is the Adams filtration $s$ and the $x$ axis is the stem $t-s$. The elements $h_0$, $h_1$, and $h_2$ have had their names suppressed. The dotted lines on the $E_{\infty}$-page represent possibly nonzero differentials that have yet to be determined. 
\clearpage
\FloatBarrier
\begin{figure}[!p]
\centering

\begin{subfigure}{\linewidth}
  \centering
  \includegraphics[width=\linewidth,
    trim=0.5cm 0cm 0cm 0cm,clip]{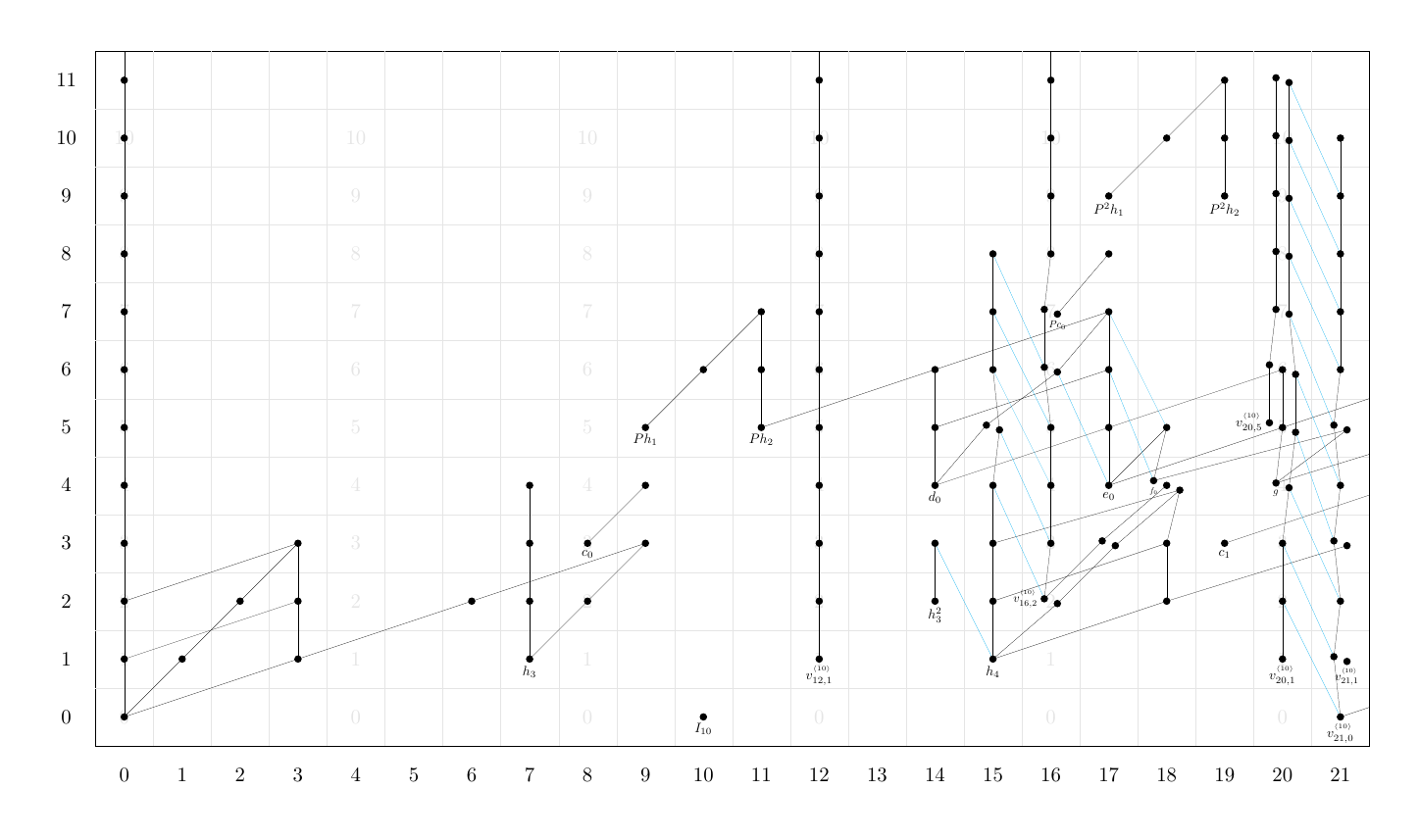}
  \caption{$E_2$-page $MO \langle 10 \rangle $ (stems 0-21)}
  \label{fig:mo10_e2}
\end{subfigure}

\vspace{-0.4\baselineskip}

\begin{subfigure}{\linewidth}
  \centering
  \includegraphics[width=\linewidth,
    trim=0.5cm 0cm 0cm 0cm,clip]{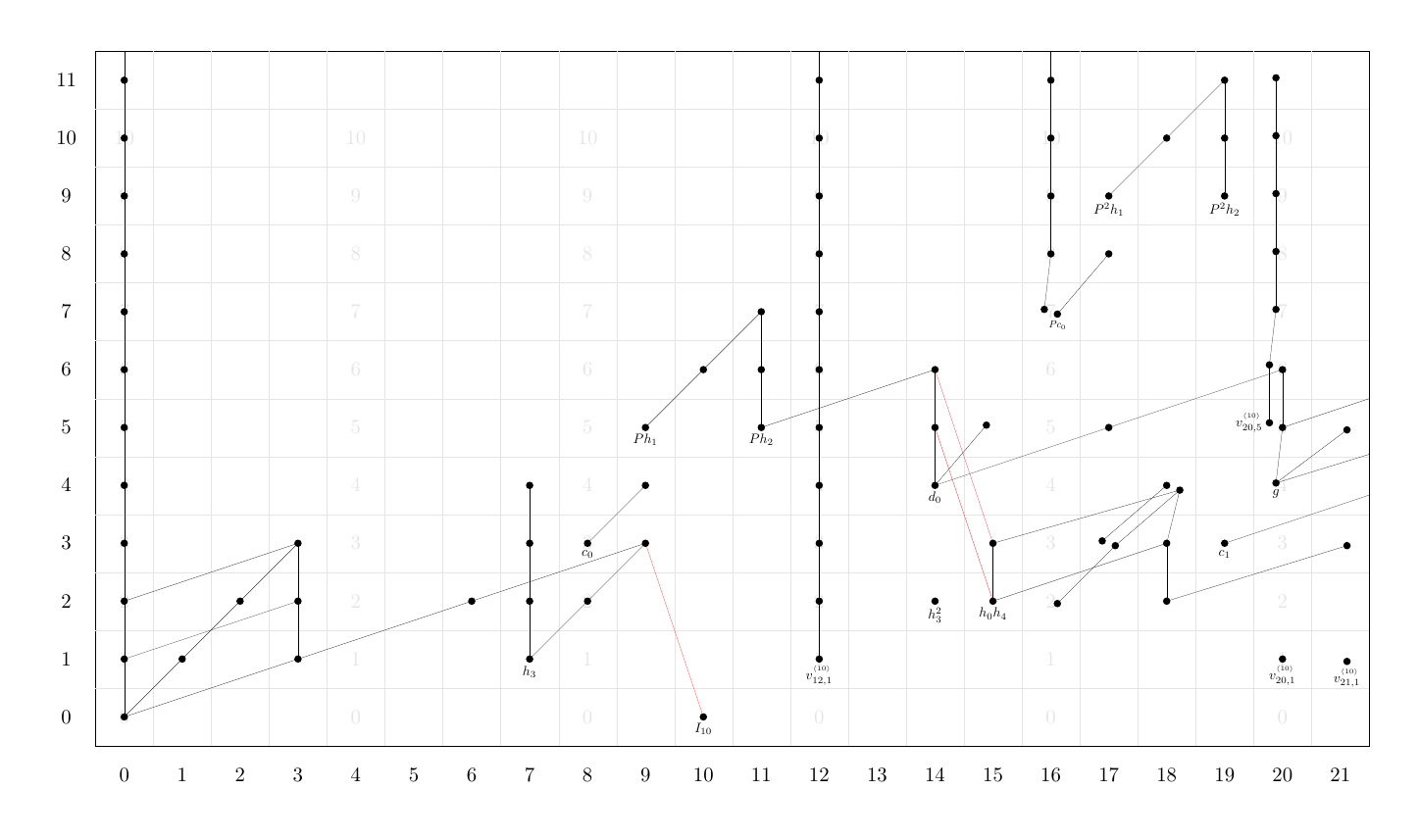}
  \caption{$E_3$-page $MO \langle 10 \rangle $ (stems 0-21)}
  \label{fig:mo10_e3}
\end{subfigure}

\caption{$E_2$/$E_3$-page $MO\langle 10 \rangle$ \label{fig:mo10_e2_e3}}
\end{figure}

\begin{figure}[!p]
\centering
\begin{subfigure}{\linewidth}
  \centering
  \includegraphics[width=\linewidth,
    trim=0.5cm 0cm 0cm 0cm,clip]{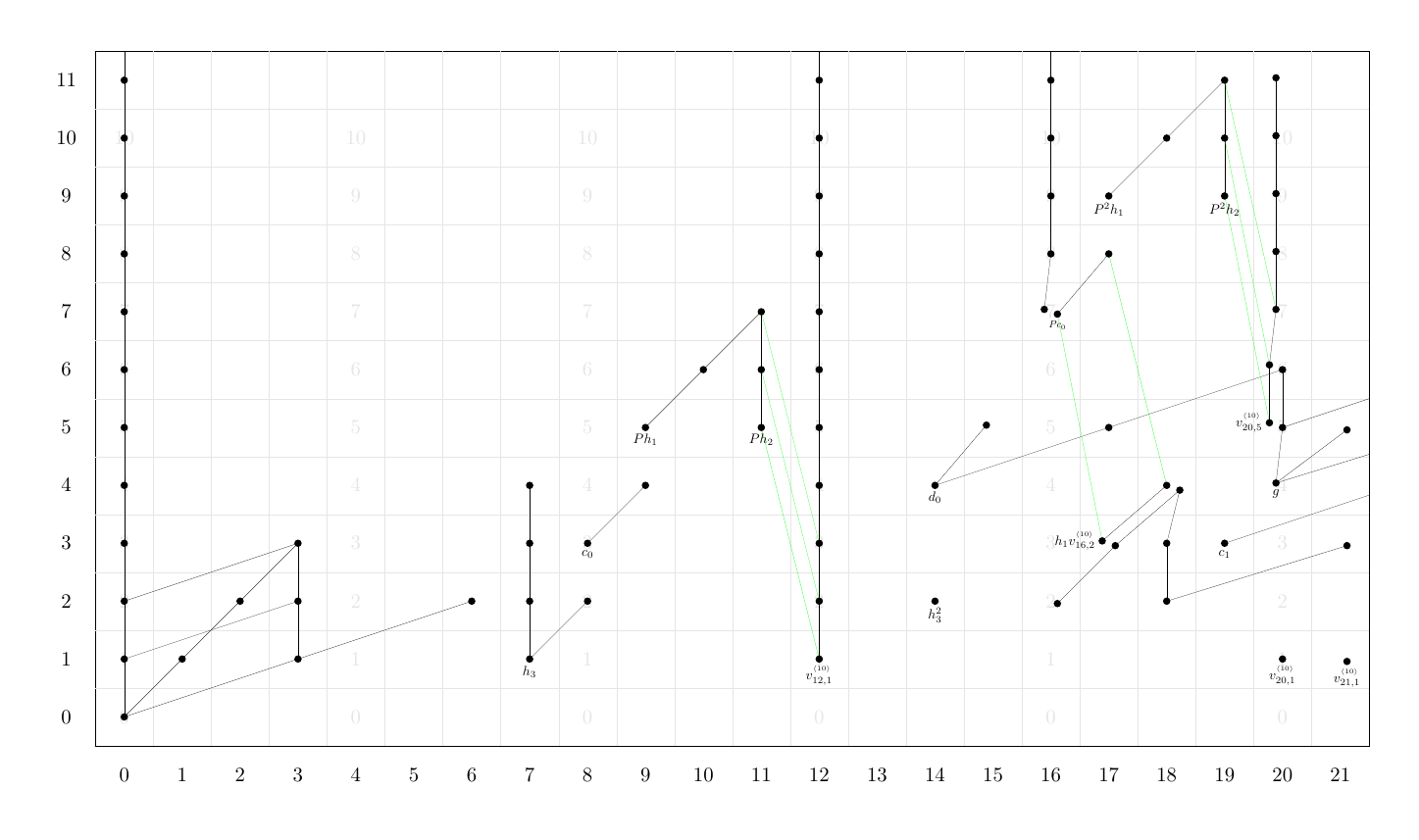}
  \caption{$E_4$-page $MO \langle 10 \rangle $ (stems 0-21)}
  \label{fig:mo10_e4_0_24}
\end{subfigure}

\vspace{-0.4\baselineskip}

\begin{subfigure}{\linewidth}
  \centering
  \includegraphics[width=\linewidth,
    trim=0.5cm 0cm 0cm 0cm,clip]{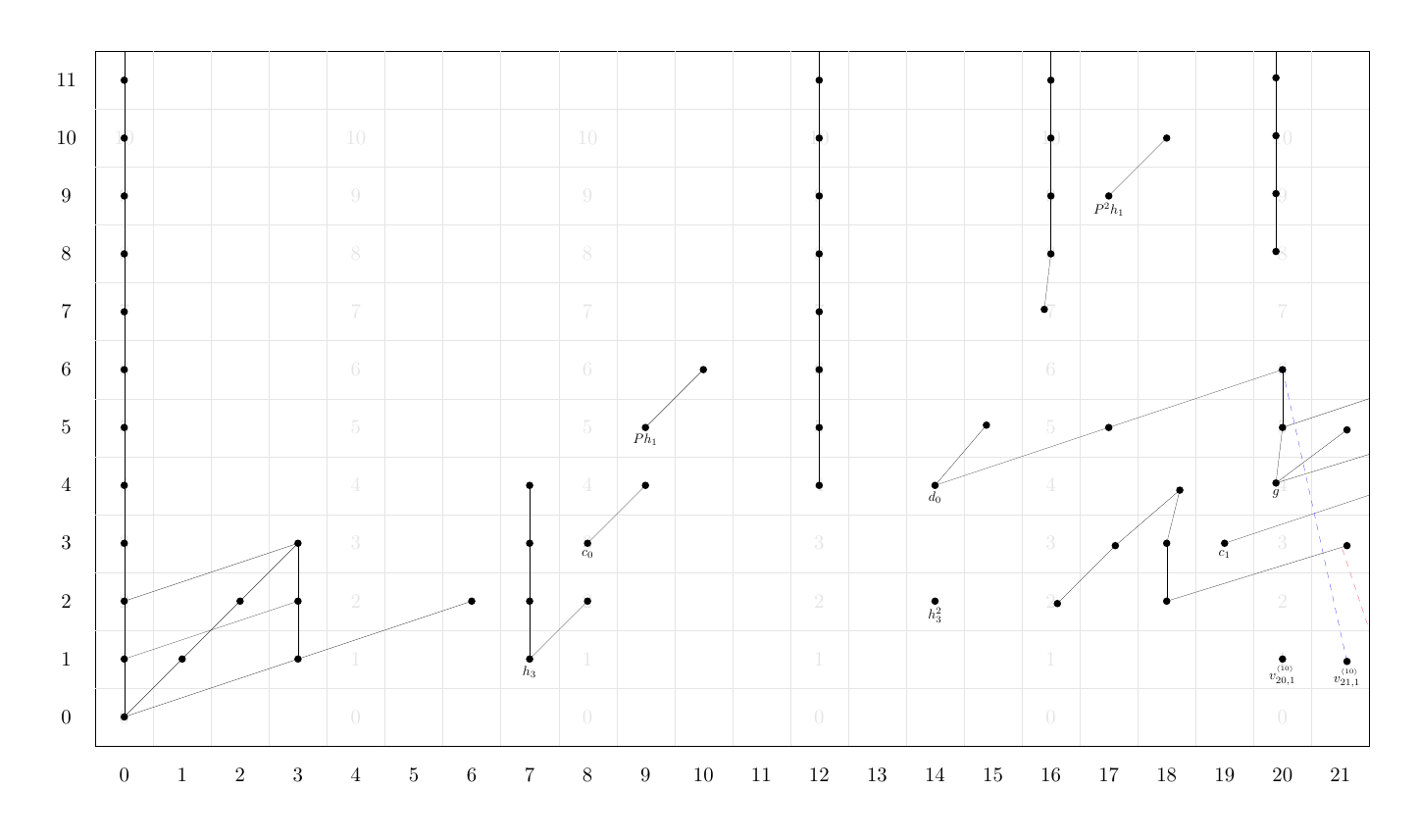}
  \caption{$E_{\infty}$-page $MO \langle 10 \rangle $ (stems 0-21)}
  \label{fig:mo10_einf_0_21}
\end{subfigure}

\caption{$E_4$/$E_{\infty}$-page $MO\langle 10 \rangle$ \label{fig:mo10_e4_einf}}
\end{figure}
\FloatBarrier
\subsubsection{$d_2$-differentials}
\begin{proposition}
Table \ref{tab:Adams d_2 mo10} describes the non-zero $d_2$-differentials in the Adams spectral sequence for $MO\langle 10 \rangle$  on all indecomposables on $E_2$ through stem 21. 
\end{proposition}
The following table lists the only indecomposables through stem 21 that, for degree reasons, can support a $d_2$-differential. 
 \begin{longtable}{llllc} 
    \caption[Possible $d_2$-differentials $MO\langle 10 \rangle$]{Possible non-zero $d_2$-differentials on indecomposable elements
    \label{tab:Adams d_2 mo10}
    } \\
    \toprule
    $x$ & $(t-s,s)$ & $d_2(x)$ & Occurs & Proof\\
    \midrule \endfirsthead
    \caption[]{Possible non-zero $d_2$-differentials on indecomposable elements} \\
    \toprule
    $x$ & $(t-s,s)$ & $d_2(x)$ & Occurs & Proof\\
    \midrule \endhead
    \bottomrule \endfoot
        $h_1$ & $(1, 1)$ & $h_0^3$ & No & \ref{prop: sphere differentials}\\
        $h_4$ & $(15, 1)$ & $h_0h_3^2$ & Yes & \ref{prop: sphere differentials}\\
        $e_0$ & $(15, 4)$ & $h_1^2d_0$ & Yes & \ref{prop: sphere differentials}\\
        $\vv(16,2,10)$ & $(16, 2)$ & $h_0^3h_4$ & Yes & \ref{lem:d_2-differentials-imj-mo10}\\
        $f_0$ & $(18, 4)$ & $h_0^2e_0$ & Yes & \ref{prop: sphere differentials}\\
        $c_1$ & $(19, 3)$ & $h_0f_0$ & No & \ref{prop: sphere differentials}\\
        $\vv(20,1,10)$ & $(20, 1)$ & $c_1$ & No & \ref{lem:d2 differential in stem 20 filtration 1 mo10} \\
        $\vv(21,0,10)$ & $(21, 0)$ & $h_0\vv(20,1,10)$ & Yes & \ref{lem:d2 differential in stem 21 filtration 0 mo10}
    \end{longtable}

For proofs of these differentials, see the following series of Lemmas. 
\begin{lemma}\label{lem:d_2-differentials-imj-mo10}
$d_2(\vv(16,2,10))=h_0^3h_4$
\end{lemma}
\begin{proof}
The element $h_0^3h_4$ is in the $h_0$-tower that ends at the ``Adams edge'' in stem 15. By Theorem \ref{mahowaldimj}, it is in the image of $J$. By Theorem \ref{hoveyimj}, it must not survive the Adams spectral sequence for $MO \langle 10 \rangle$. 
It cannot support a differential since $d_2(h_0^3h_4)=0$ in the Adams spectral sequence for $\mathbb{S}$. Thus, it must be the target of a differential. The only possibility is $d_2(\vv(16,2,10))=h_0^3h_4$. 
\end{proof}

\begin{lemma}\label{lem:d2 differential in stem 20 filtration 1 mo10}
$d_2(\vv(20,1,10))=0$
\end{lemma}
\begin{proof}
The only possible target for this differential is $c_1$, but $h_2c_1 \neq 0$ while $h_2\vv(20,1,10)=0$. The $d_2$-differential is $h_2$-linear ($h_2$ is a $d_2$-cycle), so $d_2(\vv(20,1,10))$ must equal zero. 
\end{proof}

\begin{lemma}\label{lem:d2 differential in stem 21 filtration 0 mo10}
$d_2(\vv(21,0,10))=h_0\vv(20,1,10)$
\end{lemma}
\begin{proof}
The element $\vv(21,0,10)$ supports an $h_0$-tower in stem 21. By the May-Milgram theorem \ref{maymilgram}, $d_2$-differentials out of $h_0$-towers in stem 20 are in bijective correspondence
with $\Z/4$ direct summands in $H^{21}(MO\langle 10 \rangle; \Z)$. By Proposition \ref{prop:bo10 cohomology groups} and the Thom isomorphism, $H^{21}(MO\langle 10 \rangle; \Z)$ has a single $\Z/4$ direct summand. Furthermore, the $h_0$-tower generated by $\vv(21,0,10)$ is the only such tower in stem $21$. 
Thus, $d_2(\vv(21,0,10))$ is nonzero and the only possible target is $h_0\vv(20,1,10)$. 
\end{proof}

\subsubsection{$d_3$-differentials}
\begin{proposition}
Table \ref{tab:Adams d_3 mo10} describes the non-zero $d_3$-differentials in the Adams spectral sequence for $MO\langle 10 \rangle$  on all indecomposables on $E_3$ through stem 20. 
\end{proposition}
The following table lists the only indecomposables through stem 20 that, for degree reasons, can support a $d_3$-differential. 
 \begin{longtable}{llllc}
    \caption[Possible $d_3$-differentials $MO\langle 10 \rangle$]{Possible non-zero $d_3$-differentials on indecomposable elements.
    \label{tab:Adams d_3 mo10}
    } \\
    \toprule
    $x$ & $(t-s,s)$ & $d_3(x)$ & Occurs & Proof\\
    \midrule \endfirsthead
    \caption[]{Possible non-zero $d_3$-differentials on indecomposable elements} \\
    \toprule
    $x$ & $(t-s,s)$ & $d_3(x)$ & Occurs & Proof\\
    \midrule \endhead
    \bottomrule \endfoot
        $I_{10}$ & $(10,0)$ & $h_1^2h_3$ & Yes & \ref{lem:d_3 differential in stem 10 filtration 0 mo10}\\
        $h_0h_4$ & $(15,2)$ & $h_0d_0$ & Yes & \ref{prop: sphere differentials}\\
        $h_1h_4$ & $(16,2)$ & $h_1d_0$ & No & \ref{prop: sphere differentials}\\
    \end{longtable}
\begin{proof}
The differentials on $h_0h_4$ and $h_1h_4$ were already proven in \ref{prop: sphere differentials}. For the differential on $I_{10}$ see the following Lemma.
\end{proof}
\begin{lemma}\label{lem:d_3 differential in stem 10 filtration 0 mo10}
$d_3(I_{10})=h_1^2h_3$
\end{lemma}
\begin{proof}
Stem 9 contains the elements $h_1^2h_3$ and $Ph_1$, where $h_1^2h_3$ and $Ph_1$ detect $\eta^2 \sigma$ and $\mu$ in $\pi_{9} \mathbb{S}$. But $\mu=\mu_{9}$ (see Remark 11.41 \cite{BrunerRognestmf}), so by Theorem \ref{adamsimj}, $Ph_1$ does not detect an element in $\text{Im } J$. This implies $h_1^2h_3$ must detect an element in $\text{Im } J$, which by Theorem \ref{hoveyimj}  implies it detects a class in the kernel of the unit map $\pi_{9} \mathbb{S} \rightarrow \pi_{9} MO \langle 10 \rangle$. Thus, $h_1^2h_3$ must not survive the Adams spectral sequence for $MO \langle 10 \rangle$. The only possible differential with target $h_1^2h_3$ is $d_3(I_{10})$. 
\end{proof}

\subsubsection{$d_4$-differentials}
\begin{proposition}
Table \ref{tab:Adams d_4 mo10} describes the non-zero $d_4$-differentials in the Adams spectral sequence for $MO\langle 10 \rangle$  on all indecomposables on $E_4$ through stem 20. 
\end{proposition}
The following table lists the only indecomposables through stem 20 that, for degree reasons, can support a $d_4$-differential. 
 \begin{longtable}{llllc}
    \caption[Possible $d_4$-differentials $MO\langle 10 \rangle$]{Possible non-zero $d_4$-differentials on indecomposable elements
    \label{tab:Adams d_4 mo10}
    } \\
    \toprule
    $x$ & $(t-s,s)$ & $d_4(x)$ & Occurs & Proof\\
    \midrule \endfirsthead
    \caption[]{Possible non-zero $d_4$-differentials on indecomposable elements} \\
    \toprule
    $x$ & $(t-s,s)$ & $d_4(x)$ & Occurs & Proof\\
    \midrule \endhead
    \bottomrule \endfoot
        $\vv(12,1,10)$ & $(12, 1)$ & $Ph_2$ & Yes & \ref{lem:d_4-differentials-imj-mo10}\\
        $h_1\vv(16,2,10)$ & $(17,3)$ & $Pc_0$ & Yes & \ref{lem:d_4-differentials-imj-mo10} \\
        $h_1^2h_4$ & $(17,3)$ & $Pc_0$ & No & \ref{lem:d_4-differentials-sphere-mo10} \\
       $\vv(20,5,10)$ & $(20,5)$ & $P^2h_2$ & Yes & \ref{d_4 differential in stem 20 filtration 5 mo10} \\
    \end{longtable}
\begin{proof}
For proofs of these differentials, see the following Lemmas:
\end{proof}

\begin{lemma}\label{lem:d_4-differentials-sphere-mo10}
The element $h_1^2h_4$ is a permanent cycle. 
\end{lemma}
\begin{proof}
The element $h_1^2h_4$  is a permanent cycle in the Adams spectral sequence for $\mathbb{S}$, and so by naturality is a permanent cycle in the Adams spectral sequence for $MO \langle 10 \rangle$. 
\end{proof}
\begin{lemma}\label{lem:d_4-differentials-imj-mo10}
$d_4(\vv(12,1,10))=Ph_2$ and $d_4(h_1\vv(16,2,10))=Pc_0$. 
\end{lemma}
\begin{proof}
By Theorem \ref{mahowaldimj}, the target of each of these differentials is in the image of $J$. By Theorem \ref{hoveyimj}, they then detect elements in the kernel of the unit map $\mathbb{S} \rightarrow MO\langle n \rangle$ and so they must not survive the spectral sequence. 
The element $\vv(12,1,10)$ is the only possible source of a differential with target $Ph_2$, and so $d_4(\vv(12,1,10))=Ph_2$. Other than $h_1\vv(16,2,10)$, the element $h_1^2h_4$ is the only other possible source of a differential with target $Pc_0$. By Lemma \ref{lem:d_4-differentials-sphere-mo10}, $d_4(h_1^2h_4)=0$, and so we conclude $d_4(h_1\vv(16,2,10))=Pc_0$.
\end{proof}
\begin{lemma}\label{d_4 differential in stem 20 filtration 5 mo10}
$d_4(\vv(20,5,10))=P^2h_2$
\end{lemma} 
\begin{proof}
The element $\vv(20,5,10)$ is contained in the Massey product $\langle h_0, h_0^3h_3, \vv(12,1,10) \rangle$  on the $E_2$-page with indeterminacy $\{0,h_0^4\vv(20,1,10), h_0g\}$. Note that we do not know if $h_0^4\vv(20,1,10)$ survives to the $E_4$-page. We assume it does since if it does not the proof of this lemma is a simpler version of what is recorded here. The Massey product $\langle h_0, h_0^3h_3, \vv(12,1,10) \rangle$ is in fact a Massey product on the $E_4$-page since all three elements survive to the $E_4$-page. 
By Moss's higher Leibniz rule \ref{mossleibniz},
\begin{align*}
d_4(\langle h_0, h_0^3h_3, \vv(12,1,10) \rangle)=d_4(\vv(20,5,10)) \in\;&
\langle d_4(h_0),h_0^3h_3,\vv(12,1,10)\rangle+\langle h_0,d_4(h_0^3h_3),\vv(12,1,10)\rangle+\langle h_0,h_0^3h_3,d_4(\vv(12,1,10))\rangle\\
=\;&\langle 0,h_0^3h_3,\vv(12,1,10)\rangle+\langle h_0,0,\vv(12,1,10)\rangle+\langle h_0,h_0^3h_3,Ph_2\rangle.
\end{align*}
where the first equality holds since $d_4(h_0g)=0$ and $d_4(h_0^4\vv(20,1,10))=0$. The first two terms vanish and the third contains $P^2h_2$ with zero indeterminacy. Thus, the only possibility is $d_4(\vv(20,5,10))=P^2h_2$
\end{proof}
For degree reasons, $P^2h_1$ is the only element that may possibly support a differential of length greater than 4. It does not:

\begin{lemma}
    $P^2h_1$ is a permanent cycle.
\end{lemma}
\begin{proof}
The element $P^2h_1$ is a permanent cycle in the Adams spectral sequence for $\mathbb{S}$, and so by naturality is a permanent cycle in the Adams spectral sequence for $MO \langle 10 \rangle$. 
\end{proof}
\subsubsection{The abutment}
Recall that there are no hidden $2$-extensions between elements in the image of the unit map in this range of degrees. The first possible hidden $2$-extension is in stem 16. 

\begin{proposition}
$\pi_{16}MO\langle 10 \rangle \cong \Z \oplus \Z/2$
\end{proposition}
\begin{proof}
The proof is identical to the proof of Proposition \ref{prop:extension problem stem 16 mo9}.
\end{proof}

\begin{proposition}
$\pi_{20}MO\langle 10 \rangle \cong \Z \oplus F$ where $F$ is a finitely-generated abelian group such that $0 \leq |F| \leq 16$.
\end{proposition}
\begin{proof}
Immediate from the $E_{\infty}$-page. 
\end{proof}

\begin{theorem}\label{thm: mo10 groups}
Table \ref{tab:mo10 groups} describes the 2-primary component of the cobordism groups $\Omega^{\langle 10 \rangle}_n$ for the values of $n$ listed. Here, $F$ is an finite abelian group such that $0\leq|F|\leq 16$.
\end{theorem}
\begin{center}
\captionof{table}[The 2-primary component of the $O\langle 10\rangle$-cobordism groups]{The 2-primary component of $\Omega^{\langle 10\rangle}_n$ for the values of $n$ listed.}
\label{tab:mo10 groups}
\tablehead{\hline%
$n$ & $\Omega^{\langle 10\rangle}_n$ & $n$ & $\Omega^{\langle 10\rangle}_n$\\%
\hline}
\tabletail{\hline}
\tablelasttail{\hline}
\begin{supertabular}{|c|p{5.2cm}|c|p{5.2cm}|}
$0$  & $\Z$ & $10$ & $\Z/2$  \\
$1$  & $\Z/2$ & $11$ & $(0)$  \\
$2$  & $\Z/2$ & $12$ & $\Z$ \\
$3$  & $\Z/8$ & $13$ & $(0)$ \\
$4$  & $(0)$ & $14$ & $\Z/2 \oplus \Z/2$\\
$5$  & $(0)$ & $15$ & $\Z/2$\\
$6$  & $\Z/2$ & $16$ & $\Z \oplus \Z/2$\\
$7$  & $\Z/16$ & $17$ & $(\Z/2)^3$\\
$8$  & $\Z/2\oplus\Z/2$ & $18$ & $\Z/8 \oplus \Z/2$\\
$9$  & $\Z/2 \oplus \Z/2$ & $19$ &  \\
    & & $20$ & $\Z \oplus F$ \\
\end{supertabular}
\end{center}

\subsection{$MO\langle 12 \rangle$}
\subsubsection{The cohomology of $MO\langle 12\rangle$} 

\begin{proposition}[\cite{Stong1963Determination}]
    $H^{*}(BO\langle12 \rangle;\mathbb{Z}/2)$ is isomorphic to:
    \[
        H^{*}(K(\mathbb{Z},12); \Z/2)/\langle Sq^5 i_{12}) \rangle \otimes \mathbb{Z}/2[\theta_{i} | L(i) > 7]
    \]
\end{proposition}
where $L(i)$ is one plus the number of ones in the binary expansion of $i-1$. We calculate the $\mathcal{A}$-module structure of $H^{*}(MO\langle12 \rangle;\mathbb{Z}/2)$ through degree $70$ using a computer program in Sage.

\begin{proposition}\label{prop:bo_9_10_12_equiv}
The spaces $BO\langle 9\rangle$, $BO\langle 10 \rangle$, and $BO\langle 12\rangle$ are rationally homotopy equivalent.
\end{proposition}
\begin{proof}
We will prove a rational homotopy equivalence between $BO\langle 9 \rangle$ and $BO \langle 10 \rangle$. The proof of such an equivalence between $BO\langle 10 \rangle$ and $BO\langle 12\rangle$ 
is identical. 
Let $ f: BO\langle 10 \rangle \rightarrow BO\langle 9 \rangle$ be the fibration in the Whitehead tower in Figure \ref{bo_whitehead}. The fiber of $f$ is $K(\mathbb{Z}/2,8)$, which only has torsion in its homotopy groups (and only in a single degree).
The fibration $f$ induces a long exact sequence of homotopy groups 
\[
    \dots \rightarrow \pi_{k}(K(\mathbb{Z}/2,8)) \rightarrow \pi_{k}(BO\langle 10 \rangle) \rightarrow \pi_{k}(BO\langle 9 \rangle) \rightarrow \pi_{k-1}(K(\mathbb{Z}/2,8)) \rightarrow \dots
\]

Tensoring with $\mathbb{Q}$, which causes any torsion groups to vanish, we obtain 
\[
0 \rightarrow \pi_{k}(BO\langle 10 \rangle) \otimes \mathbb{Q} \rightarrow \pi_{k}(BO\langle 9 \rangle) \otimes \mathbb{Q} \rightarrow 0
\]
Hence, $ \pi_{k}(BO\langle 10 \rangle)\otimes \mathbb Q \cong \pi_{k}(BO\langle 9 \rangle) \otimes \mathbb Q$ for all $k$. 
\end{proof}

\begin{proposition}\label{prop:rational bo12}
The rational cohomology of $BO \langle 12 \rangle$ is isomorphic to 
\[
 \mathbb{Q}[p_3,p_4,p_5,\dots]
\]
where $|p_i|=4i$. 
\end{proposition}
\begin{proof}
This result follows from combining Proposition \ref{prop: rational bo9} and Proposition \ref{prop:bo_9_10_12_equiv}.
\end{proof}

\subsubsection{Adams charts}
This section contains charts for the Adams spectral sequence for $MO\langle 12 \rangle$. The $E_r$-page is displayed along with the $d_r$-differentials. The $y$-axis is the Adams filtration $s$ and the $x$ axis is the stem $t-s$. The elements $h_0$, $h_1$, and $h_2$ have had their names suppressed. The dotted lines on the $E_{\infty}$-page represent possibly nonzero differentials that have yet to be determined. 
\clearpage
\FloatBarrier
\begin{figure}[!p]
\centering

\begin{subfigure}{\linewidth}
  \centering
  \includegraphics[width=\linewidth,
    trim=0.5cm 0cm 0cm 0cm,clip]{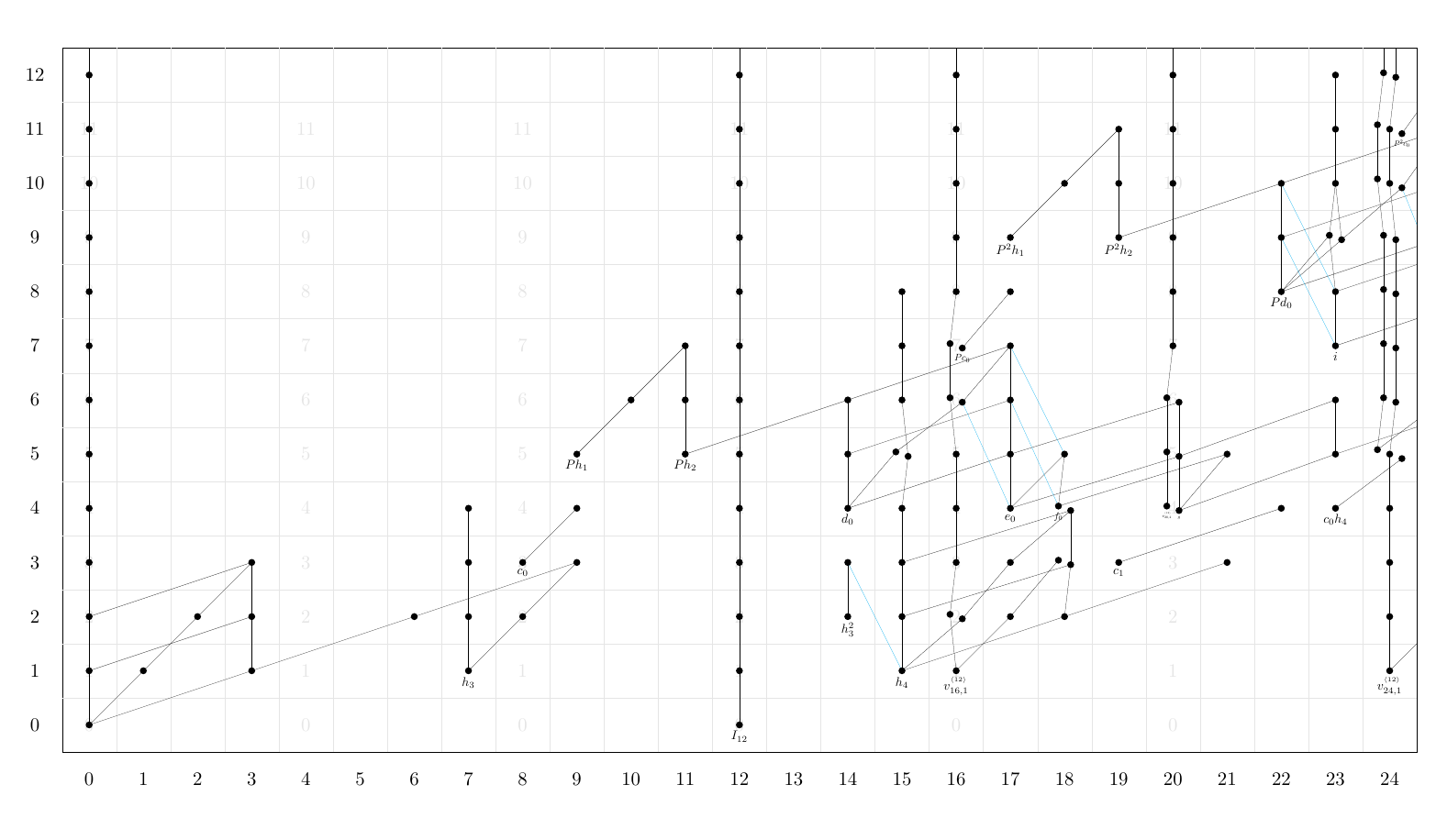}
  \caption{$E_2$-page $MO \langle 12 \rangle $ (stems 0-24)}
  \label{fig:mo12_e2_0_24}
\end{subfigure}

\vspace{-0.4\baselineskip}

\begin{subfigure}{\linewidth}
  \centering
  \includegraphics[width=\linewidth,
    trim=0.5cm 0cm 0cm 0cm,clip]{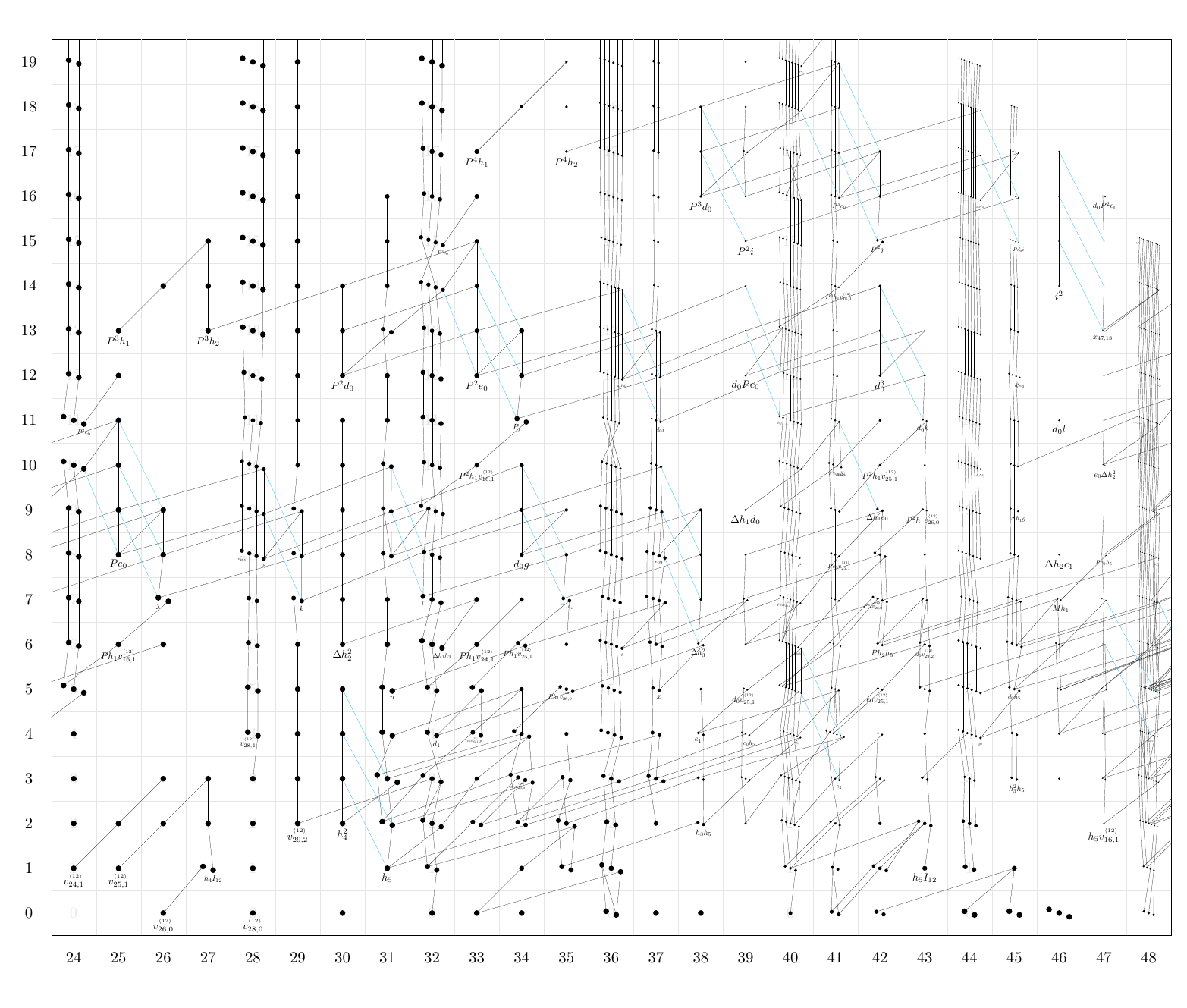}
  \caption{$E_2$-page $MO \langle 12 \rangle $ (stems 24-48)}
  \label{fig:mo12_e2_24_48}
\end{subfigure}

\caption{$E_2$-page $MO\langle 12 \rangle$ \label{fig:mo12_e2}}

\end{figure}

\begin{figure}[!p]
\centering
\begin{subfigure}{\linewidth}
  \centering
  \includegraphics[width=\linewidth,
    trim=0.5cm 0cm 0cm 0cm,clip]{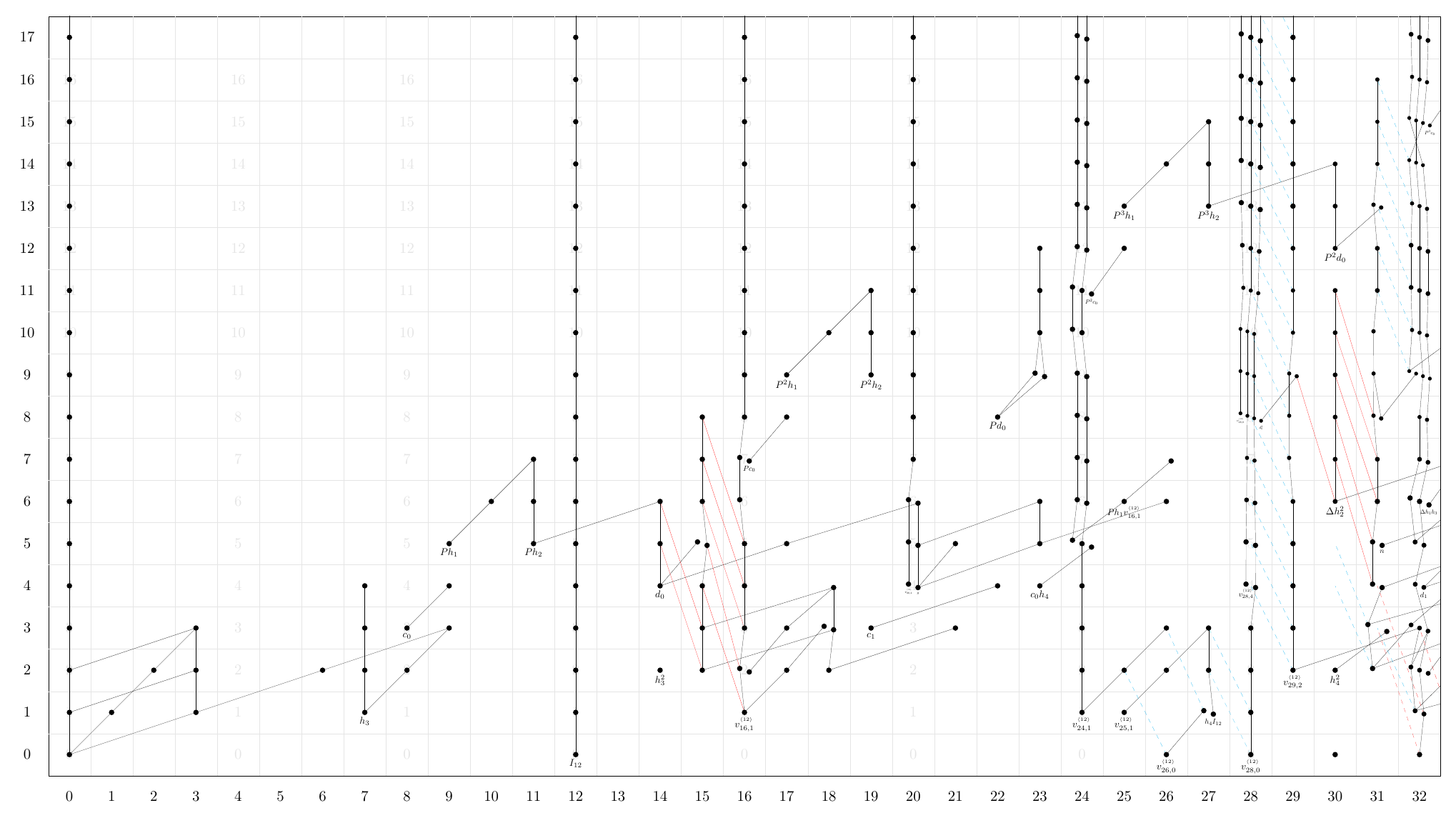}
  \caption{$E_3$-page $MO \langle 12 \rangle $ (stems 0-32)}
  \label{fig:mo12_e3}
\end{subfigure}

\vspace{-0.4\baselineskip}

\begin{subfigure}{\linewidth}
  \centering
  \includegraphics[width=\linewidth,
    trim=0.5cm 0cm 0cm 0cm,clip]{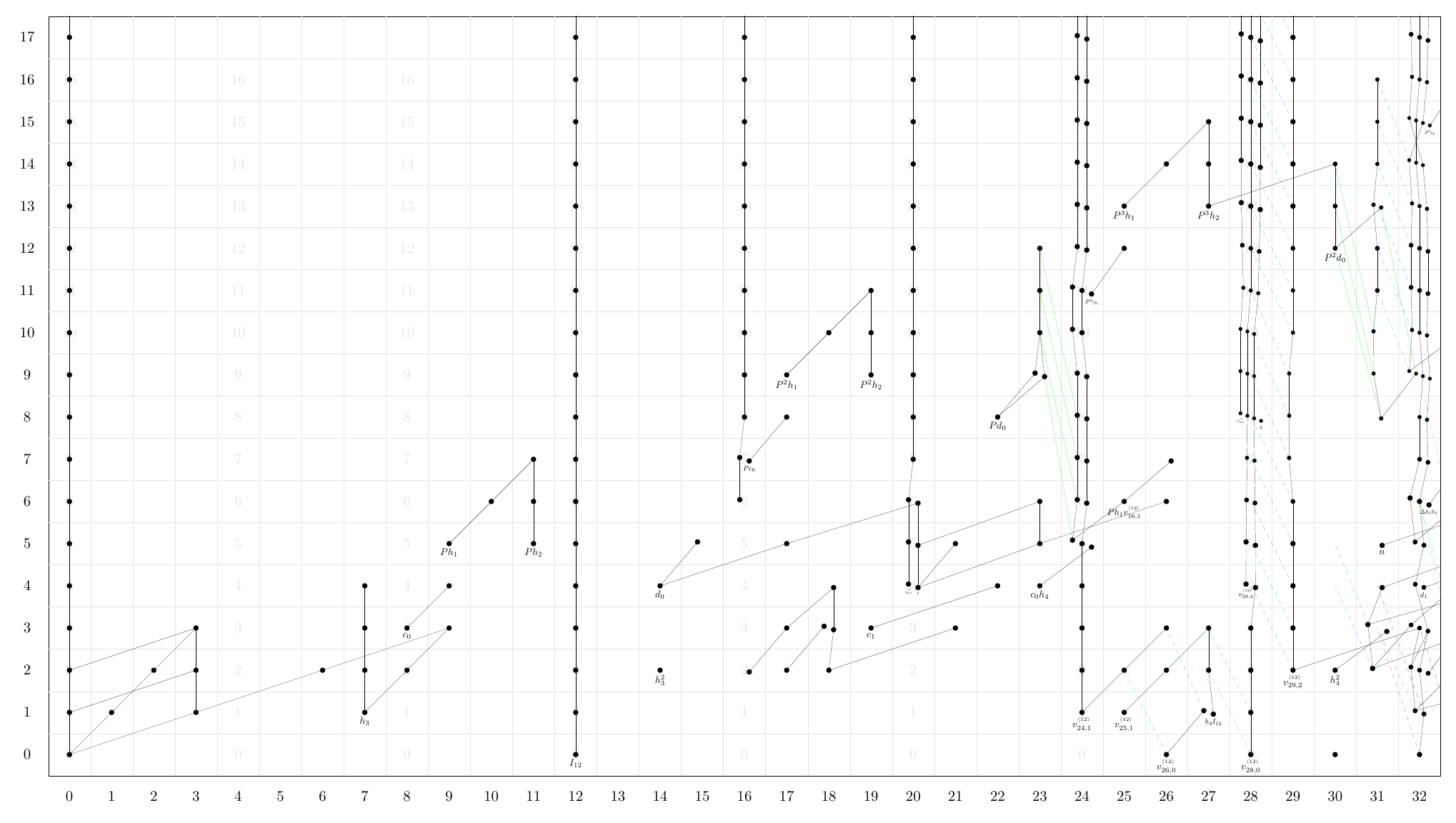}
  \caption{$E_4$-page $MO \langle 12 \rangle $ (stems 0-32)}
  \label{fig:mo12_e4}
\end{subfigure}

\caption{$E_3$/$E_4$-page $MO\langle 12 \rangle$ \label{fig:mo12_e3_e4}}
\end{figure}

\begin{figure}[!p]
\centering
\begin{subfigure}{\linewidth}
  \centering
  \includegraphics[width=\linewidth,
    trim=0.5cm 0cm 0cm 0cm,clip]{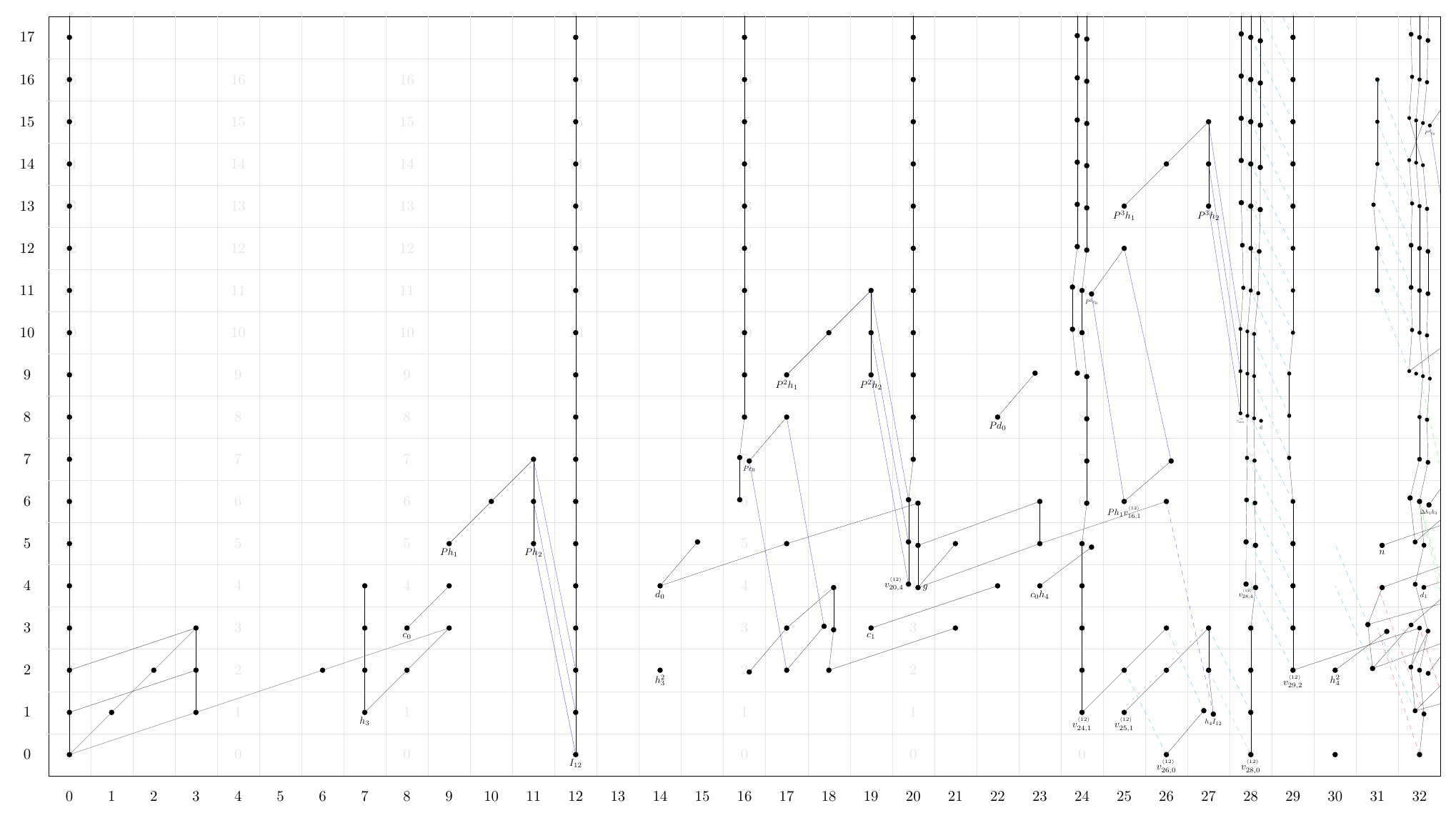}
  \caption{$E_5$-page $MO \langle 12 \rangle $ (stems 0-32)}
  \label{fig:mo12_e5}
\end{subfigure}

\vspace{-0.4\baselineskip}

\begin{subfigure}{\linewidth}
  \centering
  \includegraphics[width=\linewidth,
    trim=0.5cm 0cm 0cm 0cm,clip]{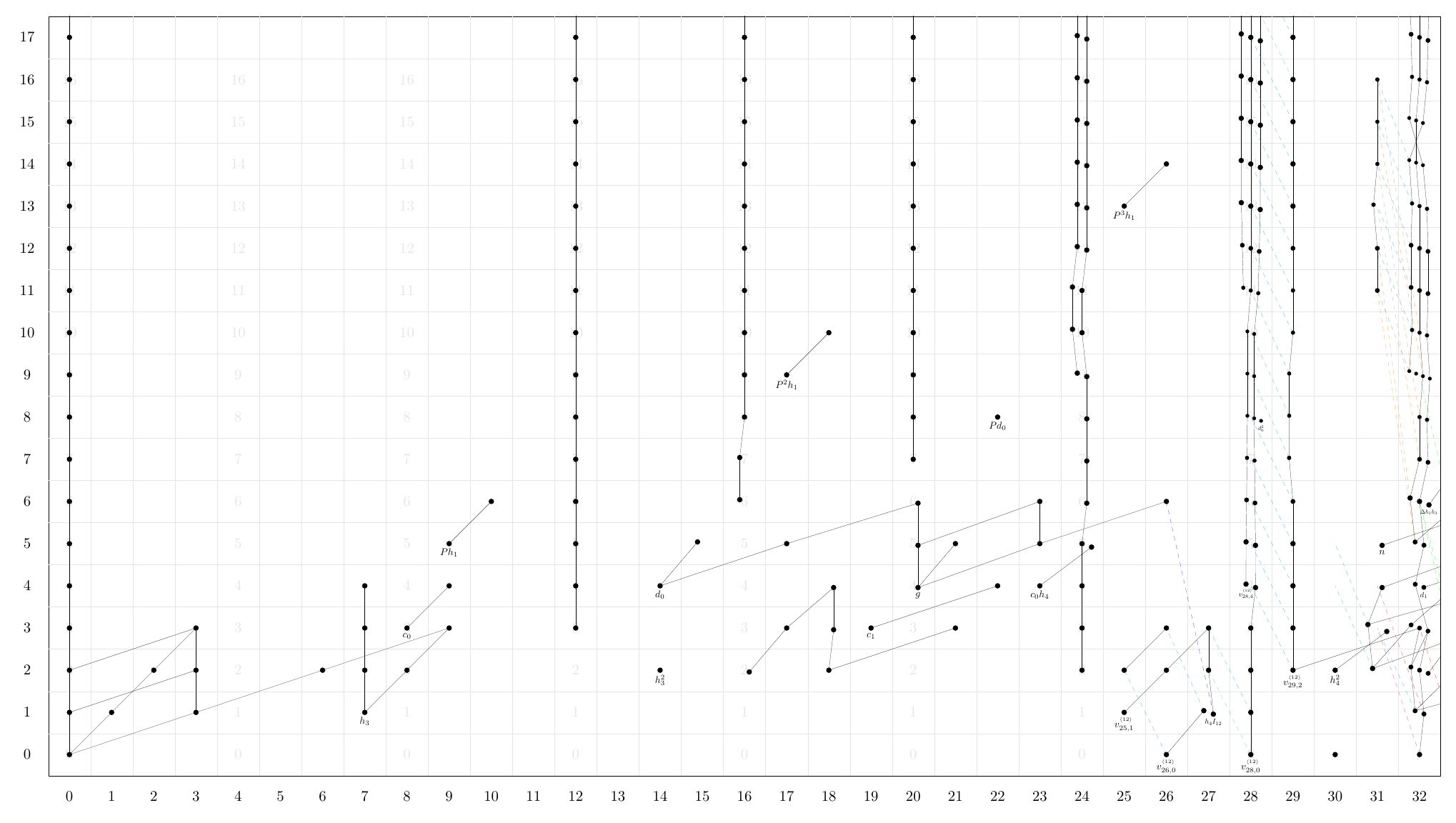}
  \caption{$E_{\infty}$-page $MO \langle 12 \rangle $ (stems 0-32)}
  \label{fig:mo12_einf}
\end{subfigure}

\caption{$E_5$/$E_{\infty}$-page $MO\langle 12 \rangle$ \label{fig:mo12_e5_einf}}

\end{figure}

\FloatBarrier

\subsubsection{$d_2$-differentials}
\begin{proposition}
Table \ref{tab:Adams d_2 mo12} describes the non-zero $d_2$-differentials in the Adams spectral sequence for $MO\langle 12 \rangle$  on all indecomposables on $E_2$ through stem 31, and select elements in stems $> 31$. 
\end{proposition}
The following table lists the only indecomposables through stem 31 that, for degree reasons, can support a $d_2$-differential. 
 \begin{longtable}{llllc}
    \caption[Possible $d_2$-differentials $MO\langle 12 \rangle$]{Possible non-zero $d_2$-differentials on indecomposable elements.
    \label{tab:Adams d_2 mo12}
    } \\
    \toprule
    $x$ & $(t-s,s)$ & $d_2(x)$ & Occurs & Proof\\
    \midrule \endfirsthead
    \caption[]{Possible non-zero $d_2$-differentials on indecomposable elements} \\
    \toprule
    $x$ & $(t-s,s)$ & $d_2(x)$ & Occurs & Proof\\
    \midrule \endhead
    \bottomrule \endfoot
        $h_1$ & $(1, 1)$ & $h_0^3$ & No & \ref{prop: sphere differentials}\\
        $h_4$ & $(15, 1)$ & $h_0h_3^2$ & Yes & \ref{prop: sphere differentials}\\
        $e_0$ & $(15, 4)$ & $h_1^2d_0$ & Yes & \ref{prop: sphere differentials}\\
        $f_0$ & $(18, 4)$ & $h_0^2e_0$ & Yes & \ref{prop: sphere differentials}\\
        $i$ & $(23, 7)$ & $h_0Pd_0$ & Yes & \ref{prop: sphere differentials}\\
        $\vv(26,0,12)$ & $(26,0)$ & $h_1\vv(24,1,12)$ & ? & \\
        $j$ & $(26, 7)$ & $h_0Pe_0$ & Yes & \ref{prop: sphere differentials}\\
        $\vv(28,0,12)$ & $(28,0)$ & $h_0h_4\vv(12,0,12)$ & ? & \\
        $\vv(29,2,12)$ & $(29,2)$ & $\{h_0^4\vv(28,0,12),\vv(28,4,12)\}$ & ? & \\
        $k$ & $(29, 7)$ & $h_0d_0^2$ & Yes & \ref{prop: sphere differentials}\\
         $h_5$ & $(31, 1)$ & $h_0h_4^2$ & Yes & \ref{prop: sphere differentials}\\
        $l$ & $(32, 7)$ & $h_0d_0e_0$ & Yes & \ref{prop: sphere differentials}\\
    \end{longtable}
    \begin{proof}
    See Proposition \ref{prop: sphere differentials}. 
    \end{proof}

\subsubsection{$d_3$-differentials}
\begin{proposition}
Table \ref{tab:Adams d_3 mo12} describes the non-zero $d_3$-differentials in the Adams spectral sequence for $MO\langle 12 \rangle$  on all indecomposables on $E_3$ through stem 31.
\end{proposition}
The following table lists the only indecomposables through stem 31 that, for degree reasons, can support a $d_3$-differential. 
 \begin{longtable}{llllc}
    \caption[Possible $d_3$-differentials $MO\langle 12 \rangle$]{Possible non-zero $d_3$-differentials on indecomposable elements.
    \label{tab:Adams d_3 mo12}
    } \\
    \toprule
    $x$ & $(t-s,s)$ & $d_3(x)$ & Occurs & Proof\\
    \midrule \endfirsthead
    \caption[]{Possible non-zero $d_3$-differentials on indecomposable elements} \\
    \toprule
    $x$ & $(t-s,s)$ & $d_3(x)$ & Occurs & Proof\\
    \midrule \endhead
    \bottomrule \endfoot
        $h_0h_4$ & $(15,2)$ & $h_0d_0$ & Yes & \ref{prop: sphere differentials}\\
        $\vv(16,1,12)$ & $(16,1)$ & $h_0^3h_4$ & Yes & \ref{d_3 differential in stem 16 filtration 1 mo12} \\
        $h_1h_4$ & $(16,2)$ & $h_1d_0$ & No & \ref{prop: sphere differentials}\\
        $h_2h_4$ & $(31,4)$ & $h_0e_0$ & No & \ref{prop: sphere differentials}\\
        $\vv(24,1,12)$ & $(24,1)$ & $c_0h_4$ & No & \ref{d_3 differential in stem 24 filtration 1 mo12} \\
        $\Delta h_2^2$ & $(30,6)$ & $h_0^2k$ & Yes & \ref{prop: sphere differentials}\\
        $h_0^3h_5$ & $(31,4)$ & $h_0\Delta h_2^2$ & Yes & \ref{prop: sphere differentials} \\
    \end{longtable}
    \begin{proof}
        For proofs of these differentials, see Proposition \ref{prop: sphere differentials} and the following series of Lemmas. 
    \end{proof}
\begin{lemma}\label{d_3 differential in stem 16 filtration 1 mo12}
$d_3(\vv(16,1,12))=h_0^3h_4$
\end{lemma}
\begin{proof}
By Theorem \ref{mahowaldimj} $h_0^3h_4$ is in the image of $J$. By Theorem \ref{hoveyimj}, $h_0^3h_4$ detects an element in the kernel of the unit map $\mathbb{S} \rightarrow MO\langle 12 \rangle$, so it must be killed in the Adams spectral sequence for $MO\langle 12\rangle$. 
The only possible source of a differential with target $h_0^3h_4$ is $\vv(16,1,12)$, so $d_3(\vv(16,1,12))=h_0^3h_4$
\end{proof}

\begin{lemma}\label{d_3 differential in stem 24 filtration 1 mo12}
$d_3(\vv(24,1,12))=0$
\end{lemma}
\begin{proof}
The only possible target of a $d_3$-differential with source $\vv(24,1,12)$ is $c_0h_4$. By Theorem 1.7 of Senger-Zhang \cite{SengerZhang2025Inertia}, the kernel of the unit map $\pi_{24}\mathbb{S} \rightarrow \pi_{24}MO\langle 12 \rangle$ is generated by the image of $J$. 
The element $c_0h_4$ detects an element that is not in the image of $J$. Thus, $d_3(\vv(24,1,12))=0$. 
\end{proof}
\subsubsection{$d_4$-differentials}
\begin{proposition}
Table \ref{tab:Adams d_4 mo12} describes the non-zero $d_4$-differentials in the Adams spectral sequence for $MO\langle 12 \rangle$  on all indecomposables on $E_4$ through stem 31, and select elements in stems $> 31$. 
\end{proposition}
The following table lists the only indecomposables through stem 31 that, for degree reasons, can support a $d_4$-differential. It also includes select differentials in stems $>31$. 
 \begin{longtable}{llllc}
    \caption[Possible $d_4$-differentials $MO\langle 12 \rangle$]{Possible non-zero $d_4$-differentials on indecomposable elements.
    \label{tab:Adams d_4 mo12}
    } \\
    \toprule
    $x$ & $(t-s,s)$ & $d_4(x)$ & Occurs & Proof\\
    \midrule \endfirsthead
    \caption[]{Possible non-zero $d_4$-differentials on indecomposable elements} \\
    \toprule
    $x$ & $(t-s,s)$ & $d_4(x)$ & Occurs & Proof\\
    \midrule \endhead
    \bottomrule \endfoot
        $h_0e_0$ & $(17, 5)$ & $h_0^8\vv(16,1,12)$ & No & \ref{prop: sphere differentials}\\
        $\vv(24,5,12)$ & $(24,5)$ & $h_1Pd_0+h_0^2i$ & Yes & \ref{lem:d_4 differential in stem 24 filtration 5 mo12} \\
        $\vv(25,1,12)$ & $(25,1)$ & $\{h_1c_0h_4,h_0^4\vv(24,1,12)\}$ & No & \ref{d_4 differential in stem 25 filtration 1 mo12} \\
        $d_0e_0+h_0^7h_5$ & $(31,8)$ & $P^2d_0$ & Yes & \ref{prop: sphere differentials} \\
        $h_0^3 \Delta h_3^2$ & $(38,9)$ & $h_0^2d_0i$ & Yes & \ref{prop: sphere differentials}\\
    \end{longtable}
\begin{proof}
For proofs of these differentials, see the following series of Lemmas. 
\end{proof}
\begin{lemma}\label{lem:d_4 differential in stem 24 filtration 5 mo12}
$d_4(\vv(24,5,12))=h_1Pd_0+h_0^2i$
\end{lemma} 
\begin{proof}
This proof was generated by \textbf{sseqcpp}. The possible targets of a $d_4$-differential with source $\vv(24,5,12)$ are $h_1Pd_0$, $h_0^2i$, and $h_1Pd_0+h_0^2i$. Note that since there are two elements in bidegree $(23,9)$ that detect elements in the kernel of the unit map, and the element $\vv(24,5,12)$ is one of the two possible sources of a differential landing in that bidegree, we know that $d_4(\vv(24,5,12)) \neq 0$. 
Suppose $d_4(\vv(24,5,12))=h_1Pd_0$. Using the relation $d_0\vv(24,5,12)=0$ and the Leibniz rule, 
\begin{align*}
d_4(d_0\vv(24,5,12))&=d_4(0)\\
                    &= 0 \\
                    &=d_4(d_0)\vv(24,5,12)+d_0d_4(\vv(24,5,12))\\
                    &=d_0h_1Pd_0\\
                    &=h_1Pd_0^2
\end{align*}
where $h_1Pd_0^2$ is nonzero. This is a contradiction.\\
Suppose $d_4(\vv(24,5,12))=h_0^2i$. Again using the relation $d_0\vv(24,5,12)=0$ and the Leibniz rule,
\begin{align*}
d_4(d_0\vv(24,5,12))&=d_4(0) \\
                    &= 0\\
                    &=d_4(d_0)\vv(24,5,12)+d_0d_4(\vv(24,5,12))\\
                    &=d_0h_0^2i
\end{align*}
where $d_0h_0^2i$ is nonzero. This is a contradiction. Hence, the only possibility is $d_4(\vv(24,5,12))=h_1Pd_0+h_0^2i$.
\end{proof}

\begin{lemma}\label{d_4 differential in stem 25 filtration 1 mo12}
$d_4(\vv(25,1,12))=0$
\end{lemma}
\begin{proof}
The possible targets of a $d_4$-differential with source $\vv(25,1,12)$ are $\vv(24,5,12)$, $h_0^4\vv(24,1,12)$, and $h_1c_0h_4$. Note that $\vv(25,1,12)$ is $h_0$-torsion while $\vv(24,5,12)$ and $h_0^4\vv(24,1,12)$ are $h_0$-power torsion free. The $d_4$-differential is $h_0$-linear, so the first two possibilities are ruled out. By Theorem 1.7 of Senger-Zhang \cite{SengerZhang2025Inertia}, the kernel of the unit map $\pi_{24}\mathbb{S} \rightarrow \pi_{24}MO\langle 12 \rangle$ is generated by the image of $J$. 
The element $h_1c_0h_4$ detects an element that is not in the image of $J$. Thus, $d_4(\vv(25,1,12))=0$. 
\end{proof}
\subsubsection{$d_5$-differentials}
\begin{proposition}
Table \ref{tab:Adams d_5 mo12} describes the non-zero $d_5$-differentials in the Adams spectral sequence for $MO\langle 12 \rangle$  on all indecomposables on $E_5$ through stem 31. 
\end{proposition}
The following table lists the only indecomposables through stem 31 that, for degree reasons, can support a $d_5$-differential. 
 \begin{longtable}{llllc}
    \caption[Possible $d_5$-differentials $MO\langle 12 \rangle$]{Possible non-zero $d_5$-differentials on indecomposable elements.
    \label{tab:Adams d_5 mo12}
    } \\
    \toprule
    $x$ & $(t-s,s)$ & $d_5(x)$ & Occurs & Proof\\
    \midrule \endfirsthead
    \caption[]{Possible non-zero $d_5$-differentials on indecomposable elements} \\
    \toprule
    $x$ & $(t-s,s)$ & $d_5(x)$ & Occurs & Proof\\
    \midrule \endhead
    \bottomrule \endfoot
        $\vv(12,0,12)$ & $(12, 0)$ & $Ph_2$ & Yes & \ref{lem:d_5-differentials-imj-mo12}\\
        $h_1\vv(16,1,12)$ & $(17, 2)$ & $Pc_0$ & Yes & \ref{lem:d_5-differentials-imj-mo12}\\
        $g$ & $(20, 4)$ & $P^2h_2$ & No & \\
        $\vv(20,4,12)$ & $(20, 4)$ & $P^2h_2$ & Yes & \ref{lem:d_5-differentials-imj-mo12}\\
         $h_4\vv(12,0,12)$ & $(27,1)$ & $h_2^2g$ & ? &  \\
        $\vv(28,8,12)$ & $(28, 8)$ & $P^3h_2$ & Yes & \ref{lem:d_5 differential in stem 28 filtration 8 mo12} \\
    \end{longtable}
\begin{proof}
The differential on $g$ is zero by comparison with $\mathbb{S}$. For reference, see Theorem 11.52 of Bruner-Rognes\cite{BrunerRognestmf}. For proofs of the other differentials, see the following series of Lemmas. 
\end{proof}

\begin{lemma}\label{lem:d_5-differentials-imj-mo12}
    \begin{enumerate}
    \item $d_5(\vv(12,0,12))=Ph_2$. 
    \item $d_5(h_1\vv(16,1,12))=Pc_0$. 
    \item $d_5(\vv(20,4,12))=P^2h_2$.
    \end{enumerate}
\end{lemma}
\begin{proof}
By Theorem \ref{mahowaldimj}, the target of each of these differentials lie in the image of $J$. By Theorem \ref{hoveyimj}, they detect elements in the kernel of the unit map $\mathbb{S}\rightarrow MO\langle 12\rangle$.
On the $E_5$-page, the first two differentials are the only possibilities with their respective targets. The other possible source of a differential with target  $P^2h_2$ is $g$, but $d_5(g)=0$. 
Thus, $d_5(\vv(20,4,12))=P^2h_2$.
\end{proof}

\begin{lemma}\label{lem:d_5 differential in stem 28 filtration 8 mo12}
$d_5(\vv(28,8,12))=P^3h_2$.
\end{lemma}
\begin{proof}
The element $\vv(28,8,12)$ is contained in the Massey product $\langle h_0, h_0^3h_3, \vv(20,4,12)\rangle$ with indeterminacy $\{h_0^4\vv(28,4,12),h_0^8\vv(28,0,12)\}$. By Moss's Higher Leibniz rule \ref{mossleibniz},
\begin{align*}
d_5(\langle h_0, h_0^3h_3, \vv(20,4,12)\rangle) \in\;&
\langle d_5(h_0),h_0^3h_3,\vv(20,4,12)\rangle+\langle h_0,d_5(h_0^3h_3),\vv(20,4,12)\rangle+\langle h_0,h_0^3h_3,d_5(\vv(20,4,12))\rangle\\
=\;&\langle 0,h_3,\vv(20,4,12)\rangle+\langle h_0,0,\vv(20,4,12)\rangle+\langle h_0,h_0^3h_3,P^2h_2\rangle.
\end{align*}
The first two terms vanish with zero indeterminacy and the third contains $P^3h_2$ with zero indeterminacy. Observe that, for degree reasons, the differentials $d_5(h_0^3\vv(28,4,12))$ and $d_5(h_0^7\vv(28,0,12))$ equal zero. By $h_0$-linearity of the $d_5$-differential and the Leibniz rule, this implies $d_5(h_0^4\vv(28,4,12))$ and  $d_5(h_0^8\vv(28,0,12))$ also equal zero. 
Hence, $d_5(\langle h_0, h_0^3h_3, \vv(20,4,12)\rangle)=d_5(\vv(28,8,12))=P^3h_2$.
\end{proof}

\begin{lemma}\label{d_8 differential in stem 24 filtration 1 mo12}
$d_8(\vv(24,1,12))=h_0^2i$.
\end{lemma}
\begin{proof}
By Theorem 4.1 of Burklund-Senger \cite{BurklundSenger2024Geography}, the element $\eta^3\overline{\kappa}$ is contained in the kernel of the unit map $\pi_{23} \mathbb{S} \rightarrow \pi_{23} MO\langle 12 \rangle$. The element $h_0^2i$ detects $\eta^3\overline{\kappa}$ and therefore must not survive the Adams spectral sequence for $MO \langle 12 \rangle$.
The only possible source of a differential is $\vv(24,1,12)$, and so $d_8(\vv(24,1,12))=h_0^2i$. 
\end{proof}

\begin{proposition}\label{prop: pi_29 mo12}
    $\pi_{29}MO\langle 12\rangle \otimes \mathbb{Q} \cong 0$
\end{proposition}
\begin{proof}
There is a stable equivalence $\mathbb{S}_{\mathbb{Q}} \cong H\mathbb{Q}$, so 
\[
\pi_{29}MO\langle 12\rangle \otimes \mathbb{Q} \cong H_{29}(MO\langle 12 \rangle; \mathbb{Q}) \cong H^{29}(MO \langle 12 \rangle; \mathbb{Q})
\]   
The group $H^{29}(MO \langle 12 \rangle; \mathbb{Q})$ is trivial by Proposition \ref{prop:rational bo12} and the Thom isomorphism. Hence, $\pi_{29}MO\langle 12\rangle \otimes \mathbb{Q} \cong 0$. 
\end{proof}

\begin{lemma}\label{lem:stem 29 differential mo12}
The differential $d_r(\vv(29,2,12))$ is nonzero for some $r\geq 2$. 
\end{lemma}
\begin{proof}
The element $\vv(29,2,12)$ generates an $h_0$-tower in stem 29. By Proposition \ref{prop: pi_29 mo12}, this $h_0$-tower cannot survive to give a $\mathbb{Z}$ summand in homotopy. By the $h_0$-linearity of the $d_r$-differential, this implies
$\vv(29,2,12)$ must support a nonzero differential. 
\end{proof}

\subsubsection{The abutment}
Recall that there are no hidden $2$-extensions between elements in the image of the unit map in this range of degrees.
\begin{proposition}
$\pi_{16}MO\langle 12 \rangle \cong \Z \oplus \Z/2$.
\end{proposition}
\begin{proof}
The proof is identical to the proof of Proposition \ref{prop:extension problem stem 16 mo9}. 
\end{proof}

\begin{proposition}\label{prop:stem 20 extension}
$\pi_{20} MO \langle 12 \rangle \cong \Z \oplus \Z/8$.
\end{proposition}
\begin{proof}
Note that $h_0^2g$ detects $4\overline{\kappa} \in \pi_{20} \mathbb{S}$. The group $\pi_{20}(MO \langle 12 \rangle)$ lies in the short exact sequence
\begin{align*}
0 \rightarrow \Z \rightarrow \pi_{20}(MO \langle 12 \rangle) \rightarrow \mathbb{Z}/8\{\overline{\kappa}\} \rightarrow 0 
\end{align*}
Suppose this short exact sequence did not split. This would imply $8\overline{\kappa} \neq 0$, which is a contradiction since $8\overline{\kappa} = 0 \in \pi_{20} \mathbb{S}$. 
\end{proof}

\begin{proposition}\label{prop:stem 24 extension}
$\pi_{24}MO\langle 12 \rangle \cong \Z \oplus \Z \oplus \Z/2$.
\end{proposition}
\begin{proof}
Note that $h_1c_0h_4$ detects $\epsilon\eta^{*} \in \pi_{24} \mathbb{S}$. The group $\pi_{24}(MO \langle 12 \rangle)$ lies in the short exact sequence
\begin{align*}
0 \rightarrow \Z \oplus \Z \rightarrow \pi_{24}(MO \langle 12 \rangle) \rightarrow \mathbb{Z}/2\{\epsilon\eta^{*}\} \rightarrow 0 
\end{align*}
Suppose this short exact sequence did not split. This would imply $2\epsilon\eta^{*} \neq 0$, which is a contradiction since $2\epsilon\eta^{*} = 0 \in \pi_{24} \mathbb{S}$. 
\end{proof}

\begin{proposition}
$\pi_{28} MO\langle 12 \rangle \cong \Z \oplus \Z \oplus \mathbb{Z}/2^{n} \oplus \mathbb{Z}/2$ for some $n>2$ and $\pi_{29} MO\langle 12 \rangle \cong 0$.
\end{proposition}
\begin{proof}
While we technically do not have a complete calculation of the $E_\infty$-page in these stems, there is enough information to deduce these groups. The nonzero differential in stem 29 established in Lemma \ref{lem:stem 29 differential mo12} kills the $h_0$-tower in that stem. There are no other elements in that stem and so $\pi_{29}MO\langle 12 \rangle=0$. 
The target of the $d_r$-differential in Lemma \ref{lem:stem 29 differential mo12} is a linear combination of the $h_0$-towers in stem 28. Depending on what the linear combination is, and the length of the differential, there will be a $\mathbb{Z}/2^n$ left over. Hence, $\pi_{28} MO\langle 12 \rangle \cong \Z \oplus \Z \oplus \mathbb{Z}/2^{n} \oplus \mathbb{Z}/2$ for some $n$.
\end{proof}

\begin{theorem}\label{thm: mo12 groups}
Table \ref{tab:mo12 groups} describes the 2-primary component of the cobordism groups $\Omega^{\langle 12 \rangle}_n$ for the values of $n$ listed. 
\end{theorem}
\begin{center}
\captionof{table}[The 2-primary component of the $O\langle 12 \rangle$-cobordism groups]{ $\Omega^{\langle 12 \rangle}_n$ through degree $30$.}
\label{tab:mo12 groups}
\tablehead{\hline%
$n$ & $\Omega^{\langle 12 \rangle}_n$ & $n$ & $\Omega^{\langle 12 \rangle}_n$\\%
\hline}
\tabletail{\hline%
\multicolumn{4}{r}{%
\small\slshape to be continued on the next page}\\}
\tablelasttail{\hline}
\begin{supertabular}{|c|p{3.5cm}|c|p{3.5cm}|}
$0$  & $\Z$ & $16$ & $\Z\oplus \Z/2$\\
$1$  & $\Z/2$ & $17$ & $(\Z/2)^3$\\
$2$  & $\Z/2$ & $18$ & $\Z/8 \oplus \Z/2$\\
$3$  & $\Z/8$ & $19$ & $\Z/2$\\
$4$  & $(0)$ & $20$ & $\Z/8 \oplus \Z$\\
$5$  & $(0)$ & $21$ & $\Z/2 \oplus \Z/2$\\
$6$  & $\Z/2$ & $22$ & $\Z/2 \oplus \Z/2$\\
$7$  & $\Z/16$ & $23$ & $\Z/2 \oplus \Z/4$\\
$8$  & $\Z/2 \oplus \Z/2$ & $24$ & $\Z/2 \oplus (\Z)^2$\\
$9$  & $(\Z/2)^{3}$ & $25$ & 2|4 \\
$10$ & $\Z/2$ & $26$ & \\
$11$ & $(0)$ & $27$ & \\
$12$ & $\Z$ & $28$ & $\Z \oplus \Z \oplus \mathbb{Z}/2^{n} \oplus \mathbb{Z}/2$ \\
$13$ & $(0)$ & $29$ & $(0)$\\
$14$ & $\Z/2 \oplus \Z/2$ & $30$ & $\Z/2 \oplus \Z/2\text{ or } \Z/4$\\
$15$ & $\Z$ & & \\
\end{supertabular}
\end{center}
\subsection{$MO\langle 16 \rangle$}
\subsubsection{The cohomology of $MO\langle 16 \rangle$}

\begin{proposition}[\cite{Stong1963Determination}]
    $H^{*}(BO\langle16 \rangle;\mathbb{Z}/2)$ is isomorphic to:
    \[
        H^{*}(K(\mathbb{Z}/2,16);\Z/2)/\langle Sq^2 i_{16} \rangle \otimes \mathbb{Z}/2[\theta_{i} | L(i) > 7]
    \]
\end{proposition}
where $L(i)$ is one plus the number of ones in the binary expansion of $i-1$. We calculate the $\mathcal{A}$-module structure of $H^{*}(MO\langle16 \rangle;\mathbb{Z}/2)$ through degree $70$ using a computer program in Sage.

\subsubsection{Adams charts}
This section contains charts for the Adams spectral sequence for $MO\langle 16 \rangle$. The $E_r$-page is displayed along with the $d_r$-differentials. The $y$-axis is the Adams filtration $s$ and the $x$ axis is the stem $t-s$. The elements $h_0$, $h_1$, and $h_2$ have had their names suppressed. The dotted lines on the $E_{\infty}$-page represent possibly nonzero differentials that have yet to be determined. Classes in red on the $E_{\infty}$-page do not survive. 
\clearpage
\FloatBarrier
\begin{figure}[!p]
\centering

\begin{subfigure}{\linewidth}
  \centering
  \includegraphics[width=\linewidth,
    trim=0.5cm 0cm 0cm 0cm,clip]{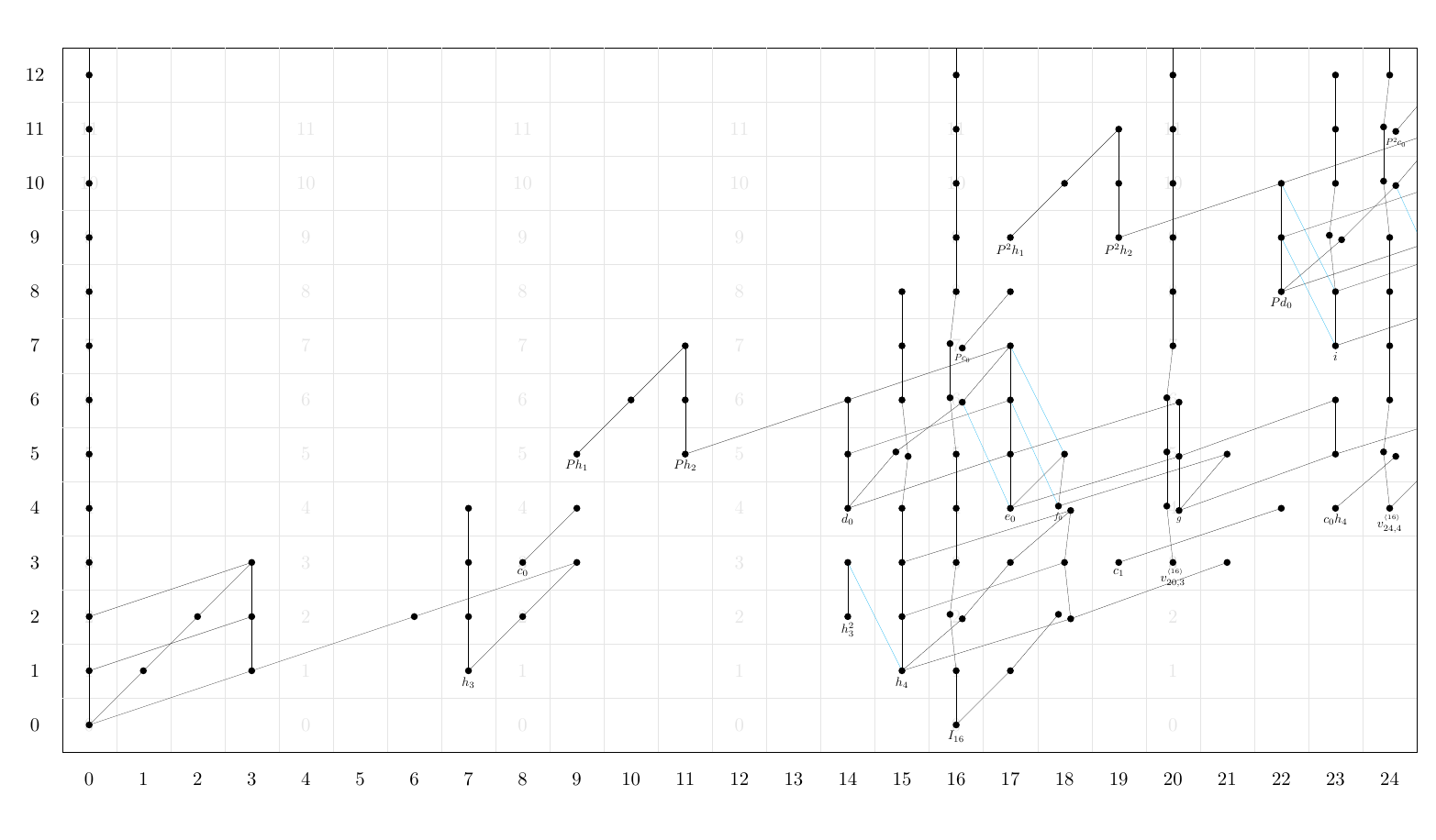}
  \caption{$E_2$-page $MO \langle 16 \rangle $ (stems 0-24)}
  \label{fig:mo16_e2_0_24}
\end{subfigure}

\vspace{-0.4\baselineskip}

\begin{subfigure}{\linewidth}
  \centering
  \includegraphics[height=14cm,width=\linewidth,
    trim=0.5cm 0cm 0cm 0cm,clip]{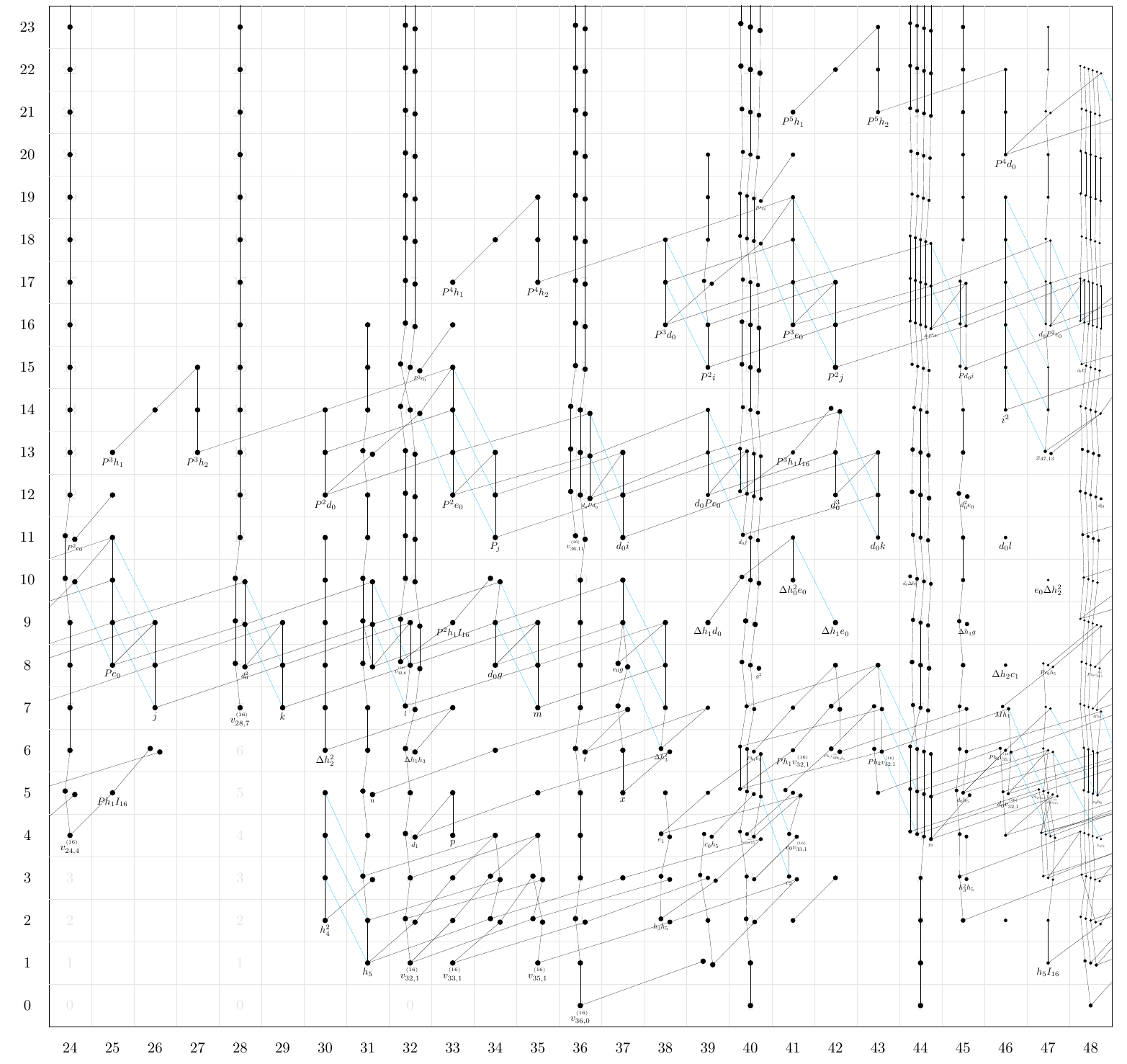}
  \caption{$E_2$-page $MO \langle 16 \rangle $ (stems 24-48)}
  \label{fig:mo16_e2_24_48}
\end{subfigure}

\caption{$E_2$-page $MO\langle 16 \rangle$ \label{fig:mo16_e2}}
\end{figure}

\begin{figure}[!p]
\centering
\begin{subfigure}{\linewidth}
  \centering
  \includegraphics[width=\linewidth,
    trim=0.5cm 0cm 0cm 0cm,clip]{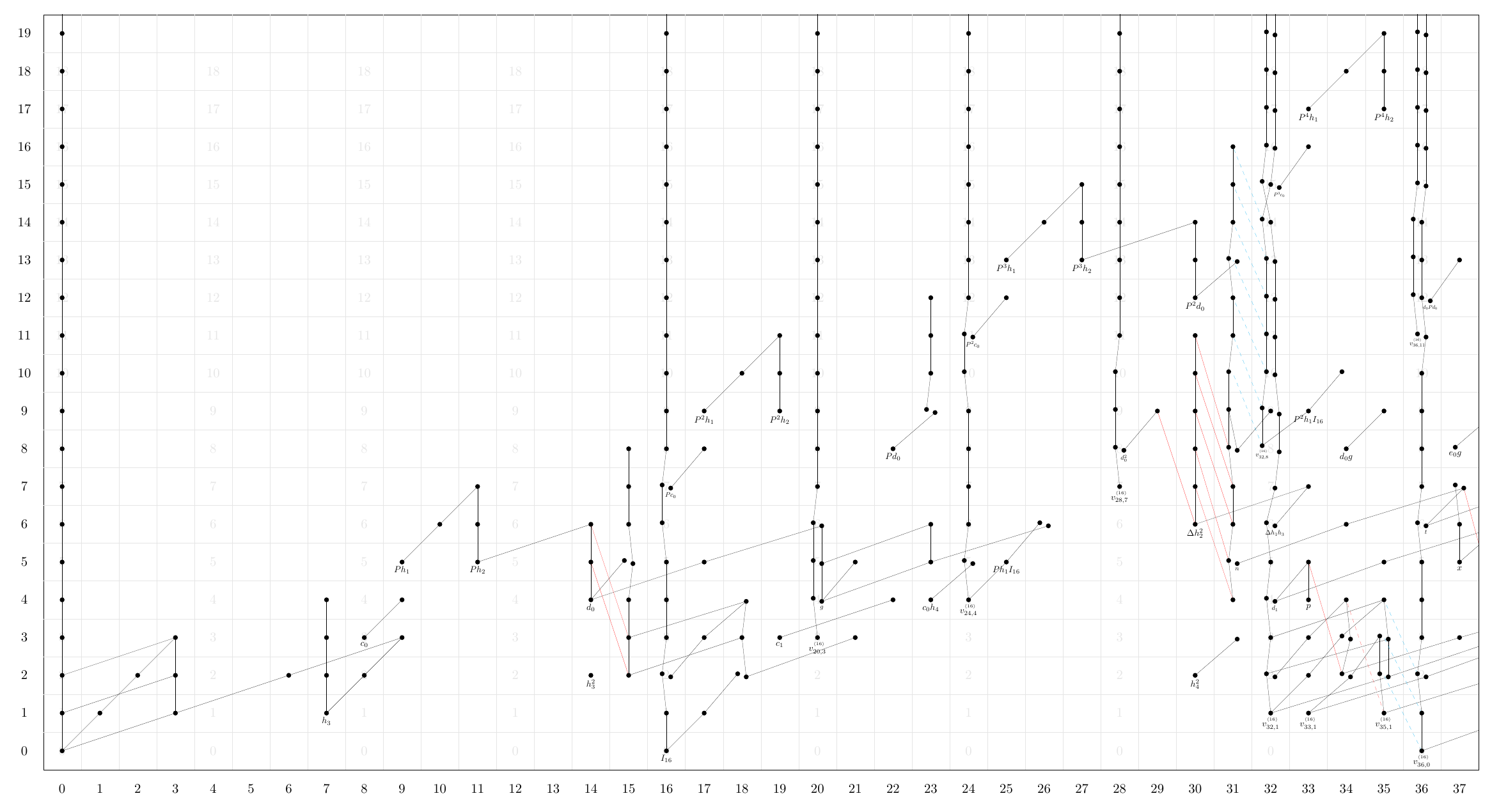}
  \caption{$E_3$-page $MO \langle 16 \rangle $ (stems 0-37)}
  \label{fig:mo16_e3}
\end{subfigure}

\vspace{-0.4\baselineskip}

\begin{subfigure}{\linewidth}
  \centering
  \includegraphics[width=\linewidth,
    trim=0.5cm 0cm 0cm 0cm,clip]{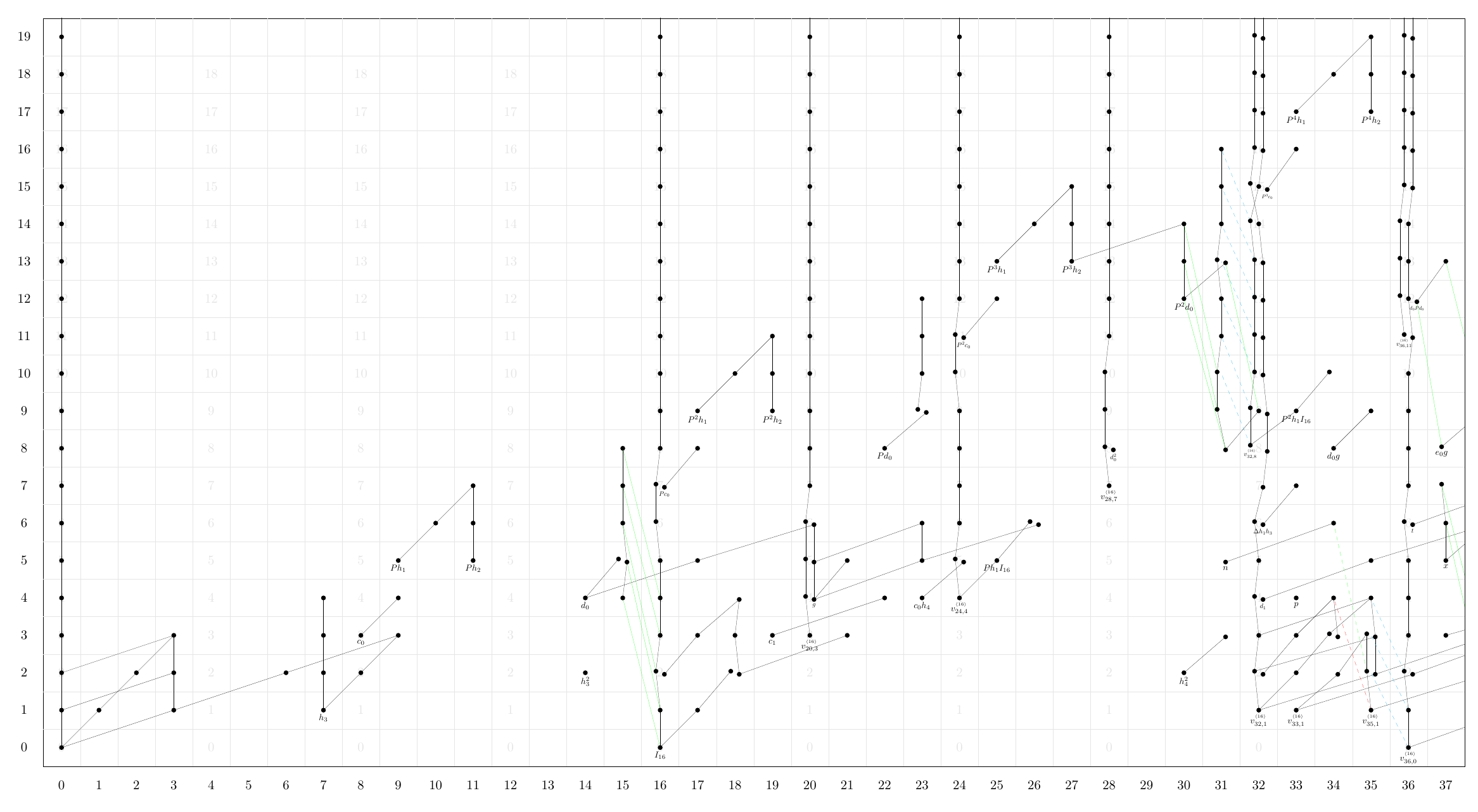}
  \caption{$E_4$-page $MO \langle 16 \rangle $ (stems 0-37)}
  \label{fig:mo16_e4}
\end{subfigure}

\caption{$E_3$/$E_4$-page $MO\langle 16 \rangle$ \label{fig:mo16_e3_e4}}
\end{figure}

\begin{figure}[!p]
\centering
\begin{subfigure}{\linewidth}
  \centering
  \includegraphics[width=\linewidth,
    trim=0.5cm 0cm 0cm 0cm,clip]{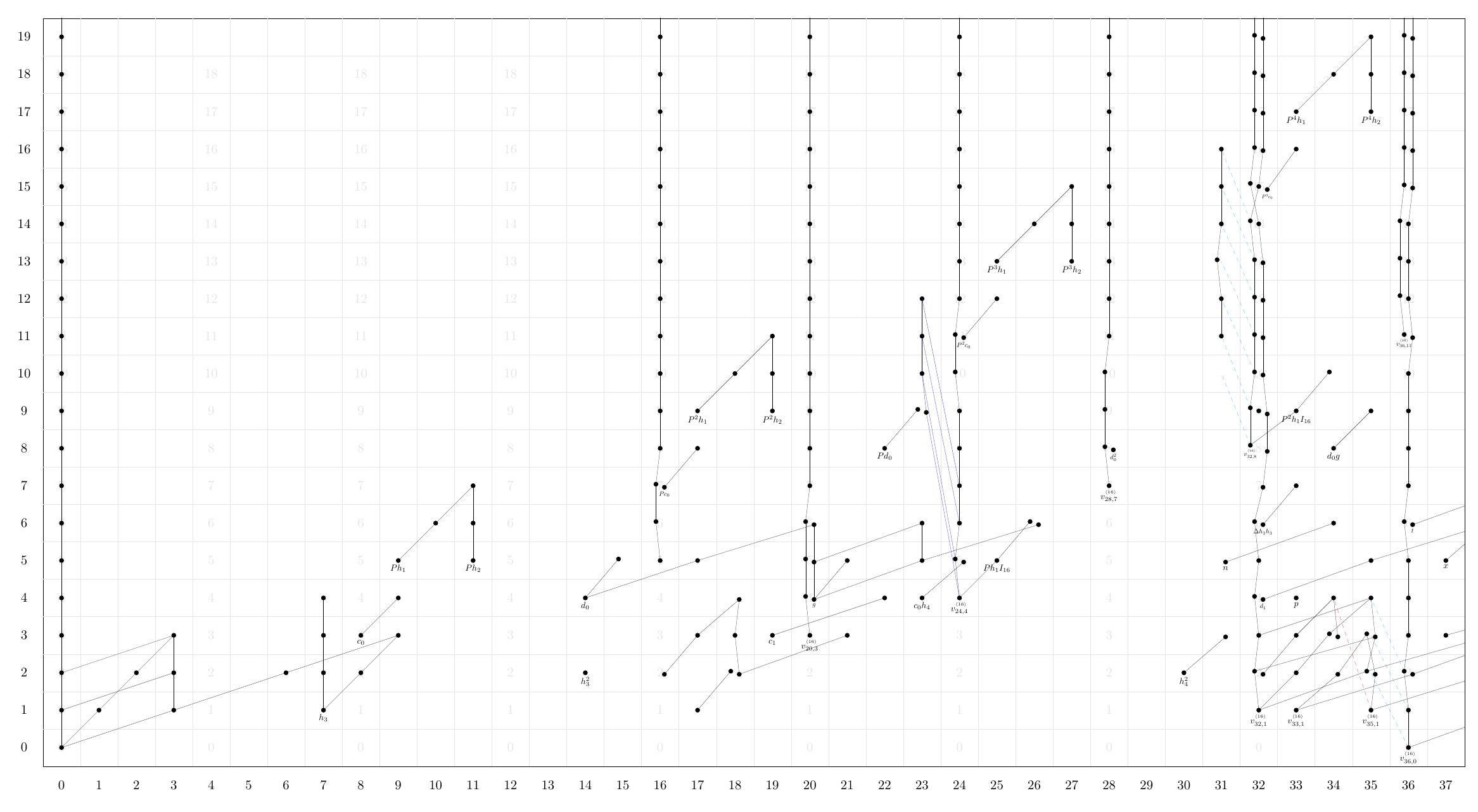}
  \caption{$E_5$-page $MO \langle 16 \rangle $ (stems 0-37)}
  \label{fig:mo16_e5}
\end{subfigure}

\vspace{-0.4\baselineskip}

\begin{subfigure}{\linewidth}
  \centering
  \includegraphics[width=\linewidth,
    trim=0.5cm 0cm 0cm 0cm,clip]{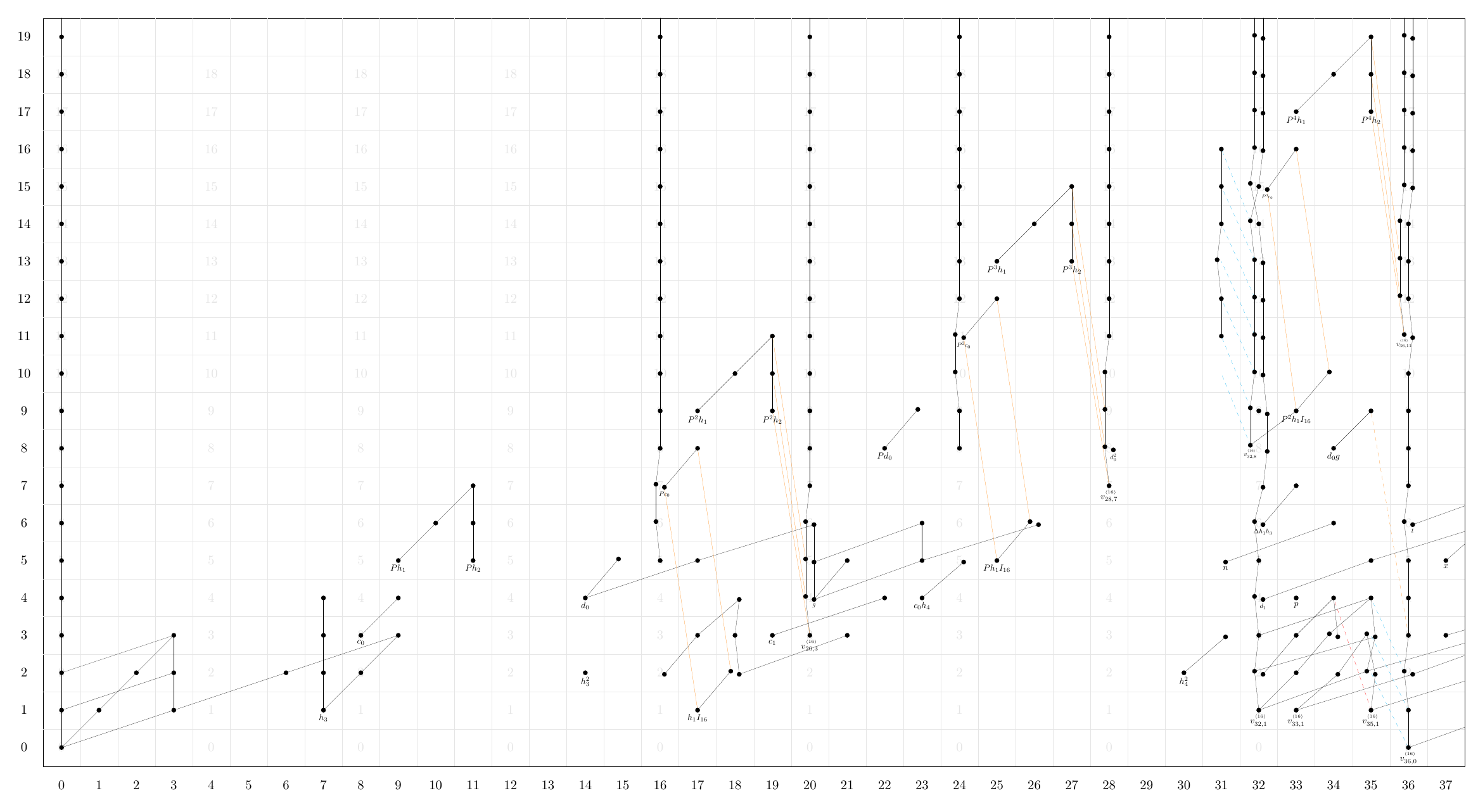}
  \caption{$E_{6}$-page $MO \langle 16 \rangle $ (stems 0-37)}
  \label{fig:mo16_e6}
\end{subfigure}

\caption{$E_5$/$E_6$-page $MO\langle 16 \rangle$ \label{fig:mo16_e5_e6}}
\label{fig:mo16_e5_e6}
\end{figure}

\begin{figure}
  \centering
  \includegraphics[width=\linewidth,
    trim=0.5cm 0cm 0cm 0cm,clip]{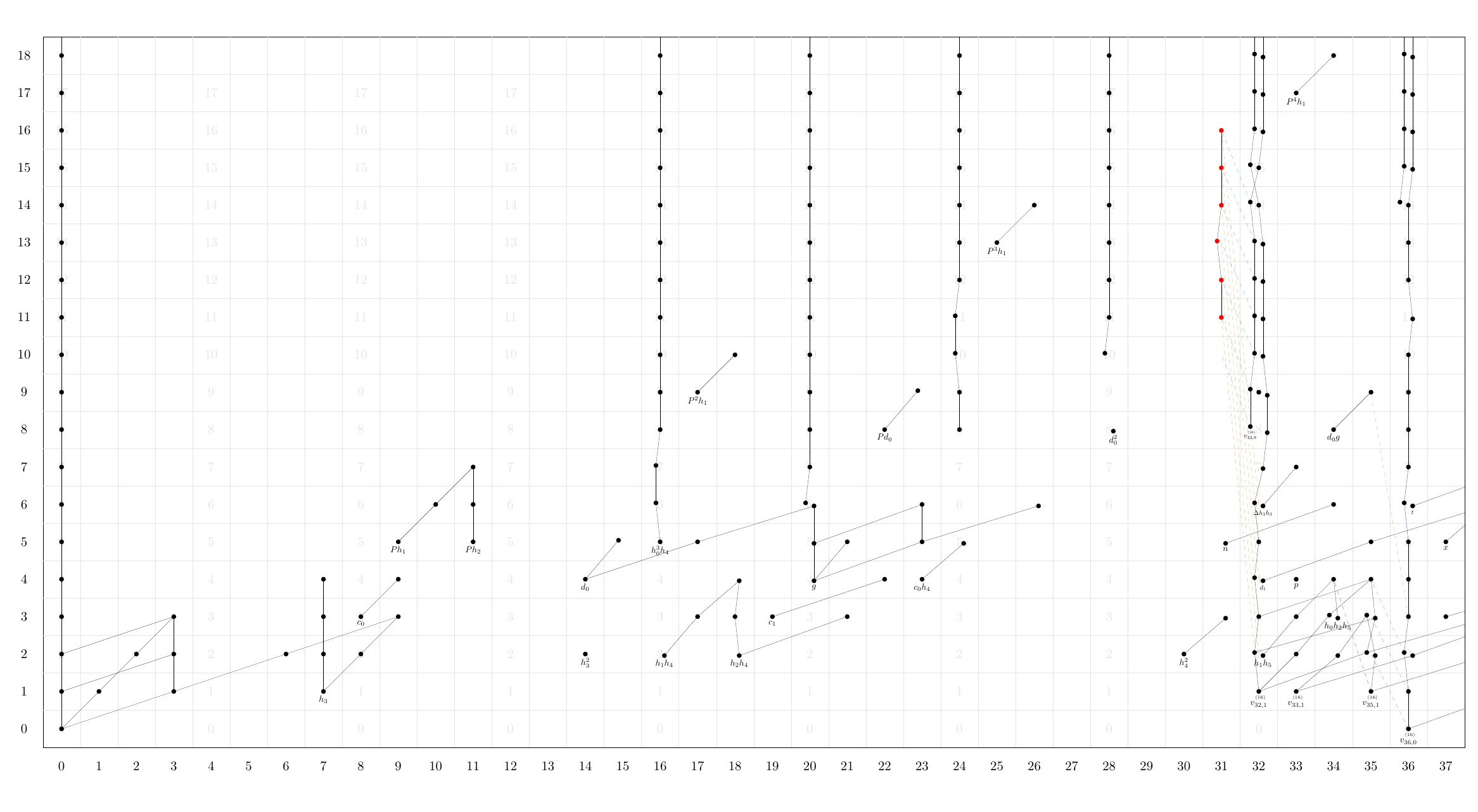}
  \caption{$E_{\infty}$-page $MO\langle 16 \rangle$ \label{fig:mo16_einf}}
\end{figure}

\FloatBarrier

\subsubsection{$d_2$-differentials}
\begin{proposition}
Table \ref{tab:Adams d_2 mo16} describes the non-zero $d_2$-differentials in the Adams spectral sequence for $MO\langle 16 \rangle$  on all indecomposables on $E_2$ through stem 37.
\end{proposition}
The following table lists the only indecomposables through stem 37 that, for degree reasons, can support a $d_2$-differential. 
 \begin{longtable}{llllc}
    \caption[Possible $d_2$-differentials $MO\langle 16 \rangle$]{Possible non-zero $d_2$-differentials on indecomposable elements.
    \label{tab:Adams d_2 mo16}
    } \\
    \toprule
    $x$ & $(t-s,s)$ & $d_2(x)$ & Occurs & Proof\\
    \midrule \endfirsthead
    \caption[]{Possible non-zero $d_2$-differentials on indecomposable elements} \\
    \toprule
    $x$ & $(t-s,s)$ & $d_2(x)$ & Occurs & Proof\\
    \midrule \endhead
    \bottomrule \endfoot
        $h_1$ & $(1, 1)$ & $h_0^3$ & No & \ref{prop: sphere differentials}\\
        $h_4$ & $(15, 1)$ & $h_0h_3^2$ & Yes & \ref{prop: sphere differentials}\\
        $e_0$ & $(15, 4)$ & $h_1^2d_0$ & Yes & \ref{prop: sphere differentials}\\
        $f_0$ & $(18, 4)$ & $h_0^2e_0$ & Yes & \ref{prop: sphere differentials}\\
        $i$ & $(23, 7)$ & $h_0Pd_0$ & Yes & \ref{prop: sphere differentials}\\
         $\vv(24,4,16)$ & $(24, 4)$ & $h_0h_2g$ & No & \ref{lem:d_2 differential in stem 24 filtration 4 mo16}\\
        $j$ & $(26, 7)$ & $h_0Pe_0$ & Yes & \ref{prop: sphere differentials}\\
        $k$ & $(29, 7)$ & $h_0d_0^2$ & Yes & \ref{prop: sphere differentials}\\
        $h_5$ & $(31, 1)$ & $h_0h_4^2$ & Yes & \ref{prop: sphere differentials}\\
        $l$ & $(32, 7)$ & $h_0d_0e_0$ & Yes & \ref{prop: sphere differentials}\\
        $P_j$ & $(34, 11)$ & $h_0^2P^2e_0$ & Yes & \ref{prop: sphere differentials} \\
        $m$ & $(35, 7)$ & $h_0d_0g$ & Yes & \ref{prop: sphere differentials} \\
        $\vv(36,0,16)$ & $(36, 0)$ & $\{h_2\vv(32,1,16),h_0\vv(35,1,16)\}$ & ? & \\
        $h_2^2h_5$ & $(37, 3)$ & $\{t, h_0^6 \vv(36,0,16)\}$ & No & \ref{prop: sphere differentials} \\
    \end{longtable}
\begin{proof}
For proof of the differential on $\vv(24,4,16)$, see Lemma \ref{lem:d_2 differential in stem 24 filtration 4 mo16}. For proofs of the other differentials, see Proposition \ref{prop: sphere differentials}. 
\end{proof}

\begin{lemma}\label{lem:d_2 differential in stem 24 filtration 4 mo16}
 $d_2(\vv(24,4,16))=0$.
\end{lemma}
\begin{proof}
This proof was generated by \textbf{sseqcpp}. The only possible target of $d_2(\vv(24,4,16))$ is $h_0h_2g$. Suppose  $d_2(\vv(24,4,16))=h_0h_2g$. The are relations $d_0\vv(24,4,16)=0$ and $d_0h_0h_2g=h_0^5x$. Apply the Leibniz rule
\begin{align*}
d_2(d_0\vv(24,4,16)) &= 0\\
                     &=d_0d_2(\vv(24,4,16))+d_2(d_0)\vv(24,4,16) \\
                     &=d_0h_0h_2g = h_0^5x
\end{align*}
which is a contradiction. Hence,  $d_2(\vv(24,4,16))=0$.
\end{proof}
\subsubsection{$d_3$-differentials}
\begin{proposition}
Table \ref{tab:Adams d_3 mo16} describes the non-zero $d_3$-differentials in the Adams spectral sequence for $MO\langle 16 \rangle$  on all indecomposables on $E_3$ through stem 37, and on select elements in stems $> 37$.
\end{proposition}
The following table lists the only indecomposables through stem 37 that, for degree reasons, can support a $d_3$-differential. It also includes differentials on select elements in stems $>37$. 
 \begin{longtable}{llllc}
    \caption[Possible $d_3$-differentials $MO\langle 16 \rangle$]{Possible non-zero $d_3$-differentials on indecomposable elements
    \label{tab:Adams d_3 mo16}
    } \\
    \toprule
    $x$ & $(t-s,s)$ & $d_3(x)$ & Occurs & Proof\\
    \midrule \endfirsthead
    \caption[]{Possible non-zero $d_3$-differentials on indecomposable elements} \\
    \toprule
    $x$ & $(t-s,s)$ & $d_3(x)$ & Occurs & Proof\\
    \midrule \endhead
    \bottomrule \endfoot
        $h_0h_4$ & $(15,2)$ & $h_0d_0$ & Yes & \ref{prop: sphere differentials}\\
        $h_1h_4$ & $(16,2)$ & $h_1d_0$ & No & \ref{prop: sphere differentials}\\
        $h_2h_4$ & $(31,4)$ & $h_0e_0$ & No & \ref{prop: sphere differentials}\\
        $\Delta h_2^2$ & $(30,6)$ & $h_0^2k$ & Yes & \ref{prop: sphere differentials}\\
        $h_0^3h_5$ & $(31,4)$ & $h_0\Delta h_2^2$ & Yes & \ref{prop: sphere differentials}\\
        $\vv(33,1,16)$ & $(33,1)$ & $\{d_1,h_0^3\vv(32,1,16)\}$ & No & \\
        $\vv(35,1,16)$ & $(35,1)$ & $h_0^2h_2h_5$ & ? & \\
        $e_1$ & $(38,4)$ & $h_1t$ & Yes & \ref{prop: sphere differentials}\\
    \end{longtable}
    \begin{proof}
        For proofs of these differentials, see Proposition \ref{prop: sphere differentials}. 
    \end{proof}

\begin{lemma}\label{lem:d_3 differential in stem 33 filtration 1 mo16}
$d_3(\vv(33,1,16))=0$.
\end{lemma}
\begin{proof}
The possible targets of a $d_3$-differential with source $\vv(33,1,16)$ are $d_1$ and $h_0^3\vv(32,1,16)$. Observe that $h_0^4\vv(32,1,16) \neq 0$ while $h_0\vv(33,1,16)=0$. Since 
the $d_3$-differential is $h_0$-linear, this implies $d_3(\vv(33,1,16)) \neq h_0^3\vv(32,1,16)$. By Theorem \ref{mahowaldimj}, the element $d_1$ is in $\text{Coker } J$. By Theorem 
\ref{thm: 2n mo_n kernel}, $d_1$ does not detect an element in the kernel of the unit map $\pi_{32} \mathbb S \rightarrow \pi_{32} MO\langle 16 \rangle$. Thus, it must survive the Adams spectral sequence 
for $MO\langle 16 \rangle$. Hence, $d_3(\vv(33,1,16))\neq d_1$ and we conclude $d_3(\vv(33,1,16))=0$. 
\end{proof}
  
\subsubsection{$d_4$-differentials}
\begin{proposition}
Table \ref{tab:Adams d_4 mo16} describes the non-zero $d_4$-differentials in the Adams spectral sequence for $MO\langle 16 \rangle$  on all indecomposables on $E_4$ through stem 37, and on select elements in stems $> 37$. 
\end{proposition}
The following table lists the only indecomposables through stem 37 that, for degree reasons, can support a $d_4$-differential. It also includes select differentials in stems $> 37$.  
 \begin{longtable}{llllc}
    \caption[Possible $d_4$-differentials $MO\langle 16 \rangle$]{Possible non-zero $d_4$-differentials on indecomposable elements
    \label{tab:Adams d_4 mo16}
    } \\
    \toprule
    $x$ & $(t-s,s)$ & $d_4(x)$ & Occurs & Proof\\
    \midrule \endfirsthead
    \caption[]{Possible non-zero $d_4$-differentials on indecomposable elements} \\
    \toprule
    $x$ & $(t-s,s)$ & $d_4(x)$ & Occurs & Proof\\
    \midrule \endhead
    \bottomrule \endfoot
        $I_{16}$ & $(16,0)$ & $h_0^3h_4$ & Yes & \ref{lem:d_4 differential in stem 16 filtration 0 mo16} \\
        $h_0e_0$ & $(17, 5)$ & $h_0^9 I_{16}$ & No & \ref{prop: sphere differentials}\\
        $\vv(24,5,12)$ & $(24,5)$ & $h_1Pd_0+h_0^2i$ & Yes & \ref{lem:d_4 differential in stem 24 filtration 5 mo12} \\
        $\vv(25,1,12)$ & $(25,1)$ & $\{h_1c_0h_4,h_0^4\vv(24,1,12)\}$ & No & \ref{d_4 differential in stem 25 filtration 1 mo12} \\
        $d_0e_0+h_0^7h_5$ & $(31,8)$ & $P^2d_0$ & Yes & \ref{prop: sphere differentials} \\
        $e_0g$ & $(37,8)$ & $d_0Pd_0$ & Yes & \ref{prop: sphere differentials}\\
        $h_3h_5$ & $(38,2)$ & $h_0x$ & Yes &  \ref{prop: sphere differentials} \\
        $h_0^3 \Delta h_3^2$ & $(38,9)$ & $h_0^2d_0i$ & Yes & \ref{prop: sphere differentials}\\
    \end{longtable}
    \begin{proof}
        For proof of the differential on $I_{16}$, see Lemma \ref{lem:d_4 differential in stem 16 filtration 0 mo16}.  For the other differentials, see Proposition \ref{prop: sphere differentials}.
    \end{proof}
\begin{lemma}\label{lem:d_4 differential in stem 16 filtration 0 mo16}
$d_4(I_{16})=h_0^3h_4$.
\end{lemma}
\begin{proof}
The element $h_0^3h_4$ is in the $h_0$-tower that ends at the ``Adams edge'' in stem 15. By Theorem \ref{mahowaldimj}, it is in the image of $J$. By Theorem \ref{hoveyimj}, it must not survive the Adams spectral sequence for $MO \langle 16 \rangle$. 
It cannot support a differential, and the only possible differential with target $h_0^3h_4$ is $d_4(I_{16})$. Hence, $d_4(I_{16})=h_0^3h_4$.
\end{proof}

\subsubsection{$d_5$-differentials}
\begin{proposition}
Table \ref{tab:Adams d_5 mo16} describes the non-zero $d_5$-differentials in the Adams spectral sequence for $MO\langle 16 \rangle$  on all indecomposables on $E_5$ through stem 37. 
\end{proposition}
The following table lists the only indecomposables through stem 37 that, for degree reasons, can support a $d_5$-differential. 
 \begin{longtable}{llllc}
    \caption[Possible $d_5$-differentials $MO\langle 16 \rangle$]{Possible non-zero $d_5$-differentials on indecomposable elements.
    \label{tab:Adams d_5 mo16}
    } \\
    \toprule
    $x$ & $(t-s,s)$ & $d_5(x)$ & Occurs & Proof\\
    \midrule \endfirsthead
    \caption[]{Possible non-zero $d_6$-differentials on indecomposable elements} \\
    \toprule
    $x$ & $(t-s,s)$ & $d_5(x)$ & Occurs & Proof\\
    \midrule \endhead
    \bottomrule \endfoot
        $\vv(24,4,16)$ & $(24, 4)$ & $\{h_0^2i,h_1Pd_0\}$ & Yes & \ref{lem:d_5 differential in stem 24 filtration 4 mo16}\\
        $\vv(33,1,16)$ & $(33, 1)$ & $\{\Delta h_1 h_3,h_0^5\vv(32,1,16)\}$ & No & \\
    \end{longtable}
\begin{proof}
The proof of $d_5(\vv(33,1,16))=0$ is identical to the proof of Lemma \ref{lem:d_3 differential in stem 33 filtration 1 mo16}. 
\end{proof}

\begin{lemma}\label{lem:d_5 differential in stem 24 filtration 4 mo16}
$d_5(\vv(24,4,16))=h_0^2i$
\end{lemma}
\begin{proof}
By Theorem \ref{mahowaldimj}, the element $h_0^2i$ is in $\text{Im } J$. By Theorem \ref{hoveyimj}, they are in the kernel of the unit map $\pi_{*} \mathbb S \rightarrow MO \langle 16 \rangle$ and so they must 
not survive the Adams spectral sequence for $MO\langle 16 \rangle$. The element $h_1c_0h_4$ does not support a nonzero differential by Proposition \ref{prop: sphere differentials}, so the only possible 
source of a differential with target $h_0^2i$ is $\vv(24,4,16)$. Hence, $d_5(\vv(24,4,16))=h_0^2i$. 
\end{proof}

\subsubsection{$d_6$-differentials}
\begin{proposition}
Table \ref{tab:Adams d_6 mo16} describes the non-zero $d_6$-differentials in the Adams spectral sequence for $MO\langle 16 \rangle$  on all indecomposables on $E_6$ through stem 37.
\end{proposition}
The following table lists the only indecomposables through stem 37 that, for degree reasons, can support a $d_6$-differential. 
 \begin{longtable}{llllc}
    \caption[Possible $d_6$-differentials $MO\langle 16 \rangle$]{Possible non-zero $d_6$-differentials on indecomposable elements.
    \label{tab:Adams d_6 mo16}
    } \\
    \toprule
    $x$ & $(t-s,s)$ & $d_6(x)$ & Occurs & Proof\\
    \midrule \endfirsthead
    \caption[]{Possible non-zero $d_6$-differentials on indecomposable elements} \\
    \toprule
    $x$ & $(t-s,s)$ & $d_4(x)$ & Occurs & Proof\\
    \midrule \endhead
    \bottomrule \endfoot
        $h_1I_{16}$ & $(17, 1)$ & $Pc_0$ & Yes & \ref{lem:d_6-differentials-imj-mo16}\\
        $\vv(20,3,16)$ & $(20, 3)$ & $P^2h_2$ & Yes & \ref{lem:d_6-differentials-imj-mo16}\\
        $\vv(28,7,16)$ & $(28, 7)$ & $P^3h_2$ & Yes & \ref{lem:d_6-differentials-imj-mo16} \\
        $\vv(36,11,16)$ & $(36, 11)$ & $P^4h_2$ & Yes & \ref{lem:d_6 differential in stem 36 filtration 11 mo16} \\
    \end{longtable}

\begin{lemma}\label{lem:d_6-differentials-imj-mo16}
\begin{enumerate}
\item $d_6(\vv(20,3,16))=P^2h_2$.
\item $d_6(\vv(28,7,16))=P^3h_2$.
\end{enumerate}
\begin{proof}
The target of each of these differentials is in $\text{Im } J$ by Theorem \ref{mahowaldimj}. By Theorem \ref{hoveyimj}, they must not survive the Adams spectral sequence for $MO\langle 16 \rangle$. 
The differentials listed are the only possible differentials with their respective targets. 
\end{proof}
\end{lemma}

\begin{lemma}\label{lem:d_6 differential in stem 36 filtration 11 mo16} 
$d_6(\vv(36,11,16)) = P^4 h_2$.
\end{lemma}

\begin{proof}
The element $\vv(36,11,16)$ is contained in the Massey product
$\langle h_0,h_0^3h_3,\vv(28,7,16)\rangle$ with indeterminacy
$h_0^{11}\vv(36,0,16)$. By Moss's Higher Leibniz rule \ref{mossleibniz},
\begin{align*}
d_6(\langle h_0, h_0^3h_3, \vv(28,7,16)\rangle) \in\;&
\langle d_6(h_0),h_0^3h_3,\vv(28,7,16)\rangle
+\langle h_0,d_6(h_0^3h_3),\vv(28,7,16)\rangle
+\langle h_0,h_0^3h_3,d_6(\vv(28,7,16))\rangle\\
=\;&\langle 0,h_0^3h_3,\vv(28,7,16)\rangle
+\langle h_0,0,\vv(28,7,16)\rangle
+\langle h_0,h_0^3h_3,P^3h_2\rangle.
\end{align*}
The first two terms vanish with zero indeterminacy and the third contains $P^4h_2$
with zero indeterminacy. Observe that, for degree reasons, the differential
$d_6(h_0^{10}\vv(36,0,16))$ equals zero. By $h_0$-linearity of the $d_6$-differential
and the Leibniz rule, this implies $d_6(h_0^{11}\vv(36,0,16))$ also equals zero.
Hence,
\[
d_6(\langle h_0, h_0^3h_3, \vv(28,7,16)\rangle)=d_6(\vv(36,11,16))=P^4h_2,
\]
as claimed.
\end{proof}

\begin{lemma}
The element $\vv(33,1,16)$ is a permanent cycle. 
\end{lemma}

\subsubsection{The abutment}
Recall that there are no hidden $2$-extensions between elements in the image of the unit map in this range of degrees.

\begin{proposition}
$\pi_{16}MO\langle 16 \rangle \cong \Z \oplus \Z/2$.
\end{proposition}
\begin{proof}
The proof is identical to the proof of Proposition \ref{prop:extension problem stem 16 mo9}. 
\end{proof}

\begin{proposition}
$\pi_{20}MO\langle 16 \rangle \cong \Z \oplus \Z/8$.
\end{proposition}
\begin{proof}
The proof is identical to the proof of Proposition \ref{prop:stem 20 extension}. 
\end{proof}

\begin{proposition}
$\pi_{24}MO\langle 16 \rangle \cong \Z \oplus \Z/2$.
\end{proposition}
\begin{proof}
The proof is identical to the proof of Proposition \ref{prop:stem 24 extension}. 
\end{proof}

\begin{proposition}\label{prop:stem 28 extension}
$\pi_{28}MO\langle 16 \rangle \cong \Z \oplus \Z/2$.
\end{proposition}
\begin{proof}
Note that $d_0^2$ detects $\kappa^{2} \in \pi_{28} \mathbb{S}$. The group $\pi_{28}(MO \langle 16 \rangle)$ lies in the short exact sequence
\begin{align*}
0 \rightarrow \Z \rightarrow \pi_{28}(MO \langle 16 \rangle) \rightarrow \mathbb{Z}/2\{\kappa^2\} \rightarrow 0 
\end{align*}
Suppose this short exact sequence did not split. This would imply $2\kappa^2 \neq 0$, which is a contradiction since $2\kappa^2 = 0 \in \pi_{28} \mathbb{S}$. 
\end{proof}

\begin{theorem}\label{thm: mo16 groups}
Table \ref{tab:mo16 groups} describes the 2-primary component of the cobordism groups
$\Omega^{\langle 16 \rangle}_n$ for the values of $n$ listed.
\end{theorem}
\begin{center}
\captionof{table}[The 2-primary component of the $O\langle 16\rangle$-cobordism groups]{The 2-primary component of $\Omega^{\langle 16\rangle}_n$ for the values of $n$ listed.}
\label{tab:mo16 groups}
\tablehead{\hline%
$n$ & $\Omega^{\langle 16\rangle}_n$ & $n$ & $\Omega^{\langle 16\rangle}_n$\\%
\hline}
\tabletail{\hline%
\multicolumn{4}{r}{%
\small\slshape to be continued on the next page}\\}
\tablelasttail{\hline}
\begin{supertabular}{|c|p{5.2cm}|c|p{5.2cm}|}
$0$  & $\Z$                        & $18$ & $\Z/8 \oplus \Z/2$\\
$1$  & $\Z/2$                      & $19$ & $\Z/2$\\
$2$  & $\Z/2$                      & $20$ & $\Z/8 \oplus \Z$\\
$3$  & $\Z/8$                      & $21$ & $\Z/2 \oplus \Z/2$\\
$4$  & $(0)$                       & $22$ & $\Z/2 \oplus \Z/2$\\
$5$  & $(0)$                       & $23$ & $(\Z/2)^2 \oplus \Z/4$\\
$6$  & $\Z/2$                      & $24$ & $\Z/2 \oplus \Z$\\
$7$  & $\Z/16$                     & $25$ & $\Z/2$\\
$8$  & $\Z/2 \oplus \Z/2$          & $26$ & $\Z/2 \oplus \Z/2$\\
$9$  & $(\Z/2)^3$          & $27$ & $(0)$\\
$10$ & $\Z/2$                      & $28$ & $\Z/2 \oplus \Z$\\
$11$ & $(0)$                       & $29$ & $(0)$\\
$12$ & $\Z$                        & $30$ & $\Z/2 \oplus \Z/2$\\
$13$ & $(0)$                       & $31$ & $\Z/2 \oplus \Z/2$\\
$14$ & $\Z/2 \oplus \Z/2$          & $32$ & $\Z \oplus \Z \oplus (\Z/2)^3$\\
$15$ & $\Z$                        & $33$ & $|\Omega^{\langle 16\rangle}_{33}|=32$\\
$16$ & $\Z \oplus \Z/2$            & $36$ & $\Z \oplus \Z \oplus F$\\
$17$ & $(\Z/2)^3$                  &      & \\
\end{supertabular}
\end{center}

\begin{center}
\captionof{table}[The 2-primary component of the $O\langle 16\rangle$-cobordism groups]{The 2-primary component of $\Omega^{\langle 16\rangle}_n$ for the values of $n$ listed.}
\label{tab:mo16 groups}
\tablehead{\hline%
$n$ & $\Omega^{\langle 16\rangle}_n$ & $n$ & $\Omega^{\langle 16\rangle}_n$\\%
\hline}
\tabletail{\hline%
\multicolumn{4}{r}{%
\small\slshape to be continued on the next page}\\}
\tablelasttail{\hline}
\begin{supertabular}{|c|p{5.2cm}|c|p{5.2cm}|}
$9$  & $\Z/2 \oplus \Z/2$ & $20$ & $\Z/8 \oplus \Z$\\
$10$ & $\Z/2$             & $21$ & $\Z/2 \oplus \Z/2$\\
$11$ & $(0)$              & $22$ & $\Z/2 \oplus \Z/2$\\
$12$ & $\Z$           & $23$ & $(\Z/2)^2 \oplus \Z/4$\\
$13$ & $(0)$              & $24$ & $\Z/2 \oplus \Z$\\
$14$ & $\Z/2 \oplus \Z/2$ & $25$ & $\Z/2$\\
$15$ & $\Z$           & $26$ & $\Z/2 \oplus \Z/2$\\
$16$ & $\Z \oplus \Z/2$ & $27$ & $(0)$\\
$17$ & $(\Z/2)^3$         & $28$ & $\Z/2 \oplus \Z$\\
$18$ & $\Z/8 \oplus \Z/2$ & $29$ & $(0)$\\
$19$ & $\Z/2$             & $30$ & $\Z/2 \oplus \Z/2$\\
      &                    & $31$ & $\Z/2 \oplus \Z/2$\\
      &                    & $32$ & $\Z \oplus \Z \oplus (\Z/2)^3$\\
      &                    & $33$ & $|\Omega^{\langle 16\rangle}_{33}|=32$\\
      &                    & $36$ & $\Z \oplus \Z \oplus F$\\
\end{supertabular}
\end{center}

\subsection{$MO\langle 17 \rangle$}
\subsubsection{The cohomology of $MO \langle 17 \rangle$}

\begin{proposition}[\cite{Stong1963Determination}]
    $H^{*}(BO\langle17 \rangle;\mathbb{Z}/2)$ is isomorphic to:
    \[
        H^{*}(K(\mathbb{Z}/2,17); \Z/2)/\langle Sq^2 i_{17} )\rangle \otimes \mathbb{Z}/2[\theta_{i} | L(i) > 8]
    \]
\end{proposition}
where $L(i)$ is one plus the number of ones in the binary expansion of $i-1$. We calculate the $\mathcal{A}$-module structure of $H^{*}(MO\langle17 \rangle;\mathbb{Z}/2)$ through degree $70$ using a computer program in Sage.

\subsubsection{Adams charts}
This section contains charts for the Adams spectral sequence for $MO\langle 17 \rangle$. The $E_r$-page is displayed along with the $d_r$-differentials. The $y$-axis is the Adams filtration $s$ and the $x$ axis is the stem $t-s$. The elements $h_0$, $h_1$, and $h_2$ have had their names suppressed. The dotted lines on the $E_{\infty}$-page represent possibly nonzero differentials that have yet to be determined. 
\clearpage
\FloatBarrier
\begin{figure}[!p]
\centering

\begin{subfigure}{\linewidth}
  \centering
  \includegraphics[width=\linewidth,
    trim=0.5cm 0cm 0cm 0cm,clip]{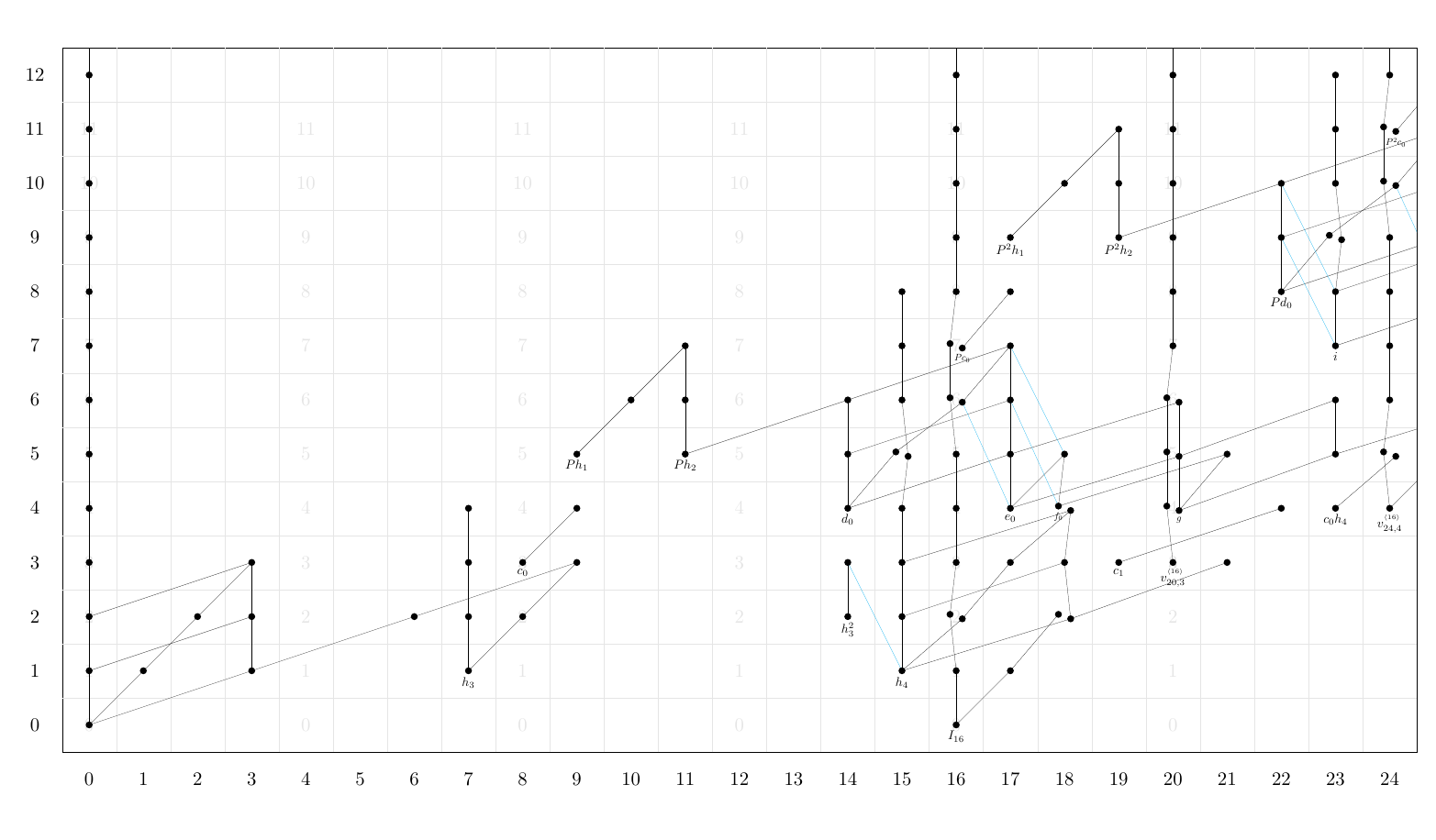}
  \caption{$E_2$-page $MO \langle 17 \rangle $ (stems 0-24)}
  \label{fig:mo17_e2_0_24}
\end{subfigure}

\vspace{-0.2\baselineskip}

\begin{subfigure}{\linewidth}
  \centering
  \includegraphics[width=0.8\linewidth,
    trim=0.5cm 0cm 0cm 0cm,clip]{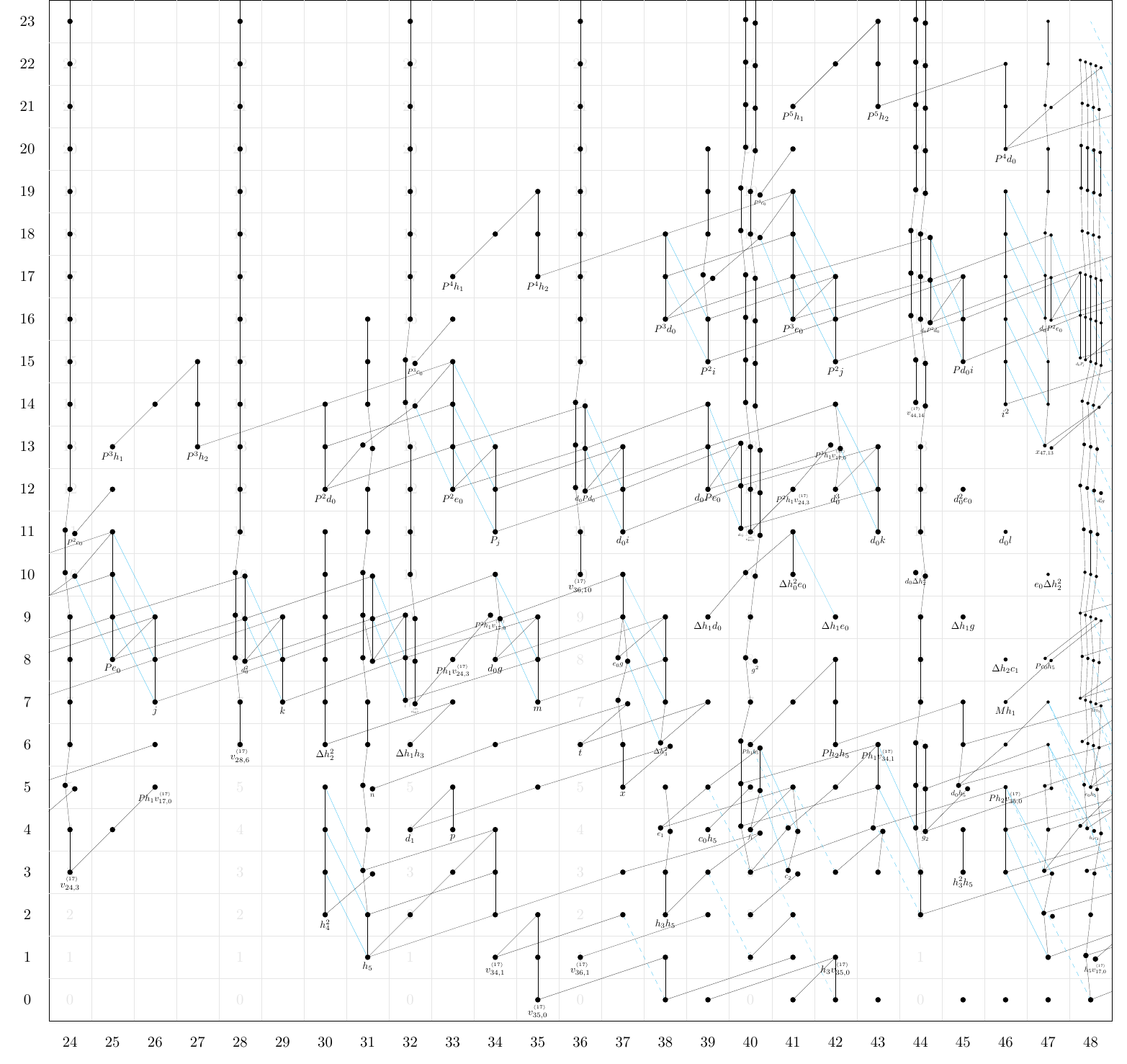}
  \caption{$E_2$-page $MO \langle 17 \rangle $ (stems 24-48)}
  \label{fig:mo17_e2_24_48}
\end{subfigure}

\caption{$E_2$-page $MO\langle 17 \rangle$ \label{fig:mo17_e2}}
\end{figure}

\begin{figure}[!p]
\centering
\begin{subfigure}{\linewidth}
  \centering
  \includegraphics[width=\linewidth,
    trim=0.5cm 0cm 0cm 0cm,clip]{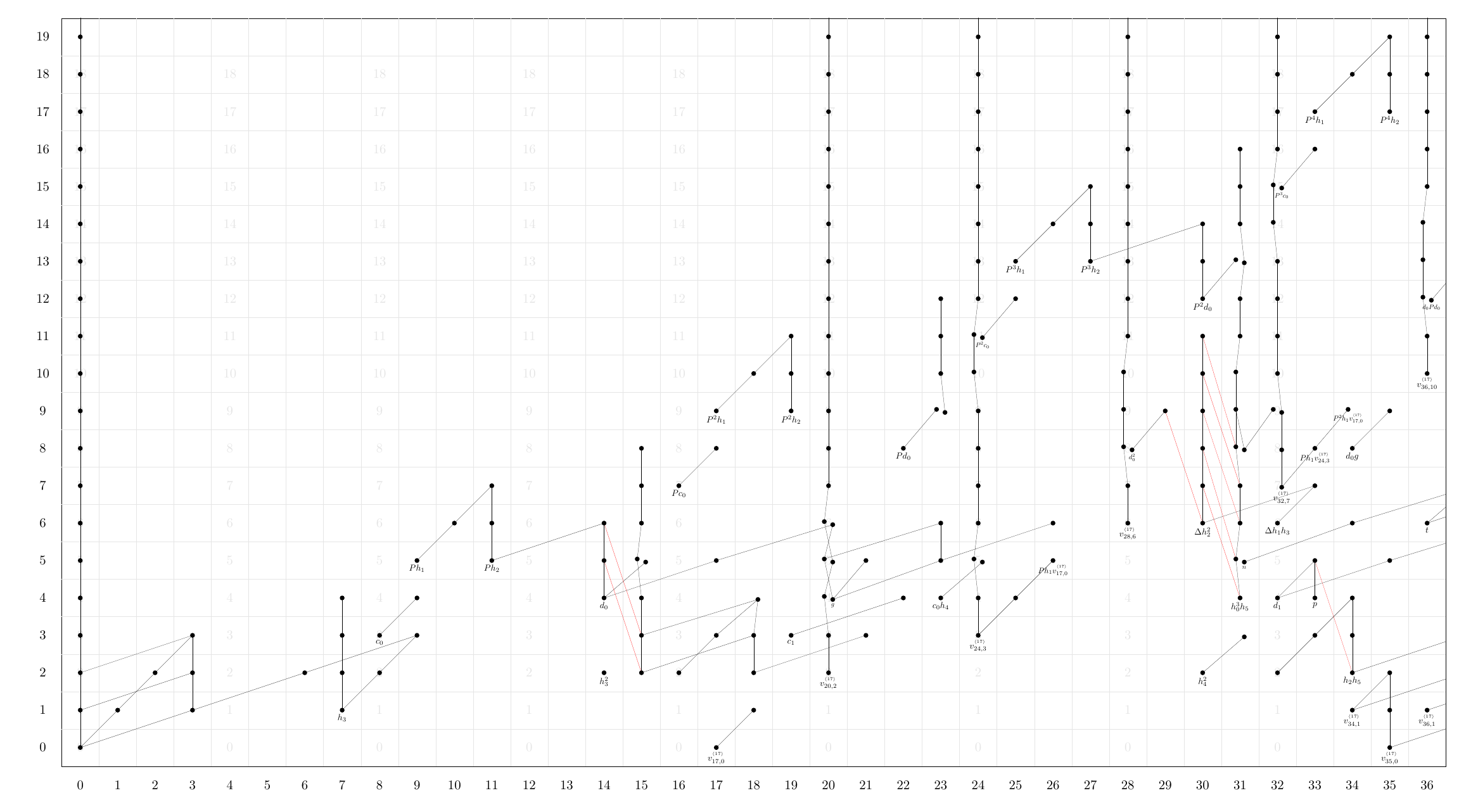}
  \caption{$E_3$-page $MO \langle 17 \rangle $ (stems 0-36)}
  \label{fig:mo17_e3}
\end{subfigure}

\vspace{-0.4\baselineskip}

\begin{subfigure}{\linewidth}
  \centering
  \includegraphics[width=\linewidth,
    trim=0.5cm 0cm 0cm 0cm,clip]{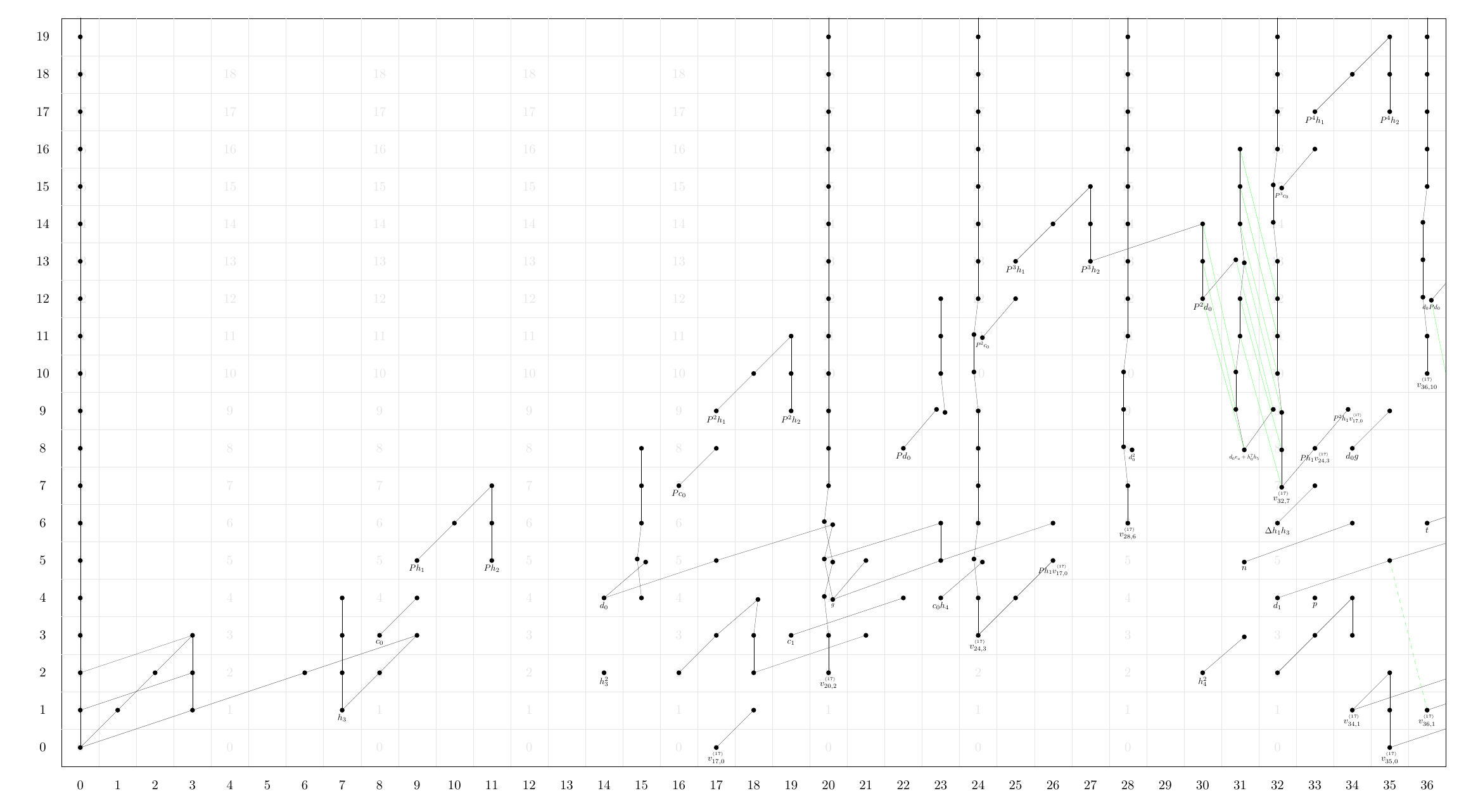}
  \caption{$E_4$-page $MO \langle 17 \rangle $ (stems 0-36)}
  \label{fig:mo17_e4}
\end{subfigure}

\caption{$E_3$/$E_4$-page $MO\langle 17 \rangle$ \label{fig:mo17_e3_e4}}
\end{figure}

\begin{figure}[!p]
\centering
\begin{subfigure}{\linewidth}
  \centering
  \includegraphics[width=\linewidth,
    trim=0.5cm 0cm 0cm 0cm,clip]{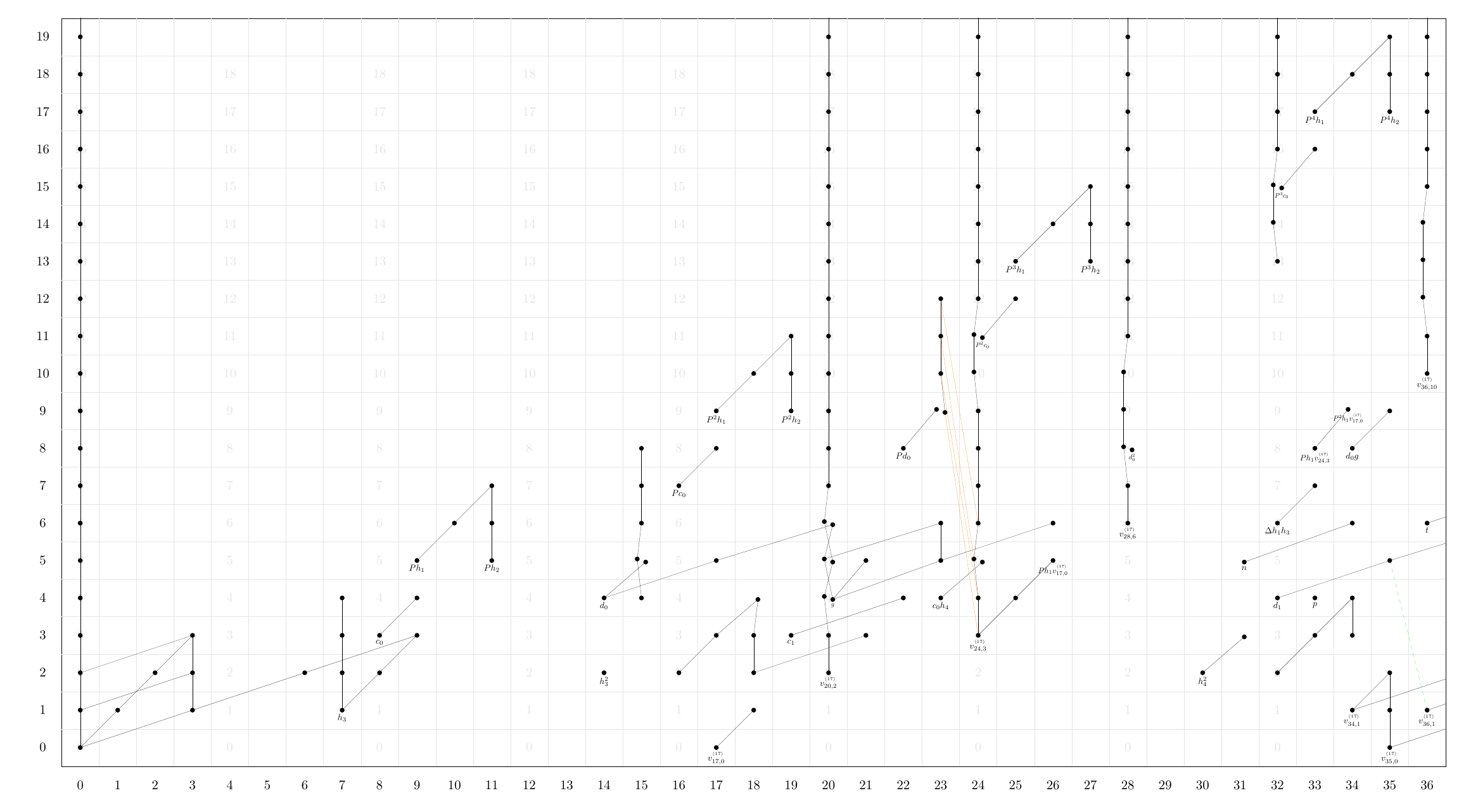}
  \caption{$E_5$=$E_6$-page $MO \langle 17 \rangle $ (stems 0-36)}
  \label{fig:mo17_e6}
\end{subfigure}

\vspace{-0.4\baselineskip}

\begin{subfigure}{\linewidth}
  \centering
  \includegraphics[width=\linewidth,
    trim=0.5cm 0cm 0cm 0cm,clip]{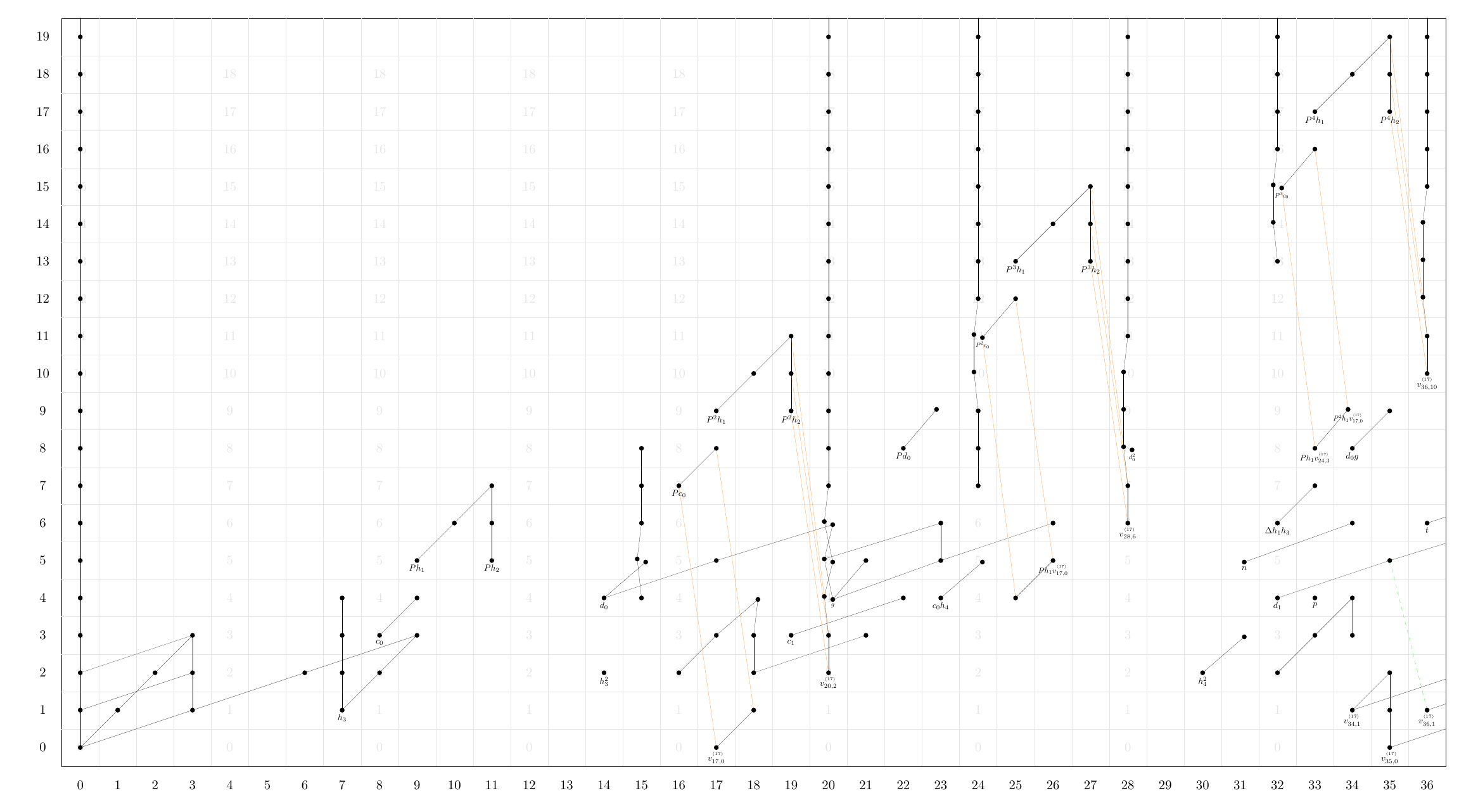}
  \caption{$E_{7}$-page $MO \langle 17 \rangle $ (stems 0-36)}
  \label{fig:mo17_e6}
\end{subfigure}

\caption{$E_5=E_6$/$E_7$-page $MO\langle 17 \rangle$ \label{fig:mo17_e6_e7}}
\end{figure}

\begin{figure}
  \centering
  \includegraphics[width=\linewidth,
    trim=0.5cm 0cm 0cm 0cm,clip]{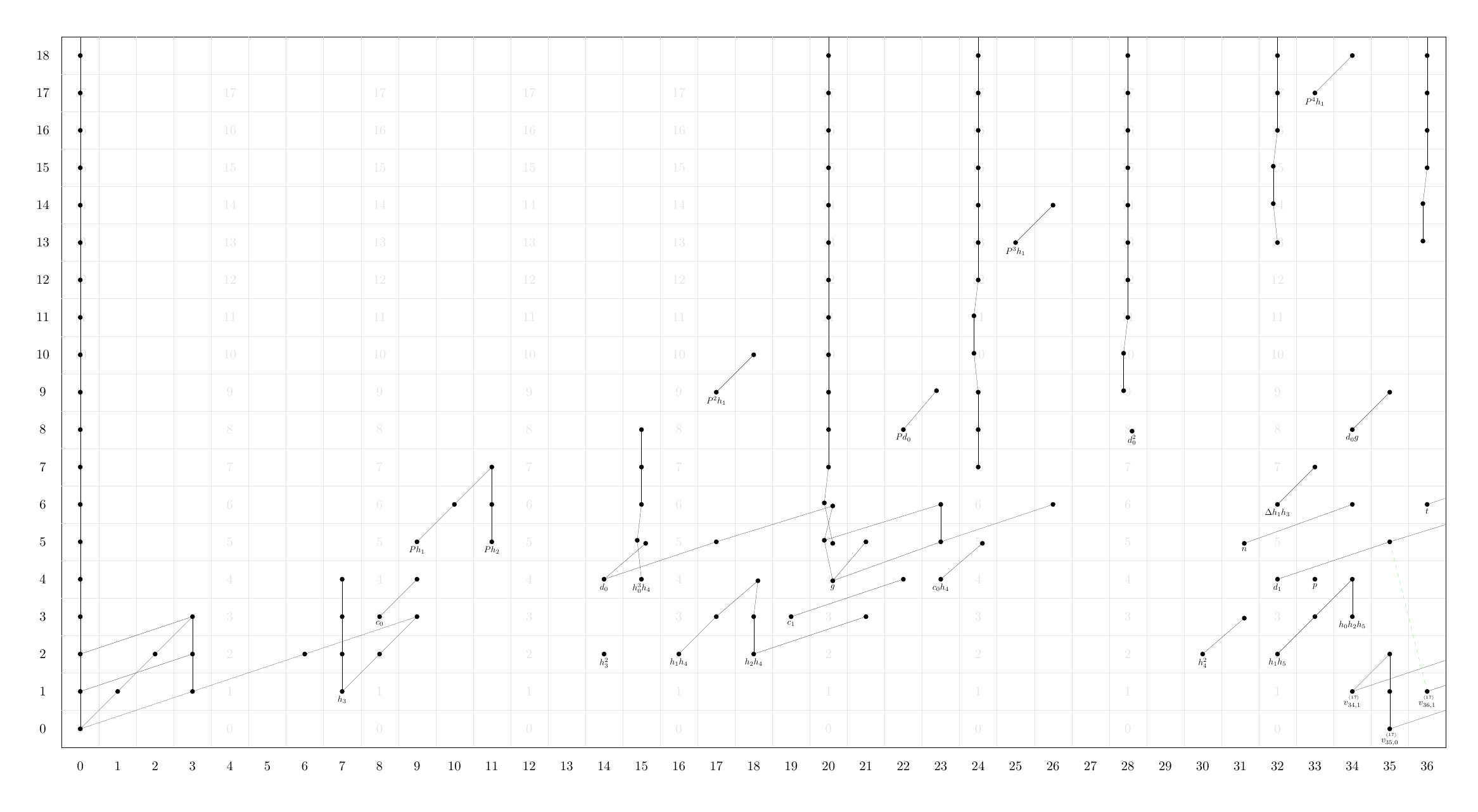}
  \caption{$E_{\infty}$-page $MO\langle 17 \rangle$ \label{fig:mo17_einf}}
\end{figure}
\FloatBarrier

\subsubsection{$d_2$-differentials}
\begin{proposition}
Table \ref{tab:Adams d_2 mo17} describes the non-zero $d_2$-differentials in the Adams spectral sequence for $MO\langle 17 \rangle$  on all indecomposables on $E_2$ through stem 36, and on select elements in higher stems.
\end{proposition}
The following table lists the only indecomposables through stem 36 that, for degree reasons, can support a $d_2$-differential. It also includes select differentials in higher stems.
 \begin{longtable}{llllc}
    \caption[Possible $d_2$-differentials $MO\langle 17 \rangle$]{Possible non-zero $d_2$-differentials on indecomposable elements
    \label{tab:Adams d_2 mo17}
    } \\
    \toprule
    $x$ & $(t-s,s)$ & $d_2(x)$ & Occurs & Proof\\
    \midrule \endfirsthead
    \caption[]{Possible non-zero $d_2$-differentials on indecomposable elements} \\
    \toprule
    $x$ & $(t-s,s)$ & $d_2(x)$ & Occurs & Proof\\
    \midrule \endhead
    \bottomrule \endfoot
        $h_1$ & $(1, 1)$ & $h_0^3$ & No & \ref{prop: sphere differentials}\\
        $h_4$ & $(15, 1)$ & $h_0h_3^2$ & Yes & \ref{prop: sphere differentials}\\
        $e_0$ & $(15, 4)$ & $h_1^2d_0$ & Yes & \ref{prop: sphere differentials}\\
        $f_0$ & $(18, 4)$ & $h_0^2e_0$ & Yes & \ref{prop: sphere differentials}\\
        $i$ & $(23, 7)$ & $h_0Pd_0$ & Yes & \ref{prop: sphere differentials}\\
        $j$ & $(26, 7)$ & $h_0Pe_0$ & Yes & \ref{prop: sphere differentials}\\
        $k$ & $(29, 7)$ & $h_0d_0^2$ & Yes & \ref{prop: sphere differentials}\\
        $h_5$ & $(31, 1)$ & $h_0h_4^2$ & Yes & \ref{prop: sphere differentials}\\
        $l$ & $(32, 7)$ & $h_0d_0e_0$ & Yes & \ref{prop: sphere differentials} \\
        $P_j$ & $(34, 11)$ & $h_0^2P^2e_0$ & Yes & \ref{prop: sphere differentials} \\
        $m$ & $(35, 7)$ & $h_0d_0g$ & Yes & \ref{prop: sphere differentials} \\
        $\vv(38,0,17)$ & $(38, 0)$ & $h_2\vv(34,1,17)$ & ? & \\
    \end{longtable}
\subsubsection{$d_3$-differentials}
\begin{proposition}
Table \ref{tab:Adams d_3 mo17} describes the non-zero $d_3$-differentials in the Adams spectral sequence for $MO\langle 17 \rangle$  on all indecomposables on $E_3$ through stem 36, and on select elements in higher stems.
\end{proposition}
The following table lists the only indecomposables through stem 36 that, for degree reasons, can support a $d_3$-differential. It also includes select differentials in higher stems.
 \begin{longtable}{llllc}
    \caption[Possible $d_3$-differentials $MO\langle 17 \rangle$]{Possible non-zero $d_3$-differentials on indecomposable elements
    \label{tab:Adams d_3 mo17}
    } \\
    \toprule
    $x$ & $(t-s,s)$ & $d_3(x)$ & Occurs & Proof\\
    \midrule \endfirsthead
    \caption[]{Possible non-zero $d_3$-differentials on indecomposable elements} \\
    \toprule
    $x$ & $(t-s,s)$ & $d_3(x)$ & Occurs & Proof\\
    \midrule \endhead
    \bottomrule \endfoot
        $h_0h_4$ & $(15,2)$ & $h_0d_0$ & Yes & \ref{prop: sphere differentials}\\
        $h_1h_4$ & $(16,2)$ & $h_1d_0$ & No & \ref{prop: sphere differentials}\\
        $h_2h_4$ & $(31,4)$ & $h_0e_0$ & No & \ref{prop: sphere differentials}\\
        $\vv(24,3,17)$ & $(24,3)$ & $h_0h_2g$ & No & \ref{lem:d_3 differential in stem 24 filtration 3 mo17}\\
        $\Delta h_2^2$ & $(30,6)$ & $h_0^2k$ & Yes & \ref{prop: sphere differentials}\\
        $h_0^3h_5$ & $(31,4)$ & $h_0h_2h_5$ & No & \ref{prop: sphere differentials}\\
        $\vv(35,0,17)$ & $(35,0)$ & $h_0h_2h_5$ & No & \ref{lem:perm cycle in stem 35 filtration 0 mo17} \\
        $e_1$ & $(38,4)$ & $h_1t$ & Yes & \ref{prop: sphere differentials}\\
    \end{longtable}
    \begin{proof}
        For proofs of these differentials, see Proposition \ref{prop: sphere differentials} and the following Lemma. 
    \end{proof}
\begin{lemma}\label{lem:d_3 differential in stem 24 filtration 3 mo17}
$d_3(\vv(24,3,17))=0$
\end{lemma}
\begin{proof}
The only possible target of a $d_3$-differential with source $\vv(24,3,17)$ is $h_0h_2g$. The element $h_0h_2g$ detects $2\nu \overline{\kappa} \in \pi_{23} \mathbb{S}$. Consider the composition of maps
\begin{align*}
\pi_{23} \mathbb{S} \rightarrow \pi_{23} MO\langle 17 \rangle \rightarrow \pi_{23} MO \langle 16 \rangle 
\end{align*}
By calculations in the previous section (see Figure \ref{fig:mo16_einf}), $h_0h_2g$ survives the Adams spectral sequence for $MO \langle 16 \rangle$, and so $2\nu \overline{\kappa}$ is in the image of the composition of maps above. Thus, $2\nu \overline{\kappa}$ maps nontrivially to $\pi_{23} MO \langle 17 \rangle$, and $h_0h_2g$ must survive the Adams spectral seqeunce for $MO \langle 17 \rangle$. 
\end{proof}

\begin{lemma}\label{lem:perm cycle in stem 35 filtration 0 mo17}
The element $\vv(35,0,17)$ is a permanent cycle. 
\end{lemma}
\begin{proof}
The possible targets of a differential with source $\vv(35,0,17)$ are $h_0h_2h_5$, $h_0^2h_2h_5$, $h_2n$, $d_0g$, and $h_1P^{4}h_1$. By
Theorem \ref{mahowaldimj}, each these elements is in $\text{Coker } J$. Thus, by Theorem \ref{thm: 2n mo_n kernel}, they must all survive 
the Adams spectral sequence for $MO\langle 17\rangle$. Hence, $\vv(35,0,17)$ is a permanent cycle. 
\end{proof}

\subsubsection{$d_4$-differentials}
\begin{proposition}
Table \ref{tab:Adams d_4 mo17} describes the non-zero $d_4$-differentials in the Adams spectral sequence for $MO\langle 17 \rangle$  on all indecomposables on $E_4$ through stem 37 and on select elements in higher stems. 
\end{proposition}
The following table lists the only indecomposables through stem 37 that, for degree reasons, can support a $d_4$-differential. It also includes select differentials in higher stems. 
 \begin{longtable}{llllc}
    \caption[Possible $d_4$-differentials $MO\langle 17 \rangle$]{Possible non-zero $d_4$-differentials on indecomposable elements
    \label{tab:Adams d_4 mo17}
    }\\
    \toprule
    $x$ & $(t-s,s)$ & $d_4(x)$ & Occurs & Proof\\
    \midrule \endfirsthead
    \caption[]{Possible non-zero $d_4$-differentials on indecomposable elements} \\
    \toprule
    $x$ & $(t-s,s)$ & $d_4(x)$ & Occurs & Proof\\
    \midrule \endhead
    \bottomrule \endfoot
        $h_0e_0$ & $(17, 5)$ & $h_0v_{16,8}$ & No &  \ref{prop: sphere differentials}\\
        $d_0e_0+h_0^7h_5$ & $(31,8)$ & $P^2d_0$ & Yes &  \ref{prop: sphere differentials}\\
        $\vv(32,7,17)$ & $(32,7)$ & $h_0^{10}h_5$ & Yes & \ref{lem:d_4 differential in stem 32 filtration 7 mo17}\\
        $\vv(35,0,17)$ & $(35,0)$ & $h_0^2h_2h_5$ & No & \ref{lem:perm cycle in stem 35 filtration 0 mo17} \\
        $\vv(36,1,17)$ & $(36,1)$ & $h_2d_1$ & ? & \\
        $e_0g$ & $(37,8)$ & $d_0Pd_0$ & Yes &  \ref{prop: sphere differentials} \\
        $h_3h_5$ & $(38,2)$ & $h_0x$ & Yes &  \ref{prop: sphere differentials} \\
        $h_0^3 \Delta h_3^2$ & $(38,9)$ & $h_0^2d_0i$ & Yes &  \ref{prop: sphere differentials} \\
    \end{longtable}

\begin{lemma}\label{lem:d_4 differential in stem 32 filtration 7 mo17}
$d_4(\vv(32,7,17))=h_0^{10}h_5$.
\end{lemma}
\begin{proof}
By Theorem \ref{hoveyimj}, $\text{Im } J$ is contained in the kernel of the unit map $\mathbb{S} \rightarrow MO \langle 17 \rangle$. By Theorem \ref{mahowaldimj}, the $h_0$-tower ending at the ``Adams'' edge in stem 31 is in $\text{Im } J$, so there must be a differential with target $h_0^{10}h_5$. 
The only possible source of such a differential is $\vv(32,7,17)$. 
\end{proof}

\subsubsection{$d_5$-differentials}
There are no possible $d_5$-differentials on indecomposables on $E_5$ through stem 37.
\subsubsection{$d_6$-differentials}
\begin{proposition}
Table \ref{tab:Adams d_6 mo17} describes the non-zero $d_6$-differentials in the Adams spectral sequence for $MO\langle 17 \rangle$  on all indecomposables on $E_6$ through stem 37.
\end{proposition}
The following table lists the only indecomposables through stem 37 that, for degree reasons, can support a $d_6$-differential.
 \begin{longtable}{llllc}
    \caption[Possible $d_6$-differentials $MO\langle 17 \rangle$]{Possible non-zero $d_6$-differentials on indecomposable elements
    \label{tab:Adams d_6 mo17}
    }\\
    \toprule
    $x$ & $(t-s,s)$ & $d_6(x)$ & Occurs & Proof\\
    \midrule \endfirsthead
      \caption[]{Possible non-zero $d_6$-differentials on indecomposable elements} \\
    \toprule
    $x$ & $(t-s,s)$ & $d_6(x)$ & Occurs & Proof\\
    \midrule \endhead
    \bottomrule \endfoot
       $\vv(24,3,17)$ & $(24,3)$ & $\{h_0^2i,h_1Pd_0\}$ & Yes & \ref{lem:d_6 differential in stem 24 filtration 4 mo17} \\
       $\vv(34,1,17)$ & $(34,1)$ & $h_1\Delta h_1h_3$ & No & \ref{lem:d_6 differential in stem 34 filtration 1 mo17}\\
      $\vv(35,0,17)$ & $(35,0)$ & $h_2n$ & No & \ref{lem:perm cycle in stem 35 filtration 0 mo17} \\
    \end{longtable}

\begin{lemma}\label{lem:d_6 differential in stem 24 filtration 4 mo17}
$d_6(\vv(24,3,17))=h_0^2i$. 
\end{lemma}
\begin{proof}
By Theorem \ref{mahowaldimj}, the element $h_0^2i$ and its $h_0$-multiples are in $\text{Im } J$. By Theorem \ref{hoveyimj}, they are in the kernel of the unit map $\pi_{*} \mathbb S \rightarrow MO \langle 17 \rangle$ and so they must 
not survive the Adams spectral sequence for $MO\langle 17 \rangle$. The element $h_1c_0h_4$ does not support a nonzero differential by Proposition \ref{prop: sphere differentials}, so the only possible 
source of a differential with target $h_0^2i$ is $\vv(24,3,17)$. Hence, $d_6(\vv(24,3,17))=h_0^2i$. 
\end{proof}

\begin{lemma}\label{lem:d_6 differential in stem 34 filtration 1 mo17}
    $d_6(\vv(34,1,17))=0.$
\end{lemma}
\begin{proof}
The only possible target of a $d_6$-differential with source $\vv(34,1,17)$ is $h_1\Delta h_1h_3$. Consider the composition of maps 
\[
\pi_{33} \mathbb S \rightarrow \pi_{33} MO\langle 17 \rangle \rightarrow \pi_{33} MO\langle 16 \rangle 
\]
The element $h_1\Delta h_1h_3$ survives the Adams spectral sequence for $MO\langle 16 \rangle$ (see Figure \ref{fig:mo16_einf}), so it detects a class with nonzero image under the composition above. Thus, $h_1\Delta h_1h_3$ must survive 
the Adams spectral sequence for $MO\langle 17\rangle$. Hence, $d_6(\vv(34,1,17))=0$.
\end{proof}

\subsubsection{$d_7$-differentials}
\begin{proposition}
Table \ref{tab:Adams d_7 mo17} describes the non-zero $d_7$-differentials in the Adams spectral sequence for $MO\langle 17 \rangle$  on all indecomposables on $E_7$. 
\end{proposition}
The following table lists the only indecomposables that, for degree reasons, can support a $d_7$-differential. 
 \begin{longtable}{llllc}
    \label{tab:Adams d_7 mo17}
     \\
    \toprule
    $x$ & $(t-s,s)$ & $d_7(x)$ & Occurs & Proof\\
    \midrule \endfirsthead
    \caption[Possible $d_7$-differentials $MO \langle 17 \rangle$]{Possible non-zero $d_7$-differentials on indecomposable elements} \\
    \toprule
    $x$ & $(t-s,s)$ & $d_7(x)$ & Occurs & Proof\\
    \midrule \endhead
    \bottomrule \endfoot
        $\vv(17,0,17)$ & $(17,0)$ & $Pc_0$ & Yes & \ref{lem:imj d7 differentials mo17} \\
        $\vv(20,2,17)$ & $(20,2)$ & $P^2h_2$ & Yes & \ref{lem:imj d7 differentials mo17} \\
        $\vv(28,6,17)$ & $(28,6)$ & $P^3h_2$ & Yes & \ref{lem:imj d7 differentials mo17}\\
        $\vv(36,10,17)$ & $(36,10)$ & $P^4h_2$ & Yes & \ref{lem:imj d7 differentials mo17}\\
        $\vv(44,14,17)$ & $(44,14)$ & $P^5h_2$ & Yes & \ref{lem:d_7-differential in stem 44 filtration 14 of mo17} 
    \end{longtable}
\begin{lemma}\label{lem:imj d7 differentials mo17}
The following differentials hold:
\end{lemma}
\begin{multicols}{2}
\begin{enumerate}[label=(\roman*)]
\item $d_7(\vv(17,0,17))=Pc_0$ 
\item $d_7(\vv(20,2,17))=P^2h_2$
\item $d_7(\vv(28,6,17))=P^3h_2$
\item $d_7(\vv(36,10,17))=P^4h_2$ 
\end{enumerate}
\end{multicols}
\begin{proof}
By Theorem \ref{mahowaldimj}, the target of each of these differentials is in $\text{Im }J$, Since $\text{Im } J$ is contained in the kernel of the unit map $\mathbb{S} \rightarrow MO\langle 17 \rangle$, these elements must not survive.
These $d_7$-differentials are the only possible differentials with their respective targets. 
\end{proof}

\begin{lemma}\label{lem:d_7-differential in stem 44 filtration 14 of mo17}
$d_7(\vv(44,14,17))=P^5h_2$
\end{lemma}
\begin{proof}
The element $\vv(44,14,17)$ is contained in the Massey product $\langle h_0,h_0^3h_3,\vv(36,10,17) \rangle$ with indeterminacy $h_0^{12}\vv(44,2,17)$. By Moss's Higher Leibniz rule \ref{mossleibniz},
\begin{align*}
d_7(\langle h_0, h_0^3h_3, \vv(36,10,17)\rangle) \in\;&
\langle d_5(h_0),h_0^3h_3,\vv(36,10,17)\rangle+\langle h_0,d_7(h_0^3h_3),\vv(36,10,17)\rangle+\langle h_0,h_0^3h_3,d_7(\vv(36,10,17))\rangle\\
=\;&\langle 0,h_3,\vv(36,10,17)\rangle+\langle h_0,0,\vv(36,10,17)\rangle+\langle h_0,h_0^3h_3,P^4h_2\rangle.
\end{align*}
The first two terms vanish with zero indeterminacy and the third contains $P^5h_2$ with zero indeterminacy. Observe that, for degree reasons, the differential $d_7(h_0^{11}\vv(44,2,17))$ equals zero. By $h_0$-linearity of the $d_7$-differential and the Leibniz rule, this implies $d_7(h_0^{12}\vv(44,2,17))$ also equals zero. 
Hence, $d_7(\langle h_0, h_0^3h_3, \vv(36,10,17)\rangle)=d_7(\vv(44,14,17))=P^5h_2$.
\end{proof}

\subsubsection{The abutment}
Recall that there are no hidden $2$-extensions between elements in the image of the unit map in this range of degrees. The first possible hidden $2$-extension is in stem 16. 

\begin{proposition}
$\pi_{20}MO\langle 17 \rangle \cong \Z \oplus \Z/8$.
\end{proposition}
\begin{proof}
The proof is identical to the proof of Proposition \ref{prop:stem 20 extension}. 
\end{proof}

\begin{proposition}
$\pi_{24}MO\langle 17 \rangle \cong \Z \oplus \Z/2$.
\end{proposition}
\begin{proof}
The proof is identical to the proof of Proposition \ref{prop:stem 24 extension}. 
\end{proof}

\begin{proposition}
$\pi_{28}MO\langle 17 \rangle \cong \Z \oplus \Z/2$.
\end{proposition}
\begin{proof}
The proof is identical to the proof of Proposition \ref{prop:stem 28 extension}.
\end{proof}

\begin{proposition}
$\pi_{32} MO\langle 17 \rangle \cong \Z \oplus (\Z/2)^3$
\end{proposition}
\begin{proof}
Note that all elements in $\pi_{32} \mathbb S$ are 2-torsion. The only possible hidden 2-extensions in stem 32 have source in the image of the unit map. Thus, 
none occur. 
\end{proof}

\begin{theorem}\label{thm: mo17 groups}
Table \ref{tab:mo17 groups} describes the 2-primary component of the $O\langle 17\rangle$-cobordism groups
$\Omega^{\langle 17 \rangle}_n$ for the values of $n$ listed.
\end{theorem}
\begin{center}
\captionof{table}[The 2-primary component of the $O\langle 17\rangle$-cobordism groups]{The 2-primary component of $\Omega^{\langle 17\rangle}_n$ for the values of $n$ listed.}
\label{tab:mo17 groups}
\tablehead{\hline%
$n$ & $\Omega^{\langle 17\rangle}_n$ & $n$ & $\Omega^{\langle 17\rangle}_n$\\%
\hline}
\tabletail{\hline%
\multicolumn{4}{r}{%
\small\slshape to be continued on the next page}\\}
\tablelasttail{\hline}
\begin{supertabular}{|c|p{5.2cm}|c|p{5.2cm}|}
$0$  & $\Z$ & $18$ & $\Z/8 \oplus \Z/2$\\
$1$  & $\Z/2$   & $19$ & $\Z/2$\\
$2$  & $\Z/2$   & $20$ & $\Z$\\
$3$  & $\Z/8$   & $21$ & $\Z/2 \oplus \Z/2$\\
$4$  & $(0)$    & $22$ & $\Z/2 \oplus \Z/2$\\
$5$  & $(0)$    & $23$ & $\Z/8 \oplus \Z/2$\\
$6$  & $\Z/2$   & $24$ & $\Z \oplus \Z/2 \oplus \Z/2$\\
$7$  & $\Z/16$  & $25$ & $\Z/2$\\
$8$  & $\Z/2\oplus\Z/2$   & $26$ & $\Z/2 \oplus \Z/2$\\
$9$  & $(\Z/2)^3$   & $27$ & $(0)$\\
$10$ & $\Z/2$   & $28$ & $\Z$\\
$11$ & $\Z/8$   & $29$ & $(0)$\\
$12$ & $(0)$    & $30$ & $\Z/2$\\
$13$ & $(0)$    & $31$ & $\Z/2 \oplus \Z/2$\\
$14$ & $\Z/2 \oplus \Z/2$ & $32$ & $\Z \oplus (\Z/2)^3$\\
$15$ & $\Z/32$  & $33$ & $(\Z/2)^4$\\
$16$ & $\Z/2$   & $34$ & $\Z/4 \oplus (\Z/2)^4$\\
$17$ & $(\Z/2)^3$ & & \\
\end{supertabular}
\end{center}

\clearpage

\section{On the boundaries of $(k-1)$-connected $(2k+d)$-manifolds}
\label{chap:5}
For $m \geq 5$, let $\Theta_m$ denote the group of oriented, smooth, closed manifolds $\Sigma$ that are homotopy equivalent to the $m$-sphere, where the group operation is connected sum. There is the Kervaire-Milnor exact sequence
\begin{align*}
0 \rightarrow bP_{m+1} \rightarrow \Theta_{m} \rightarrow \text{Coker}(J)_{m}
\end{align*}
where $bP_{m+1}\subseteq \Theta_m$ consists of all homotopy spheres that are the boundaries of parallelizable $(m+1)$-manifolds. In \cite{Stolz1985}, Stolz laid the groundwork to identify large classes of homotopy spheres that bound parallelizable manifolds. The unit map $\mathbb{S} \rightarrow MO \langle k \rangle$ may be extended to a cofiber sequence
\begin{align*}
\mathbb{S} \rightarrow MO\langle k \rangle \rightarrow MO\langle k \rangle/ \mathbb{S} \xrightarrow{\partial} \mathbb{S}^1
\end{align*}
Stolz defined a spectrum $A[k]$ together with a map
 \begin{align*}
b: A[k] \rightarrow MO\langle k \rangle /\mathbb{S}
 \end{align*}
 such that the following is true:
 \begin{theorem}[Stolz Lemma 12.5 \cite{Stolz1985}]\label{thm: stolz_lemma}
Let $k>2$ and $d \ge 0$ be integers.  Suppose that, for every element $\alpha \in \pi_{2k+d}(A[k])$, the image of $\alpha$ under the composite
$$\pi_{2k+d}(A[k]) \stackrel{b_*}{\longrightarrow} \pi_{2k+d} \left( \mathrm{MO} \langle k \rangle / \mathbb{S} \right) \stackrel{\partial_*}{\longrightarrow} \pi_{2k+d-1}\mathbb{S}$$
is in $\text{Im } J$.  Then the boundary of any $(k-1)$-connected, almost closed $(2k+d)$-manifold also bounds a parallelizable manifold.
 \end{theorem}
 A compact, oriented, smooth manifold $M$ is said to be \textit{almost closed} if its boundary $\partial M$ is a homotopy sphere. 
 Using this approach, Stolz proved
 \begin{theorem}[Stolz Theorem B \cite{Stolz1985}]
Let $M$ be a $(k-1)$-connected almost closed $(2k+d)$ manifold. If $k \cong 2 \pmod{8}$, $k>10$, and $0\leq d \leq 3$, then $\partial M$ bounds a parallelizable manifold. 
 \end{theorem}
 Much later, Burklund-Hahn-Senger proved the following theorem.
 \begin{theorem}[Burklund-Hahn-Senger Theorem 1.1 \cite{BHS2023Boundaries}]
        Let $k>232$ and $0 \le d \le 3$ be integers.  Then the boundary of every $(k-1)$-connected, almost closed $(2k+d)$-manifold also bounds a parallelizable manifold.  
 \end{theorem}
With additional hypotheses, they give a similar result for $k > 124$. While the authors mention that the bounds can likely be improved, Frank \cite{frank1968} and Stolz \cite{Stolz1985} gave counterexamples to the theorem being true in general. In particular, they show the existence of:
\begin{itemize}
        \item A 3-connected, almost closed 9-manifold with boundary non-trivial in $\text{Coker } J$.
        \item A 7-connected, almost closed 17-manifold with boundary non-trivial in $\text{Coker } J$.
\end{itemize}
There are results that do include some smaller values of $k$. As a consequence of Senger-Zhang Theorem \ref{thm: 2n mo_n kernel}, we have
\begin{theorem}\label{thm: Senger-Zhang}
Suppose $n \geq 10$. The boundary of every $(n-1)$-connected almost closed $(2n+1)$-manifold also bounds a parallelizable manifold.
\end{theorem}
\begin{theorem}[Theorem 1.4 Burklund-Senger \cite{BurklundSenger2024Geography}]
Suppose $n >2$ and $n\neq 9,12$. then the boundary of every $(n-1)$-connected almost closed $2n$-manifold also bounds a parallelizable manifold.
\end{theorem}
Following Remark 1.5 of \cite{BurklundSenger2024Geography}, the cases $3 \le n \le 8$ are due to Wall \cite{Wall1962}; the cases $n \equiv 3,5,6,7 \pmod 8$ follow from work of Wall and Kervaire--Milnor \cite{Wall1962,KervaireMilnor1963}; the cases $n \equiv 2 \pmod 8$ are due to Schultz \cite{Schultz1972}; the cases $n \equiv 1 \pmod 8$ with $n \ge 129$ are due to Stolz \cite{Stolz1985}; and the cases $n \equiv 0 \pmod 4$ with $n \ge 128$ are due to Burklund-Hahn-Senger \cite{BurklundHahnSenger2019}. The remaining cases are established in \cite{BurklundSenger2024Geography}.

We are interested in exploring the cases not covered by the above theorems using our calculations from Chapter \ref{chap:fivebrane} and Chapter \ref{chap:o_n_cobordism}. Before doing so, we first state a relevant result about $\text{Coker } J$ at odd primes.  

\begin{theorem}[Ravenel Theorem 1.1.14 \cite{ravenelgreenbook}]\label{thm: imj odd}
For $p \geq 3$ the $p$-component of $\text{Coker} J$ has the following generators in dimensions $\leq 3pq-6$ (where $q=2p-2$), each with order $p$:
\[
\beta_{1}\in \pi_{pq-2}\mathbb S, \qquad \alpha_{1}\beta_{1}\in \pi_{(p+1)q-3} \mathbb S
\]
\end{theorem}
Next, observe that we may restrict to studying the unit map $\pi_{*} \mathbb S \rightarrow MO \langle k \rangle$. Consider the cofiber sequence
\[
\mathbb S \longrightarrow MO\langle k\rangle \longrightarrow MO\langle k\rangle/\mathbb S.
\]
and its associated long exact sequence in homotopy 
\[
\dots \rightarrow \pi_{n}(MO\langle k\rangle/\mathbb S)
\xrightarrow{\partial_*}
\pi_{n-1}(\mathbb S)
\longrightarrow
\pi_{n-1}(MO\langle k\rangle) \rightarrow \dots
\]
By exactness, we have
\[
\operatorname{Im}(\partial_*)
=
\ker\!\left(
\pi_{n-1}(\mathbb S)\longrightarrow \pi_{n-1}(MO\langle k\rangle)
\right).
\]
It follows that
\[
\operatorname{Im}(\partial_*\circ b_*)
\subseteq
\operatorname{Im}(\partial_*)
=
\ker\!\left(
\pi_{n-1}(\mathbb S)\longrightarrow \pi_{n-1}(MO\langle k\rangle)
\right).
\]
where $b_{*}$ is the map in Theorem \ref{thm: stolz_lemma}. Thus, we prove the following. 
 \begin{theorem}\label{thm: mo9 unit map}
Let $0 \leq n \leq 31$ and $n \neq 10,13,20,29,30$. The kernel of the unit map
\begin{align*}
\pi_{n}\mathbb{S} \rightarrow \pi_{n}MO\langle 9 \rangle 
\end{align*}
is generated by the image of $J$ except when $n=17,18,19,21,22,23,26,31$. The exceptional cases, with generators of order divisible by $3$ omitted, are generated by the image of the $J$-homomorphism and
\begin{itemize}
\item When $n=17$, $\eta \eta_{4} \in \pi_{17}\mathbb{S}$. 
\item When $n=18$, $\{h_2h_4\}=\nu^{*} \in \pi_{18}\mathbb{S}$.
\item When $n=19$, $\overline{\sigma} \in \pi_{19}\mathbb{S}$.
\item When $n=21$, $\nu \{h_2h_4\}=\nu \nu^{*} \in \pi_{21}\mathbb{S}$.
\item When $n=22$, $\nu \overline{\sigma} \in \pi_{22}\mathbb{S}$.
\item When $n=23$, either $2 \nu \overline{\kappa}$ or $\nu \overline{\kappa} \in \pi_{23}\mathbb{S}$.
\item When $n=26$, $\nu^{2}\overline{\kappa} \in \pi_{26}\mathbb{S}$
\item When $n=31$, $[n] \in \pi_{31}\mathbb{S}$.
\end{itemize}
\end{theorem}
\begin{proof}
According to Table A3.4 of Ravenel \cite{ravenelgreenbook}, the $3$-primary component of $\text{Coker } J$ is trivial in the dimensions $n$ considered here. Since $p(2p-2)-2$ is $\geq 38$ when $p\geq 5$, the $p$-primary component 
of $\text{Coker } J$ is also trivial in the dimensions $n$ considered here. Thus, we need only consider the $2$-primary component of the unit map. The targets of the Adams differentials in Lemmas \ref{lem:d_2-unit-kernel-of-mo9-stem-17}, \ref{lem:d_2 differential in 19 stem mo9}, \ref{d_2 differential in 20 stem mo9}, \ref{d_2 differential in stem 24 filtration 4 mo9}, \ref{lem:d_2 differential in mo9 stem 27 filtration 4}, and \ref{lem:d_2 differential in stem 32 filtration 3 mo9}, are 
$h_1^2h_4$, $h_2h_4$, $c_1$, $h_0h_2g$, $h_2^2g$, and $n$. These elements detect the classes $\eta \eta_{4}$, $\nu^{*}$, $\overline{\sigma}$, $2 \nu \overline{\kappa}$, $\nu^{2}\overline{\kappa}$, and $[n]$ in $\pi_{*} \mathbb S$. 
 Lemma \ref{d_2 differential in 20 stem mo9} implies $d_2(h_2\vv(20,1,9))=h_2c_1$, which detects $\nu \overline{\sigma} \in \pi_{22}\mathbb{S}$. The ambiguity 
 for $n=23$ is because $d_3(\vv(26,2,9))$ could possibly equal $h_2g$, which detects $\nu \overline{\kappa} \in \pi_{23}\mathbb{S}$. 
\end{proof}
\begin{remark}
The case $n=17$ was already known by Burklund-Senger Theorem \ref{lem:mo9_17_kernel}. The case $n=18$ was already known by Senger-Zhang Theorem 6.1 \cite{SengerZhang2025Inertia}.
\end{remark}
We have the following immediate corollaries of Theorem \ref{thm: mo9 unit map}.
\begin{corollary}
Suppose $M$ is an $8$-connected, almost closed $(18+d)$-manifold $d=7,8,10,11$. Then the boundary $\partial M \in \Theta_{18+d-1}$ has trivial image $0=[\partial M]\in (\text{Coker } J)_{18+d-1}$. In particular, the boundary of every 
such manifold $M$ also bounds a parallelizable manifold.
\end{corollary}

\begin{corollary}
Let $d\in \{0,1,2,4,5,9,14\}$. A homotopy $(18+d-1)$-sphere $\Sigma$ is the boundary of an $8$-connected $(18+d)$-manifold if and only if
\begin{itemize}
    \item when $d=0$, 
    \[
    [\Sigma]\in \{0,\eta\eta_{4}\}\subset (\text{Coker } J)_{17};
    \]
    
    \item when $d=1$,
    \[
    [\Sigma]\in \{0, \nu^{*}\} \subset (\text{Coker } J)_{18};
    \]
    
    \item when $d=2$,
    \[
    [\Sigma]\in \{0, \overline{\sigma}\} \subset (\text{Coker } J)_{19};
    \]
    
    \item when $d=4$,
    \[
    [\Sigma]\in \{0, \nu\nu^{*}\} \subset (\text{Coker } J)_{21};
    \]
    
    \item when $d=5$,
    \[
    [\Sigma]\in \{ 0,\nu\overline{\sigma}\} \subset (\text{Coker } J)_{22};
    \]
    \item when $d=14$,
    \[
    [\Sigma]\in \{0,[n]\} \subset (\text{Coker } J)_{31}.
    \]
\end{itemize}
\end{corollary}

We proceed to state a few analogous results for more highly connected manifolds. 

\begin{lemma}\label{lem: kappa2 mo12}
The element $\kappa^2 \in \pi_{28} \mathbb S$ has nonzero image under the unit map $e^{\langle 12 \rangle}: \pi_{28} \mathbb S \rightarrow \pi_{28} MO \langle 12 \rangle$.
\end{lemma}
\begin{proof}
Inspecting Figure \ref{fig:mo9_einf}, we find that $d_0^{2}$ survives the Adams spectral sequence for $MO\langle 9 \rangle$. The element $d_0^{2}$ detects $\kappa^{2}$, and so $e^{\langle 9 \rangle}(\kappa^{2})$ is nonzero. 
The map $e^{\langle 9 \rangle}$ factors through $e^{\langle 12 \rangle}$. Hence, $e^{\langle 12 \rangle}(\kappa^{2})$ must also be nonzero. 
\end{proof}

\begin{proposition}
The kernel of the unit map 
\[
\pi_{28} \mathbb S \rightarrow \pi_{28} MO\langle 12 \rangle 
\]
is generated by the image of the $J$-homomorphism.
\end{proposition}
\begin{proof}
According to Table A3.4 of Ravenel \cite{ravenelgreenbook}, the 3-primary component of $\text{Coker } J$ is trivial in dimension 28. Observe that $p(2p-2)-2$ is $> 28$ when $p\geq 5$. Hence, by Theorem \ref{thm: imj odd}, the $p$-primary component of $\text{Coker } J$ is trivial 
in dimension 28 for all $p \geq 5$. The $2$-primary component of $\text{Coker } J$ is generated by $\kappa^{2}$. This element is not in the kernel 
by Lemma \ref{lem: kappa2 mo12}.
\end{proof}
\begin{corollary}
The boundary of any $11$-connected, almost closed $29$-manifold also bounds a parallelizable manifold. 
\end{corollary}
We now give an alternate proof of the case of $n=16$ of Theorem \ref{thm: Senger-Zhang}. 
\begin{proposition}
The kernel of the unit map 
\[
\pi_{33} \mathbb S \rightarrow \pi_{33} MO\langle 16 \rangle 
\]
is generated by the image of the $J$-homomorphism. 
\end{proposition}
\begin{proof}
According to Table A3.4 of Ravenel \cite{ravenelgreenbook}, the 3-primary component of $\text{Coker } J$ is trivial in dimension 33. Observe that $p(2p-2)-2$ is $> 33$ when $p\geq 5$. Hence, by Theorem \ref{thm: imj odd}, the $p$-primary component of $\text{Coker } J$ is trivial 
in dimension 33 for all  $p \geq 5$. By table $A3.3$ of Ravenel \cite{ravenelgreenbook}, the $2$-primary component of $\text{Coker J}$ in dimension 33 is detected by $h_1^2h_5$, $p$, $h_1\Delta h_1h_3$, $P^4h_1$. Inspecting Figure \ref{fig:mo16_einf}, each of those elements survives the Adams spectral sequence 
for $MO\langle 16\rangle$. 
\end{proof}
\begin{corollary}
The boundary of any $15$-connected, almost closed $34$-manifold also bounds a parallelizable manifold. 
\end{corollary}

\begin{lemma}\label{lem: kappa2 mo17}
The elements $\overline{\kappa}_{2}$, $2\overline{\kappa}_{2}$, and $4\overline{\kappa}_{2} \in (\pi_{44} \mathbb S)^{\wedge}_{2}$ have nonzero image under the 2-local unit map $(\pi_{44} \mathbb S)^{\wedge}_{2} \rightarrow (\pi_{44} MO \langle 17 \rangle)^{\wedge}_{2}$.
\end{lemma}
\begin{proof}
Inspecting Figure \ref{fig:mo17_e2_24_48}, we find that $g_2$, $h_0g_2$, and $h_0^2g_2$ all survive the Adams spectral sequence for $MO\langle 17 \rangle$. The element $g$ detects $\overline{\kappa}_{2} \in (\pi_{44} \mathbb S)^{\wedge}_{2}$.
\end{proof}

\begin{proposition}
The kernel of the unit map 
\[
\pi_{44} \mathbb S \rightarrow \pi_{44} MO\langle 17 \rangle 
\]
is generated by the image of the $J$-homomorphism. 
\end{proposition}
\begin{proof}
According to Table A3.4 and Table A3.5 of Ravenel \cite{ravenelgreenbook}, the $3$ and $5$-primary components of $\text{Coker } J$ are trivial in dimension 44. Observe that 
Observe that $p(2p-2)-2$ is $> 44$ when $p \geq 7$. Hence, by Theorem \ref{thm: imj odd}, the $p$-primary component of $\text{Coker } J$ is trivial 
in dimension 44 for all  $p \geq 7$. The $2$-primary component of $\text{Coker } J$ consists of $\overline{\kappa}_{2}$ and its 2-multiples. These elements are not in the kernel by Lemma \ref{lem: kappa2 mo17}. 
\end{proof}

\begin{corollary}
The boundary of any $16$-connected, almost closed $45$-manifold also bounds a parallelizable manifold. 
\end{corollary}

\clearpage




% ...


\phantomsection
\section*{REFERENCES}

\addcontentsline{toc}{section}{References}

\bibliography{bib/mybib}

\bibliographystyle{abbrv}

\clearpage

\phantomsection
\section*{ABSTRACT}

\addcontentsline{toc}{section}{Abstract}
\centerline{\bf ON THE COBORDISM GROUPS OF $O\langle n \rangle$-MANIFOLDS}

{\setlength\baselineskip{0.3in}
	\begin{center}
	by\\
	\medskip
	{\bf HASSAN ABDALLAH}\\
	\medskip
	{\bf May 2026}\\
	\end{center}
	\Vspc
	\begin{tabular}{ll}
		{\bf Advisor:} & Dr. Andrew Salch \\
		{\bf Major:} & Mathematics \\
		{\bf Degree:} & Doctor of Philosophy
	\end{tabular}
}

\bigskip \bigskip

Given a smooth manifold, a tangential $O\langle n \rangle$-structure is a lift of the classifying map of its tangent bundle to $BO\langle n \rangle$. An $O\langle n \rangle$-manifold is a manifold equipped with an equivalence class of tangential $O\langle n \rangle$-structures. The cases $n=2$ and $n=4$ recover the familiar notions of oriented and spin manifolds. Larger values of $n$ furnish higher analogs of these with additional geometric properties.
In this thesis, we calculate the cobordism groups of compact $O\langle n\rangle $-manifolds in a range of degrees, for $n$ between 9 and 17. We derive various consequences for the boundaries of almost-closed manifolds.

\clearpage

\phantomsection
\section*{AUTOBIOGRAPHICAL STATEMENT}

\addcontentsline{toc}{section}{Autobiographical Statement}

B.S. Mathematics, Wayne State University 2017 \\
M.A. Applied Mathematics, Wayne State University, 2020\\
Ph.D. Mathematics, Wayne State University, 2026

\end{document}